\documentclass[12pt]{article}
\usepackage[utf8]{inputenc}
\usepackage[T1]{fontenc} 
\usepackage{amsmath}
\usepackage{amsfonts}
\usepackage{amssymb}
\usepackage{amsthm}
\usepackage{color}
\usepackage{esint}
\usepackage{enumerate}
\usepackage{psfrag}
\usepackage{stmaryrd}
\usepackage[english]{babel}
\usepackage{authblk}
\usepackage{hyperref}
\usepackage{graphicx}
\usepackage{graphics}
\usepackage{epstopdf}
\usepackage{epsfig}
\usepackage{setspace}

\usepackage[sans]{dsfont}
\usepackage{algorithm}
\usepackage{algpseudocode}
\usepackage{subcaption}
\usepackage{tikz}
\usepackage{pgfplots} 
\pgfplotsset{
    legend cell align = {left},
    compat=1.9,
    tick pos = left,
    grid = both,
    grid style = {densely dotted, gray}
}
\usetikzlibrary{math}
\usepackage{booktabs} 
\usepackage{colortbl} 
\usepackage{pgfplotstable} 

\newcommand{\RR}{\mathbb{R}}
\newcommand{\PP}{\mathbb{P}}
\newcommand{\NN}{\mathbb{N}}
\newcommand{\dis}{\displaystyle}

\newcommand{\eps}{\varepsilon}

\newcommand*{\colorboxed}{}
\def\colorboxed#1#{%
  \colorboxedAux{#1}%
}
\newcommand*{\colorboxedAux}[3]{%
  \begingroup
    \colorlet{cb@saved}{.}%
    \color#1{#2}%
    \boxed{%
      \color{cb@saved}%
      #3%
    }%
  \endgroup
}

\DeclareMathOperator{\spn}{Span}

\newtheorem{theorem}{Theorem}

\newtheorem{rmq}[theorem]{Remark}
\newtheorem{proposition}[theorem]{Proposition}

\begin{document}

\title{A Multiscale Finite Element Method for reaction-diffusion eigenproblems arising from neutronics}

\date{\today}

\author{Claude Le Bris}
\author{Alb\'eric Lefort}
\author{Frédéric Legoll}
\affil{\small École Nationale des Ponts et Chaussées, Institut Polytechnique de Paris, CNRS, 6 et 8 avenue Blaise Pascal, 77420 Champs-sur-Marne, France}
\affil{\small MATHERIALS project-team, Inria Paris, 48 rue Barrault, 75013 Paris, France}
\affil{\small Emails: \{claude.le-bris,alberic.lefort,frederic.legoll\}@enpc.fr}

\maketitle

\begin{abstract}
  We consider reaction-diffusion eigenproblems with oscillatory diffusion and reaction coefficients. The reaction coefficient magnitude is large: the corrector equation identified by periodic homogenization involves both the diffusion and the reaction operators. We study the numerical approximation of this problem using the Multiscale Finite Element Method (MsFEM). This now classical method is a finite element type method that performs a Galerkin approximation of the oscillatory problem on a specific, problem dependent, basis set. The basis functions are precomputed in an offline stage. Inspired by homogenization theory and using some filtering ideas, we show how to define these basis functions in order to obtain an efficient method. The comprehensive set of numerical experiments that we present, in periodic and non-periodic cases, for the scalar-valued version of the problem (which is then self-adjoint) and for the vector-valued version of the problem (which is then in general non self-adjoint), demonstrates the performance of the approach. Some theoretical arguments complement the numerical observations.
\end{abstract}


\section{Introduction}

We consider the reaction-diffusion eigenvalue problem
\begin{equation} \label{pb reaction diffusion stationnaire}
  \Sigma^\eps \, u^\eps - \eps^2 \operatorname{div} \left(A^\eps \nabla u^\eps \right) = \lambda^\eps \, \sigma^\eps \, u^\eps \ \ \text{in $\Omega$}, \qquad u^\eps = 0 \ \ \text{on $\partial \Omega$},
\end{equation}
on a bounded domain $\Omega \subset \RR^d$. The coefficients $A^\eps$, $\Sigma^\eps$ and $\sigma^\eps$ are assumed to oscillate on a typical lengthscale $\eps$, which is much smaller than the size of the domain $\Omega$. We are interested in the first eigencouple $(\lambda^\eps,u^\eps) \in \RR \times H^1_0(\Omega)$ of~\eqref{pb reaction diffusion stationnaire}.

Problem~\eqref{pb reaction diffusion stationnaire} may be either scalar-valued (in which case $A^\eps(x)$ is a $d \times d$ matrix and $\Sigma^\eps(x)$ and $\sigma^\eps(x)$ are scalars, for any $x \in \Omega$) or vector-valued (in which case $A^\eps(x)$ is a fourth-order tensor and $\Sigma^\eps(x)$ and $\sigma^\eps(x)$ are matrices). In both cases, we assume that $A^\eps$, $\Sigma^\eps$ and $\sigma^\eps$ satisfy all the required assumptions so that~\eqref{pb reaction diffusion stationnaire} is well-posed and so that the smallest eigenvalue $\lambda^\eps$ is real and simple, while the associated eigenvector $u^\eps$ has all its components being real-valued. We refer to Sections~\ref{sec:one-group energy case} and~\ref{sec:MsFEM multigroup} for precise statements.

Equation~\eqref{pb reaction diffusion stationnaire} appears in several applied fields, including neutronics problems, in which case $u^\eps(x)$ represents the density of neutrons at point $x \in \Omega$. The oscillatory nature of the coefficients of~\eqref{pb reaction diffusion stationnaire} is then directly related to the heterogeneity of the physical properties within the domain. All neutrons may be considered to have the same energy, in which case the problem is modeled by the scalar-valued version of~\eqref{pb reaction diffusion stationnaire}. More complex models consider several families of neutrons, with different energies, which leads to a vector-valued problem. We refer e.g. to~\cite{ciarletNumericalAnalysisMixed2018,duderstadt,mccallien_1970,woznicki_1998,maligeEtudeMathematiqueNumerique1996} for more details on this applicative context.

\medskip

We seek a numerical approximation of~\eqref{pb reaction diffusion stationnaire}. A classical finite element method (FEM), say a standard $\PP_1$ approach, requires a mesh sufficiently fine to capture the heterogeneous nature of the coefficients. This is the case even if the objective is simply to approximate the eigenvalue $\lambda^\eps$, or to approximate the macroscopic features of the eigenfunction $u^\eps$. Using such a fine mesh leads to a prohibitively expensive discrete problem. Alternative approaches are thus in order. In the case of purely diffusive problems with oscillatory coefficients, several dedicated methods have been introduced to adequately capture the oscillatory behaviour of $A^\eps$ on a coarse mesh: we mention the Heterogeneous Multiscale Method~\cite{abdulleHeterogeneousMultiscaleMethod2012,eHeterogeneousMultiscaleMethods2003}, the Localized Orthogonal Decomposition~\cite{altmannNumericalHomogenizationScale2021,malqvistLocalizationEllipticMultiscale2014}, and the Multiscale Finite Element Method (MsFEM, \cite{houMultiscaleFiniteElement1997a,Efendiev,lebrisExamplesComputationalApproaches2017}). To the best of our knowledge, none of these methods has been adapted to the case of reaction-diffusion eigenproblems of the form~\eqref{pb reaction diffusion stationnaire}, and this is the objective of this work, focusing on MsFEM-type approaches.

We recall that MsFEM consists in a Galerkin approximation of the problem under consideration, on a basis set generated by precomputed basis functions which are well-adapted to the fine-scale properties of the differential operator of interest (as opposed to a basis of generic polynomial functions, as in classical FEM). These basis functions are defined as the solutions to local problems that resemble the problem of interest. The MsFEM is a two-step procedure. The offline stage, which is the computationally expensive phase, involves solving local problems (posed on each element of a coarse mesh) to construct the basis functions. The coefficients of these local problems are oscillatory: a fine mesh of each coarse element is introduced to perform in practice these computations. In the online phase, the global problem is discretized (in a Galerkin fashion) on the discretization space built in the offline phase.

The precomputation of the basis functions in the offline stage has a cost, since a fine mesh has to be used. Overall, there is however a significant computational gain if the global problem is to be solved multiple times (and this fact is also true for the other multiscale numerical approaches alluded to above). In such cases, the adapted basis functions indeed have to be computed only once and the dimension of the global discrete problem is drastically reduced in comparison to that of a direct approach put in action on a fine mesh. In our setting, a first multi-query context is the situation when one wants to compute several eigencouples (and not only the first one) of~\eqref{pb reaction diffusion stationnaire}. As shown in~\cite[Chapter~4, Section~4.1]{these_lefort}, we can actually use the same MsFEM basis functions for all eigencouples. A second relevant context is when one considers the time-dependent problem associated to~\eqref{pb reaction diffusion stationnaire}. We can then use the same MsFEM basis functions for all time steps (see~\cite[Chapter~5]{these_lefort}). Yet another multi-query context is related to some specific applications of~\eqref{pb reaction diffusion stationnaire}, namely neutronics problems, as described in~\cite{allaireOptimizationNuclearFuel2002}. In that context, the first eigencouple of~\eqref{pb reaction diffusion stationnaire} has to be computed for a lot of different spatial recombinations of the coefficients $A^\eps$, $\Sigma^\eps$ and $\sigma^\eps$ (in the simplest case, this corresponds to exchanging the values of the coefficients over non-overlapping subdomains of $\Omega$). An idea is then to re-use the MsFEM basis functions for all these different spatial recombinations by recombining the basis functions in the same way as the coefficients are recombined (see~\cite[Chapter~4, Section~4.2]{these_lefort} for some preliminary tests).

\medskip

The numerical method we describe in this article aims to solve the problem~\eqref{pb reaction diffusion stationnaire} in a general setting, and not only when $A^\eps$, $\Sigma^\eps$ and $\sigma^\eps$ are periodic. However, to proceed in a pedagogical manner, we first detail some known homogenization results in the periodic framework, which serve as the foundation for constructing the numerical method.

The article is organized as follows. In Section~\ref{sec:one-group energy case}, we consider the scalar-valued version of~\eqref{pb reaction diffusion stationnaire}, which is then a self-adjoint problem. We first recall some periodic homogenization results, and next describe a preliminary method, which is a MsFEM-type method which uses some objects introduced by the homogenization theory. We then present the actual MsFEM approach, which is a generalization (based on filtering ideas) of the preliminary method. This actual method can be put in practice without any geometric assumptions (e.g. periodicity) on the coefficients. We investigate the efficiency of this method on several numerical cases, including periodic cases (for which we can also consider, as a matter of comparison, the preliminary method) and non-periodic cases (to assess the robustness of the method). In Section~\ref{sec:MsFEM multigroup}, we explain how to extend the MsFEM approach to the vector-valued version of~\eqref{pb reaction diffusion stationnaire}. In that case, which is very relevant from the application viewpoint (think of the multiple energy case in neutronics), the problem is in general non self-adjoint, and thus mathematically more challenging. Similarly to Section~\ref{sec:one-group energy case}, we first recall some periodic homogenization results, before introducing our MsFEM method and illustrating its efficiency on the basis of several numerical test cases. We conclude this article by collecting some numerical analysis results. In Section~\ref{sec:demo MsFEM preliminaire}, we present a proof of convergence of the preliminary method, in the scalar-valued case. In Section~\ref{sec:demo small perturbations}, we present a complete analysis of the filtering method in a restricted setting, complemented with several numerical comparisons in more general settings.

\medskip

In short, the conclusions of this article are the following. First, in the periodic setting, the two numerical approaches we introduce, the preliminary one and the actual one, provide results with very close accuracy. The filtering ideas that we put in practice thus yield an accurate approximation of the theoretical objects introduced by the homogenization theory. Second, our actual approach is robust with respect to the presence of small scales, and its accuracy improves, at fixed size of the coarse mesh, when $\eps$ decreases (so that each coarse element contains a larger number of periodic cells). Third, for periodic and quasi-periodic problems deemed equally difficult (by the equally bad accuracy provided by a standard $\PP_1$ method), our approach yields results of similar accuracy, thereby demonstrating its robustness with respect to the microstructure. Fourth, the accuracy improves when the coarse mesh size decreases, with the limitation that each coarse element should contain a sufficient number of small scale oscillations. Last, all these conclusions are not limited to the scalar-valued variant of the problem, but also hold for the vector-valued version (which is, as pointed out above, mathematically different from the scalar version in that it is non self-adjoint).

\section{The scalar-valued case} \label{sec:one-group energy case}

We focus in this section on the case where the problem of interest is scalar-valued. We thus consider~\eqref{pb reaction diffusion stationnaire} where, for any $x \in \Omega$, $A^\eps(x)$ is a $d \times d$ symmetric matrix, $A^\eps$ belongs to $L^\infty(\Omega)$ and satisfies the following bounds: there exists $\beta \geq \alpha > 0$ such that, for any $\eps$ and any $\xi,\eta \in \RR^d$, we have
\begin{equation*} 
  \alpha \, |\xi|^2 \leq \xi^T A^\eps(x) \, \xi \qquad \text{and} \qquad | \eta^T A^\eps(x) \, \xi | \leq \beta \, |\eta| \, |\xi| \qquad \text{a.e. in $\Omega$}. 
\end{equation*}
The functions $\Sigma^\eps$ and $\sigma^\eps$ are assumed to be scalar-valued, to belong to $L^\infty(\Omega)$ and to be bounded away from 0:
\begin{equation*} 
  \alpha \leq \sigma^\eps(x) \qquad \text{and} \qquad \alpha \leq \Sigma^\eps(x) \qquad \text{a.e. in $\Omega$}. 
\end{equation*}
In Sections~\ref{sec:known results periodic framework} and~\ref{sec:MsFEM preliminaire}, we assume the coefficients to be periodic (which, we recall, is only a preliminary step toward the generality we aim at):
\begin{equation} \label{eq:period}
  A^\eps := A(\cdot/\eps), \quad \Sigma^\eps := \Sigma(\cdot/\eps) \quad \text{and} \quad \sigma^\eps := \sigma(\cdot/\eps),
\end{equation}
where $A$, $\Sigma$ and $\sigma$ are $Y$-periodic functions, where $Y$ is the periodicity cell (a typical choice is $Y = (0,1)^d$). We assume that $A$, $\Sigma$ and $\sigma$ all belong to $L^\infty(Y)$ and satisfy the following bounds: for any $\xi,\eta \in \RR^d$,
\begin{equation} \label{hypothese coercivite D}
  \alpha \, |\xi|^2 \leq \xi^T A(y) \, \xi, \quad | \eta^T A(y) \, \xi | \leq \beta \, |\eta| \, |\xi|, \quad \alpha \leq \sigma(y), \quad \alpha \leq \Sigma(y)
\end{equation}
a.e. in $Y$. Under this periodicity assumption, after having recalled homogenization results in Section~\ref{sec:known results periodic framework}, we will be in position, in Section~\ref{sec:MsFEM preliminaire}, to introduce a preliminary MsFEM-type method. The periodicity assumption is then relaxed in Section~\ref{sec:actual_msfem_general}, where we introduce our actual MsFEM approach. Numerical results are collected in Section~\ref{sec:num_one_group}.

\subsection{Periodic homogenization results} \label{sec:known results periodic framework}

We recall the following results.

\begin{theorem}[Theorem~2.2 of~\cite{allaireHomogenizationSpectralProblem2000}] \label{thm:spectral}
  Under Assumptions~\eqref{eq:period} and~\eqref{hypothese coercivite D}, Problem~\eqref{pb reaction diffusion stationnaire} admits a countable number of eigenvalues, which are all real and positive. The first (i.e. smallest) eigenvalue is simple, and the corresponding eigenfunction can be chosen positive on $\Omega$. In addition, all eigenfunctions of~\eqref{pb reaction diffusion stationnaire} belong to $H^1_0(\Omega) \cap C^{0,s}(\Omega)$ for some $s > 0$.
\end{theorem}

The first assertion of Theorem~\ref{thm:spectral} stems from the fact that, since the diffusion matrix $A$ is symmetric, Problem~\eqref{pb reaction diffusion stationnaire} defines a compact self-adjoint operator acting in $L^2(\Omega)$.

\medskip

We next consider the following cell eigenvalue problem: find $(\lambda^\infty,\psi) \in \RR \times H^1_{\rm per}(Y)$ such that
\begin{equation} \label{pb spectral}
  \Sigma \, \psi - \operatorname{div} \left(A \nabla \psi \right) = \lambda^\infty \, \sigma \, \psi \ \ \text{in $\RR^d$}, \qquad \text{$\psi$ is $Y$-periodic},
\end{equation}
where we recall that $H^1_{\rm per}(Y)$ is the subspace of $H^1_{\rm loc}(\RR^d)$ made of $Y$-periodic functions.

\begin{theorem}[Corollary~2.5 and Proposition~2.6 of~\cite{allaireHomogenizationSpectralProblem2000}] \label{thm:spectral psi}
  Under Assumption~\eqref{hypothese coercivite D}, Problem~\eqref{pb spectral} admits a countable number of eigenvalues, which are all real and positive. The first (i.e. smallest) eigenvalue is simple, and the corresponding eigenfunction can be chosen positive on $Y$. In addition, all eigenfunctions of~\eqref{pb spectral} belong to $H^1_{\rm per}(Y) \cap C^{0,s}_{\rm per}(Y)$ for some $s > 0$.
\end{theorem}

Since $\psi$ is positive and belongs to $C^{0,s}_{\rm per}(Y)$, we have that
\begin{equation} \label{eq:borne_psi}
  0 < c_- \leq \psi(y) \leq c_+ \quad \text{on $Y$},
\end{equation}
for some $c_+ \geq c_- > 0$.

\begin{rmq}
  The eigenvectors of~\eqref{pb reaction diffusion stationnaire} and~\eqref{pb spectral} are defined up to a multiplicative constant. For the first eigenvectors, we choose to fix these constants so that the $L^2$ norm of the eigenvectors is equal to 1, and such that the (first) eigenvectors are positive.
\end{rmq}

The next result (which is key to understand the construction of our approach) has been proved in~\cite{allaireAnalyseAsymptotiqueSpectrale1997,maligeEtudeMathematiqueNumerique1996}.

\begin{theorem}[From~\cite{allaireAnalyseAsymptotiqueSpectrale1997}] \label{changement variable veps}
  Consider the assumptions~\eqref{eq:period} and~\eqref{hypothese coercivite D}. Let $(\lambda^\eps,u^\eps)$ be an eigencouple of~\eqref{pb reaction diffusion stationnaire}, and let $(\lambda^\infty,\psi)$ be the first eigencouple of~\eqref{pb spectral}. Consider $\dis v^\eps := u^\eps/\psi(\cdot/\eps)$. We then have
  \begin{equation} \label{pb reaction diffusion stationnaire veps}
    - \operatorname{div}\left[\widetilde{A}\left(\frac{\cdot}{\eps}\right) \nabla v^\eps \right] = \nu^\eps \, \sigma\left(\frac{\cdot}{\eps}\right) \, \psi^2\left(\frac{\cdot}{\eps}\right) \, v^\eps \ \ \text{in $\Omega$}, \qquad v^\eps = 0 \ \ \text{on $\partial \Omega$},
  \end{equation}
  where $\widetilde{A}(y) = A(y) \, \psi^2(y)$ and $\dis \nu^\eps = \frac{\lambda^\eps - \lambda^\infty}{\eps^2}$.
\end{theorem}

The function $v^\eps$ is thus an eigenvector of a purely diffusive (generalized) eigenvalue problem. Note that, since we have taken $\psi$ to be the {\em first} eigenfunction of~\eqref{pb spectral}, the function $\psi$ does not vanish and we can thus properly define $v^\eps$. In the case when $u^\eps$ is the first eigenfunction of~\eqref{pb reaction diffusion stationnaire}, $v^\eps$ is the first eigenfunction of~\eqref{pb reaction diffusion stationnaire veps}.

The homogenized limit of~\eqref{pb reaction diffusion stationnaire veps} is easier to identify than that of~\eqref{pb reaction diffusion stationnaire}, since the former problem is a purely diffusive problem. In particular, it is easy to establish a priori bounds on~\eqref{pb reaction diffusion stationnaire veps}: the eigenvalue $\nu^\eps$ can be bounded using the min-max principle, while the eigenvector $v^\eps$ can be bounded in $H^1_0(\Omega)$ by standard energy estimates. This has been achieved in~\cite{allaireAnalyseAsymptotiqueSpectrale1997}.

Consider the $m$-th eigencouple $\left(\nu^{\eps,m},v^{\eps,m}\right)$ of~\eqref{pb reaction diffusion stationnaire veps}, where we have ordered the eigenvalues as $0 < \nu^{\eps,1} < \nu^{\eps,2} \leq \nu^{\eps,3} \leq \dots$ ($\nu^{\eps,m}$ is hence the $m$-th eigenvalue and $v^{\eps,m}$ is an associated eigenvector). When $\eps \to 0$, and up to a subsequence, $\nu^{\eps,m}$ converges to some $\nu^{\star,m}$, and $v^{\eps,m}$ converges (weakly in $H^1(\Omega)$) to some $v^{\star,m}$, where $\left(\nu^{\star,m},v^{\star,m}\right)$ is the $m$-th eigencouple of the homogenized problem
\begin{equation} \label{pb reaction diffusion homogeneise}
  - \operatorname{div} \left(\widetilde{A}^\star \, \nabla v^\star \right) = \nu^\star \, \sigma^\star \, v^\star \ \ \text{in $\Omega$}, \qquad v^\star = 0 \quad \text{on $\partial \Omega$}.
\end{equation}
More precisely, $\nu^{\star,m}$ is the $m$-th eigenvalue of~\eqref{pb reaction diffusion homogeneise} and $v^{\star,m}$ is an associated eigenvector. The homogenized coefficients are given by
\begin{equation*} 
  \sigma^\star = \int_Y \sigma \, \psi^2
\end{equation*}
and, for any $1 \leq i,j \leq d$,
\begin{equation*} \label{expression D etoile bis}
  [\widetilde{A}^\star]_{ij} = \int_Y \psi^2 \, e_i^T A \left(e_j + \nabla \widetilde{w}_j \right),
\end{equation*}
where the corrector function $\widetilde{w}_j$ satisfies the equation
\begin{equation} \label{eq correcteur}
  -\operatorname{div}_y \left[ \psi^2 \, A \left( e_j + \nabla_y \widetilde{w}_j \right) \right] = 0 \ \ \text{in $\RR^d$}, \qquad \text{$\widetilde{w}_j$ is $Y$-periodic}.
\end{equation}
The convergence of the eigenvectors holds up to a subsequence because of a possible multiplicity of the limit eigenvalue. If the limit eigenvalue is simple (which is for instance the case for the first one), then the whole sequence converges.

\medskip

In~\cite[Remark~6.2]{allaireHomogenizationSpectralProblem2000}, a strong $H^1_0(\Omega)$ convergence result is also stated, as a corollary of~\cite[Theorem~6.1]{allaireHomogenizationSpectralProblem2000}.

\begin{theorem}[From~\cite{allaireHomogenizationSpectralProblem2000}] \label{thm:strong convergence veps}
  Consider the assumptions~\eqref{eq:period} and~\eqref{hypothese coercivite D}. Let $v^\eps$ be the first eigenvector of~\eqref{pb reaction diffusion stationnaire veps} and $v^\star$ be the first eigenvector of~\eqref{pb reaction diffusion homogeneise}. Let $\dis v^{\eps,1} := v^\star + \eps \sum_{j=1}^d \widetilde{w}_j(\cdot/\eps) \, \partial_j v^\star$, where $\{ \widetilde{w}_j \}_{1\leq j \leq d}$ are the corrector functions defined by~\eqref{eq correcteur}. We assume that $v^\star \in H^1_0(\Omega) \cap H^2(\Omega)$. Then
  \begin{equation} \label{thm:strong convergence veps L2}
    R(\eps) := \| v^\eps - v^{\eps,1} \|_{L^2(\Omega)} \underset{\eps \to 0}{\longrightarrow} 0
  \end{equation}
  and
  \begin{equation} \label{thm:strong convergence veps H1}
    \widetilde{R}(\eps) := \| \nabla v^\eps - \nabla v^{\eps,1} \|_{L^2(\Omega)} \underset{\eps \to 0}{\longrightarrow} 0.
  \end{equation}
\end{theorem}

\subsection{Description of a preliminary method} \label{sec:MsFEM preliminaire}

We introduce a regular coarse mesh $\mathcal{T}_H$ of $\Omega$ consisting of triangles $K$ with characteristic size $H$ (see Figure~\ref{fig:mesh} below; in the following and to fix the ideas, we will have in mind the two-dimensional situation, although our approach can be extended to the three-dimensional situation). The mesh $\mathcal{T}_H$ is coarse in the sense that $H$ does not need to be smaller than $\eps$ (in practice, in our numerical experiments, we always take $H \geq \eps$). Let $N$ denote the number of interior vertices of the mesh $\mathcal{T}_H$, and $V_{\PP_1,H}$ the approximation space consisting of $\PP_1$ functions on the mesh $\mathcal{T}_H$. For any $1 \leq i \leq N$, we denote $\chi_i^{\PP_1}$ the $\PP_1$ function associated with the internal vertex $i$.

We seek the first eigencouple $(\lambda^\eps,u^\eps)$ of~\eqref{pb reaction diffusion stationnaire}, and we assume in this section, as mentioned above, to be in the periodic setting~\eqref{eq:period}. We have seen in Theorem~\ref{changement variable veps} that, with the change of unknown function $v^\eps = u^\eps/\psi(\cdot/\eps)$, the problem can be recast as a purely diffusive generalized eigenvalue problem: look for the first eigencouple $(\nu^\eps,v^\eps)$ of
\begin{equation} \label{pb veps periodic}
  \left\{
  \begin{aligned}
    - \operatorname{div}\left[ \psi^2\left(\frac{x}{\eps}\right) A\left(\frac{x}{\eps}\right) \nabla v^\eps \right] &= \nu^\eps \, \sigma\left(\frac{x}{\eps}\right) \, \psi^2\left(\frac{x}{\eps}\right) \, v^\eps \ \ \text{in $\Omega$},
    \\
    v^\eps &= 0 \ \ \text{on $\partial \Omega$}.
  \end{aligned}
  \right.
\end{equation}
Assume temporarily that our problem of interest is~\eqref{pb veps periodic}, and that we wish to solve it using an MsFEM method. Since this problem is a purely diffusive problem for the eigenvector $v^\eps$, we can use the space generated by the MsFEM basis functions $\left\{ \chi_i^{\eps,\psi} \right\}_{1 \leq i \leq N}$, solutions in $H^1_0(\Omega)$ to
\begin{equation} \label{chi psi}
  \forall K \in \mathcal{T}_H,
  \quad 
  \left\{
  \begin{aligned}
    -\operatorname{div}\left[ \psi^2\left(\frac{\cdot}{\eps}\right) A\left(\frac{\cdot}{\eps}\right) \nabla \chi_i^{\eps,\psi} \right] &= 0 && \text{in $K$},
    \\
    \chi_i^{\eps,\psi} &= \chi_i^{\PP_1} && \text{on $\partial K$}.
  \end{aligned}
  \right.
\end{equation}
These basis functions are called MsFEM-lin basis functions in the literature, in view of the specific boundary conditions in~\eqref{chi psi} (we refer to~\cite{houMultiscaleFiniteElement1997a} for a seminal introduction of the MsFEM approach). An MsFEM approach on~\eqref{pb veps periodic} thus consists in a Galerkin approximation of~\eqref{pb veps periodic} on the finite dimensional space
$$
V_{\eps,\psi,H}^{\rm interm} = \spn \left\{ \chi_1^{\eps,\psi}, \hdots, \chi_N^{\eps,\psi} \right\},
$$
where the superscript 'interm' stands for {\em intermediate}. We thus look for the first eigencouple $(\nu_H^{\eps,\psi},v_H^{\eps,\psi}) \in \RR \times V_{\eps,\psi,H}^{\rm interm}$ such that, for any $w \in V_{\eps,\psi,H}^{\rm interm}$,
\begin{equation*} 
\int_\Omega \psi^2\left(\frac{\cdot}{\eps}\right) (\nabla w)^T A\left(\frac{\cdot}{\eps}\right) \nabla v_H^{\eps,\psi} = \nu_H^{\eps,\psi} \int_\Omega \sigma\left(\frac{\cdot}{\eps}\right) \, \psi^2\left(\frac{\cdot}{\eps}\right) \, v_H^{\eps,\psi} \, w.
\end{equation*}

\begin{rmq}
  In practice, problems~\eqref{chi psi} need to be discretized, using for instance a fine triangular mesh of each element $K \in \mathcal{T}_H$, with a mesh of characteristic size $h \ll \eps$ (see Figure~\ref{fig:mesh}).
\end{rmq}

We now return to our problem of interest, which is, we recall, to look for the first eigencouple $(\lambda^\eps,u^\eps)$ of
\begin{equation} \label{pb reaction diffusion stationnaire periodic}
  \left\{
  \begin{aligned}
    \Sigma\left(\frac{x}{\eps}\right) \, u^\eps - \eps^2 \operatorname{div}\left(A\left(\frac{x}{\eps}\right) \nabla u^\eps \right) &= \lambda^\eps \, \sigma\left(\frac{x}{\eps}\right) \, u^\eps \ \ \text{in $\Omega$},
    \\
    u^\eps &= 0 \ \ \text{on $\partial \Omega$}.
  \end{aligned}
  \right.
\end{equation}
In view of the relation $u^\eps = \psi(\cdot/\eps) \, v^\eps$ and of the approximation procedure for $v^\eps$ presented above, it is natural to introduce the MsFEM basis functions $\left\{ \phi_i^{\eps,\psi} \right\}_{1 \leq i \leq N}$ defined by
\begin{equation} \label{eq:def_phi_prelim}
\forall 1 \leq i \leq N, \quad \phi_i^{\eps,\psi} = \psi(\cdot/\eps) \, \chi_i^{\eps,\psi},
\end{equation}
and the approximation space defined by
\begin{equation} \label{eq:def_V_eps_psi_H}
V_{\eps,\psi,H} = \spn \left\{ \phi_1^{\eps,\psi}, \hdots, \phi_N^{\eps,\psi} \right\}.
\end{equation}
The MsFEM approximation of~\eqref{pb reaction diffusion stationnaire periodic} consists in finding the first eigencouple $(\lambda_H^{\eps,\psi},u^{\eps,\psi}_H) \in \RR \times V_{\eps,\psi,H}$ such that, for any $w \in V_{\eps,\psi,H}$,
\begin{equation} \label{pb reaction diffusion stationnaire MsFEM ueps triche}
  \eps^2 \int_\Omega (\nabla w)^T A\left(\frac{x}{\eps}\right) \nabla u^{\eps,\psi}_H + \int_\Omega \Sigma\left(\frac{x}{\eps}\right) \, u^{\eps,\psi}_H \, w = \lambda_H^{\eps,\psi} \int_\Omega \sigma\left(\frac{x}{\eps}\right) \, u^{\eps,\psi}_H \, w,
\end{equation}
an approach that we denote the {\em preliminary MsFEM approach}. We establish below the following error bound between the solutions to~\eqref{pb reaction diffusion stationnaire MsFEM ueps triche} and~\eqref{pb reaction diffusion stationnaire periodic}:
$$
\frac{\| u^\eps - u^{\eps,\psi}_H \|_{H^1(\Omega)}}{\| u^\eps \|_{H^1(\Omega)}} \leq C \left( \eps + R(\eps) + \eps \, \sqrt{\frac{\eps}{H}} + H^2 \right),
$$
where the function $R(\eps)$ (which goes to 0 when $\eps \to 0$) is defined in~\eqref{thm:strong convergence veps L2}. We refer to Section~\ref{sec:demo MsFEM preliminaire} for a precise statement of this result (see Theorem~\ref{thm taux de convergence methode triche} there), along with its proof. We postpone the presentation of the numerical results obtained with the approach~\eqref{pb reaction diffusion stationnaire MsFEM ueps triche} to Section~\ref{sec:num_one_group}.

\subsection{Description of the actual MsFEM method} \label{sec:actual_msfem_general}

Our objective is to construct a numerical method that does not rely on the periodicity of the microstructure. In this general framework, the function $\psi$, that we explicitly used in the preliminary method, does not exist. To proceed, we thus need to construct a proxy, satisfying the following two requirements: (i) the construction of this proxy is not restricted to the case when $A^\eps$, $\Sigma^\eps$ and $\sigma^\eps$ are periodic, and (ii) in the periodic setting, this proxy is a reliable approximation of $\psi(\cdot/\eps)$. We then hope (and the numerical results presented in Section~\ref{sec:num_one_group} will confirm this) that this proxy leads to an efficient MsFEM method, both in periodic and non-periodic cases. The construction of the proxy is presented in Section~\ref{sec:approx_psi}. The resulting general MsFEM approach is next introduced in Section~\ref{sec:actual_msfem}.

\subsubsection{Construction of a proxy approximating $\psi(\cdot/\eps)$} \label{sec:approx_psi}

The approach we present here is based on a filtering technique introduced in~\cite{cances2005long}, and next used in the context of periodic homogenization in~\cite{blancImprovingComputationHomogenized2010}. Our approach is also based on oversampling, an idea introduced in the MsFEM context in~\cite{houMultiscaleFiniteElement1997a}.

On each element $K$ of the coarse mesh $\mathcal{T}_H$, we wish to construct a proxy, denoted $\widetilde{\psi}^\eps$ (to simplify the notation, we do not explicitly mention the fact that this proxy depends on $K$), that, in the periodic case, approximates the function $\psi(\cdot/\eps)|_K$. A first step consists in introducing a square-shaped oversampling patch (and more generally, an hypercube) $S_K$ around the element $K$, as shown in Figure~\ref{fig:patch_oversampling_comparison} below. A naive idea is then to consider the eigenvector with periodic boundary conditions $\widetilde{\psi}^{\eps,\#} \in H^1_{\rm per}(S_K)$ associated to the first eigenvalue $\lambda^{\eps,\#}$ of the following eigenvalue problem, posed on $S_K$: for any $v \in H^1_{\rm per}(S_K)$,
\begin{equation} \label{pb patch oversampling}
  \int_{S_K} \Sigma^\eps \, \widetilde{\psi}^{\eps,\#} \, v + \eps^2 \int_{S_K} (\nabla v)^T A^\eps \nabla \widetilde{\psi}^{\eps,\#} = \lambda^{\eps,\#} \int_{S_K} \sigma^\eps \, \widetilde{\psi}^{\eps,\#} \, v.
\end{equation}
Note that it is possible to impose periodic boundary conditions on the boundary of $S_K$ because this domain is a cube (and not a polyhedra, as is often the case in oversampling). In the case~\eqref{eq:period} when $A^\eps$, $\Sigma^\eps$ and $\sigma^\eps$ are periodic, if the size of $S_K$ is an integer multiple of the size of the periodic cell $\eps \, Y$, then it can be shown that $\widetilde{\psi}^{\eps,\#} = \psi(\cdot/\eps)$ on $S_K$. However, in the general case, $\widetilde{\psi}^{\eps,\#}$ and $\psi(\cdot/\eps)$ can be very different one from each other, as discussed in~\cite{allaire2012homogenization}. Following~\cite{blancImprovingComputationHomogenized2010}, we are going to use a filter function, in order to mitigate the fact that $S_K$ does not contain an integer number of periodic cells.

Let $\tau_0$ be a function such that, for some $k \in \NN^\star$ (in practice, we will take $k=1$ or 2),
\begin{equation} \label{hypotheses filtre varphi0}
  \left\{
  \begin{aligned}
    &\tau_0 \in C^{k+1}([0,1]), \qquad \tau_0 > 0 \ \ \text{in $(0,1)$}, \qquad \int_0^1 \tau_0 = 1,
    \\
    &\forall 0 \leq i \leq k-1, \ \ \tau_0^{(i)}(0) = \tau_0^{(i)}(1) = 0. 
  \end{aligned}
  \right.
\end{equation}
In dimension $d$, on the hypercube $S_K$, we then introduce the filter function $\tau_K$ defined by
\begin{equation} \label{hypotheses filtre varphi}
  \forall x \in S_K = \prod_{i=1}^d (a_i,b_i), \qquad \tau_K(x) = \frac{1}{|S_K|} \, \prod_{i=1}^d \tau_0\left(\frac{x_i - a_i}{b_i - a_i}\right),
\end{equation}
where $|S_K| = \prod_{i=1}^d (b_i - a_i)$ is the volume of $S_K$. This prefactor ensures that $\dis \int_{S_K} \tau_K = 1$.

The filtered variant of~\eqref{pb patch oversampling} is obtained as follows. Since $\lambda^{\eps,\#}$ is the smallest eigenvalue, Problem~\eqref{pb patch oversampling} can be recast as
$$
\inf \left\{ \int_{S_K} \Sigma^\eps \, \psi^2 + \eps^2 \int_{S_K} (\nabla \psi)^T A^\eps \nabla \psi, \quad \psi \in H^1_{\rm per}(S_K), \quad \int_{S_K} \sigma^\eps \, \psi^2 = 1 \right\}. 
$$
We then introduce the filter function $\tau_K$ in all the above integrals and consider
\begin{multline} \label{eq:filtered_var}
\inf \left\{ \int_{S_K} \tau_K \, \Sigma^\eps \, \psi^2 + \eps^2 \int_{S_K} \tau_K \, (\nabla \psi)^T A^\eps \nabla \psi, \right. \\ \left. \psi \in H^1(S_K), \quad \int_{S_K} \tau_K \, \nabla \psi = 0, \quad \int_{S_K} \tau_K \, \sigma^\eps \, \psi^2 = 1 \right\}. 
\end{multline}

\begin{rmq}
  The first constraint in~\eqref{eq:filtered_var} can easily be understood in the one-dimensional case. A function $\psi \in H^1(S_K)$ is periodic if and only if $\dis \int_{S_K} \psi' = 0$, a constraint that we modify in~\eqref{eq:filtered_var} by introducing the filter function $\tau_K$ in the integral.
\end{rmq}

\medskip

The Euler-Lagrange equation of~\eqref{eq:filtered_var} reads as follows: look for a Lagrange multiplier $\widetilde{\mu}^\eps \in \RR^d$ and the eigenvector $\widetilde{\psi}^\eps \in H^1(S_K)$ associated to the smallest eigenvalue $\widetilde{\lambda}^\eps \in \RR$ such that, for any $v \in H^1(S_K)$ and any $\mu \in \RR^d$,
\begin{equation} \label{pb patch oversampling filtre tmp}
  \left\{
  \begin{aligned}
    \eps^2 \int_{S_K} \tau_K \, (\nabla v)^T A^\eps \nabla \widetilde{\psi}^\eps + \int_{S_K} \tau_K \, \Sigma^\eps \, \widetilde{\psi}^\eps \, v
    &=
    \widetilde{\lambda}^\eps \int_{S_K} \tau_K \, \sigma^\eps \, \widetilde{\psi}^\eps \, v + \widetilde{\mu}^\eps \cdot \int_{S_K} \tau_K \, \nabla v,
    \\
    \mu \cdot \int_{S_K} \tau_K \, \nabla \widetilde{\psi}^\eps
    &=
    0.
  \end{aligned}
  \right.
\end{equation}
We refer to Section~\ref{sec:demo small perturbations} for some elements of analysis of this method, along with some numerical results. In short, we observe in Section~\ref{sec:demo small perturbations} that $\widetilde{\lambda}^\eps$ is a converging approximation (when $\eps \to 0$) of the eigenvalue $\lambda^\infty$ of~\eqref{pb spectral}, and that, on a domain interior to $S_K$ (and in particular on the element $K$ itself), the eigenvector $\widetilde{\psi}^\eps$ is also a converging approximation of $\psi(\cdot/\eps)$, provided a filter of order $k \geq 2$ is used. 

\subsubsection{Our MsFEM approach} \label{sec:actual_msfem}

We now describe the actual MsFEM method that we propose, which is based on a combination of the preliminary method described in Section~\ref{sec:MsFEM preliminaire} and of the filtering ideas presented in Section~\ref{sec:approx_psi}.

\medskip

As mentioned above, for each element $K$ of the mesh $\mathcal{T}_H$, we construct a square-shaped oversampling patch $S_K$ centered around the element $K$, as shown on Figure~\ref{fig:patch_oversampling_comparison}. The size of $S_K$ is defined by an oversampling ratio $\rho$ which is the ratio between the size of the patch and the size of the element $K$. We typically take $S_K$ to be a cube with edges of length $2H$.

\begin{figure}[htbp]
  \centering
  \begin{subfigure}[t]{0.45\textwidth}
    \centering
    \tikzset{every picture/.style={line width=0.75pt, scale=1.2}} 

    \begin{tikzpicture}[x=0.75pt,y=0.75pt,yscale=-1,xscale=1]
      
      \draw   (191,80.96) -- (350,80.96) -- (350,240.96) -- (191,240.96) -- cycle ;
      \draw [color={rgb, 255:red, 155; green, 155; blue, 155 }  ,draw opacity=1 ]   (191,240.96) -- (350,80.96) ;
      \draw [color={rgb, 255:red, 155; green, 155; blue, 155 }  ,draw opacity=1 ]   (191,200.36) -- (350,200.37) ;
      \draw [color={rgb, 255:red, 155; green, 155; blue, 155 }  ,draw opacity=1 ]   (191,120.36) -- (350,120.37) ;
      \draw [color={rgb, 255:red, 155; green, 155; blue, 155 }  ,draw opacity=1 ]   (191,160.96) -- (350,160.96) ;
      \draw [color={rgb, 255:red, 155; green, 155; blue, 155 }  ,draw opacity=1 ]   (230.2,240.96) -- (230.2,80.96) ;
      \draw [color={rgb, 255:red, 155; green, 155; blue, 155 }  ,draw opacity=1 ]   (310.2,240.96) -- (310.2,80.96) ;
      \draw [color={rgb, 255:red, 155; green, 155; blue, 155 }  ,draw opacity=1 ]   (270.5,240.96) -- (270.5,80.96) ;
      \draw   (191,80.96) -- (350,80.96) -- (350,240.96) -- (191,240.96) -- cycle ;
      \draw [color={rgb, 255:red, 155; green, 155; blue, 155 }  ,draw opacity=1 ]   (191,160.96) -- (270.5,80.96) ;
      \draw [color={rgb, 255:red, 155; green, 155; blue, 155 }  ,draw opacity=1 ]   (270.5,240.96) -- (350,160.96) ;
      \draw [color={rgb, 255:red, 155; green, 155; blue, 155 }  ,draw opacity=1 ]   (190.45,120.96) -- (230.2,80.96) ;
      \draw [color={rgb, 255:red, 155; green, 155; blue, 155 }  ,draw opacity=1 ]   (310.2,240.96) -- (349.95,200.96) ;
      \draw [color={rgb, 255:red, 155; green, 155; blue, 155 }  ,draw opacity=1 ]   (191,200.37) -- (310.2,80.96) ;
      \draw [color={rgb, 255:red, 155; green, 155; blue, 155 }  ,draw opacity=1 ]   (230.2,240.96) -- (349.4,121.56) ;
      \draw  [fill={rgb, 255:red, 74; green, 144; blue, 226 }  ,fill opacity=1 ] (253.07,109.55) -- (327.27,109.55) -- (327.27,173.51) -- (253.07,173.51) -- cycle ;
      \draw  [fill={rgb, 255:red, 255; green, 255; blue, 255 }  ,fill opacity=1 ] (270.5,160.96) -- (310.02,121.55) -- (310.29,160.68) -- cycle ;
      \draw   (270.41,161.24) -- (309.93,121.82) -- (310.2,160.96) -- cycle ;
      \draw (292,142) node [anchor=north west][inner sep=0.75pt]    {$\textcolor[rgb]{0.82,0.01,0.11}{\textbf K}$};
      \draw (258.07,114.55) node [anchor=north west][inner sep=0.75pt]  [color={rgb, 255:red, 208; green, 2; blue, 27 }  ,opacity=1 ]  {\LARGE $\textcolor[rgb]{0,0.9,0.2}{\textbf S_{\textcolor[rgb]{0,0.9,0.2}{\textbf K}}}$};
      \draw (350,219) node [anchor=north west][inner sep=0.75pt]    {$\partial \Omega$};
    \end{tikzpicture}
    \caption{Interior element}
  \end{subfigure}
  \hfill
  \begin{subfigure}[t]{0.45\textwidth}
    \centering
    \tikzset{every picture/.style={line width=0.75pt, scale=1.2}} 
    
    \begin{tikzpicture}[x=0.75pt,y=0.75pt,yscale=-1,xscale=1]
      
      \draw   (191,80.96) -- (350,80.96) -- (350,240.96) -- (191,240.96) -- cycle ;
      \draw [color={rgb, 255:red, 155; green, 155; blue, 155 }  ,draw opacity=1 ]   (191,240.96) -- (350,80.96) ;
      \draw [color={rgb, 255:red, 155; green, 155; blue, 155 }  ,draw opacity=1 ]   (191,200.36) -- (350,200.37) ;
      \draw [color={rgb, 255:red, 155; green, 155; blue, 155 }  ,draw opacity=1 ]   (191,120.36) -- (350,120.37) ;
      \draw [color={rgb, 255:red, 155; green, 155; blue, 155 }  ,draw opacity=1 ]   (191,160.96) -- (350,160.96) ;
      \draw [color={rgb, 255:red, 155; green, 155; blue, 155 }  ,draw opacity=1 ]   (230.2,240.96) -- (230.2,80.96) ;
      \draw [color={rgb, 255:red, 155; green, 155; blue, 155 }  ,draw opacity=1 ]   (310.2,240.96) -- (310.2,80.96) ;
      \draw [color={rgb, 255:red, 155; green, 155; blue, 155 }  ,draw opacity=1 ]   (270.5,240.96) -- (270.5,80.96) ;
      \draw   (191,80.96) -- (350,80.96) -- (350,240.96) -- (191,240.96) -- cycle ;
      \draw [color={rgb, 255:red, 155; green, 155; blue, 155 }  ,draw opacity=1 ]   (191,160.96) -- (270.5,80.96) ;
      \draw [color={rgb, 255:red, 155; green, 155; blue, 155 }  ,draw opacity=1 ]   (270.5,240.96) -- (350,160.96) ;
      \draw [color={rgb, 255:red, 155; green, 155; blue, 155 }  ,draw opacity=1 ]   (190.45,120.96) -- (230.2,80.96) ;
      \draw [color={rgb, 255:red, 155; green, 155; blue, 155 }  ,draw opacity=1 ]   (310.2,240.96) -- (349.95,200.96) ;
      \draw [color={rgb, 255:red, 155; green, 155; blue, 155 }  ,draw opacity=1 ]   (191,200.37) -- (310.2,80.96) ;
      \draw [color={rgb, 255:red, 155; green, 155; blue, 155 }  ,draw opacity=1 ]   (230.2,240.96) -- (349.4,121.56) ;
      \draw  [fill={rgb, 255:red, 74; green, 144; blue, 226 }  ,fill opacity=1 ] (295,80.96) -- (350,80.96) -- (350,135) -- (295,135) -- cycle ;
      \draw  [fill={rgb, 255:red, 255; green, 255; blue, 255 }  ,fill opacity=1 ] (310.2,121.55) -- (350,80.96) -- (350,121.55) -- cycle ;
      \draw  [fill={rgb, 255:red, 225; green, 60; blue, 60 }  ,fill opacity=0.3 ] (350,135) -- (367,135) -- (367,64) -- (295,64) -- (295,80.96) -- (350,80.96) --cycle ;
      \draw   (310.2,121.55) -- (350,80.96) -- (350,121.55) -- cycle ;
      \draw (332,102) node [anchor=north west][inner sep=0.75pt]    {$\textcolor[rgb]{0.82,0.01,0.11}{\textbf K}$};
      \draw (298.07,86.55) node [anchor=north west][inner sep=0.75pt]  [color={rgb, 255:red, 208; green, 2; blue, 27 }  ,opacity=1 ]  {\LARGE $\textcolor[rgb]{0,0.9,0.2}{\textbf S_{\textcolor[rgb]{0,0.9,0.2}{\textbf K}}}$};
      \draw (350,219) node [anchor=north west][inner sep=0.75pt]    {$\partial \Omega $};
    \end{tikzpicture}
    \caption{Boundary element}
  \end{subfigure}
  \caption{Oversampling patch $S_K$ for an element $K$ in the interior of $\Omega$ (left) or at the boundary of $\Omega$ (right)}
  \label{fig:patch_oversampling_comparison}
\end{figure}
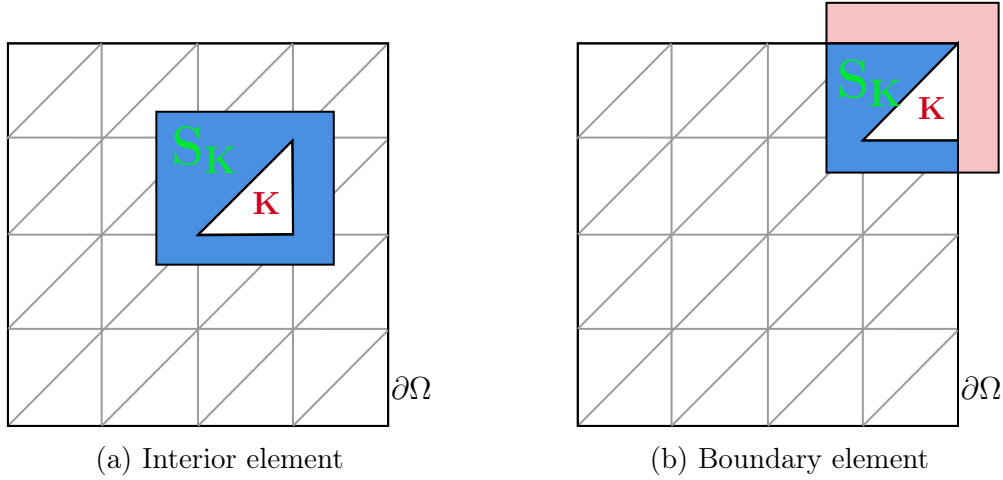

\medskip

We then proceed as follows:
\begin{description}
\item[$\bullet$ Step~1 (offline):] Construction of a proxy.
\end{description}
For each element $K$ of the coarse mesh $\mathcal{T}_H$, we solve~\eqref{pb patch oversampling filtre tmp} on the patch $S_K$ (using, in practice, a fine mesh), and denote here $(\widetilde{\mu}^\eps_{S_K},\widetilde{\lambda}^\eps_{S_K},\widetilde{\psi}_{S_K}^\eps) \in \RR^d \times \RR \times H^1(S_K)$ its solution. On each element $K \in \mathcal{T}_H$, we then define $\widetilde{\psi}^\eps_K := \widetilde{\psi}_{S_K}^\eps|_K$ as a proxy of the function $\psi(\cdot/\eps)$.

\begin{description}
\item[$\bullet$ Step~2 (offline):] Computation of the MsFEM-lin basis functions.
\end{description}
Inspired by~\eqref{chi psi}, we define the functions $\left\{ \chi_i^\eps \right\}_{1 \leq i \leq N}$ as the solutions in $H^1_0(\Omega)$ to
\begin{equation} \label{chi phi}
  \forall K \in \mathcal{T}_H,
  \quad 
  \left\{
  \begin{aligned}
    -\operatorname{div}\left( (\widetilde{\psi}^\eps_K)^2 \, A^\eps \nabla \chi_i^\eps \right) &= 0 && \text{in $K$},
    \\
    \chi_i^\eps &= \chi_i^{\PP_1} &&\text{on $\partial K$}.
  \end{aligned}
  \right.
\end{equation}
Inspired by~\eqref{eq:def_phi_prelim}, we next define the MsFEM basis functions $\left\{ \phi_i^\eps \right\}_{1 \leq i \leq N}$ on $\Omega$ as
\begin{equation*}
  \forall K \in \mathcal{T}_H, \qquad \phi_i^\eps = \chi_i^\eps \ \widetilde{\psi}^\eps_K \ \ \text{in $K$},
\end{equation*}
and we finally consider the approximation space
\begin{equation} \label{MsFEM espace 1 groupe}
  V_{\eps,H} = \spn \left\{ \phi_1^\eps, \hdots, \phi_N^\eps \right\}.
\end{equation}
Note that the function $\widetilde{\psi}^\eps$ may jump at the boundary between one element $K$ and its neighbour. In general, and as is classical for methods using oversampling, the space $V_{\eps,H}$ is hence not a subspace of $H^1(\Omega)$.

\begin{description}
\item[$\bullet$ Step~3 (online):] Solution to the global problem.
\end{description}
We perform a Galerkin approximation of~\eqref{pb reaction diffusion stationnaire} on the space~\eqref{MsFEM espace 1 groupe} and look for the first eigencouple $(\lambda_H^\eps,u^\eps_H) \in \RR \times V_{\eps,H}$ such that, for any $w \in V_{\eps,H}$,
\begin{equation} \label{pb reaction diffusion stationnaire MsFEM 1 groupe ueps}
  \eps^2 \sum_{K \in \mathcal{T}_H} \int_K (\nabla w)^T A^\eps \nabla u^\eps_H + \int_\Omega \Sigma^\eps \, u^\eps_H \, w = \lambda^\eps_H \int_\Omega \sigma^\eps \, u^\eps_H \, w.
\end{equation}
The first integral is written as a broken sum since the discretization is non-conforming.

\subsection{Numerical results} \label{sec:num_one_group}

We now present several numerical experiments illustrating the efficiency of our MsFEM approach. All the computations have been performed with FreeFEM~\cite{hechtNewDevelopmentFreeFem2012} and the associated scripts are available at~\cite{code_alberic}. We only consider here two-dimensional test-cases, although, as mentioned above, our approach can in principle be extended to the three-dimensional setting.

We consider two types of coefficients: periodic coefficients and quasi-periodic coefficients (as an example of non-periodic coefficients). In dimension 2 on the domain $\Omega = (0,1)^2$, we consider a coarse triangular mesh $\mathcal{T}_H$ with $H = 1/8$ (unless otherwise stated). Each element $K$ of this coarse mesh is itself meshed with a fine triangular mesh $\mathcal{T}_h$ with $h \ll \eps$, as illustrated in Figure~\ref{fig:mesh}.

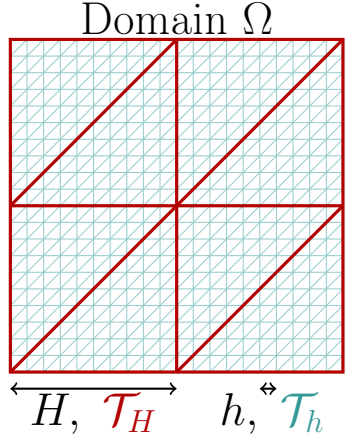
\begin{figure}[htbp]
  \centering
  \begin{tikzpicture}[scale=1.1]
    
    \foreach \x in {0,0.2,...,4}
    \draw[teal!40!white, thin] (\x,0) -- (\x,4);
    \foreach \y in {0,0.2,...,4}
    \draw[teal!40!white, thin] (0,\y) -- (4,\y);
    
    \foreach \i in {0,...,19}
    \foreach \j in {0,...,19}
    \draw[teal!40!white, thin] (0.2*\i,0.2*\j) -- (0.2*\i+0.2,0.2*\j+0.2);
    
    \draw[red!70!black, very thick] (0,0) -- (0,4) -- (4,4) -- (4,0) -- cycle;
    \draw[red!70!black, very thick] (0,2) -- (4,2);
    \draw[red!70!black, very thick] (2,0) -- (2,4);
    
    \draw[red!70!black, very thick] (0,0) -- (2,2);
    \draw[red!70!black, very thick] (2,0) -- (4,2);
    \draw[red!70!black, very thick] (0,2) -- (2,4);
    \draw[red!70!black, very thick] (2,2) -- (4,4);
    
    \node at (2,4.25) {\Large Domain $\Omega$};
    
    \draw[<->, thick] (0,-0.2) -- (2,-0.2);
    \node at (1,-0.5) {\Large $\displaystyle H,\ \textcolor{red!70!black}{\mathcal{T}_H}$};
    
    \draw[<->, thick] (3,-0.2) -- (3.2,-0.2);
    \node at (3.15,-0.5) {\Large $\displaystyle h,\ \textcolor{teal!80!white}{\mathcal{T}_h}$};
        
  \end{tikzpicture}
  \caption{Schematic representation of the domain $\Omega$, the coarse mesh $\mathcal{T}_H$ (in red) and the fine mesh $\mathcal{T}_h$ (in green).}
  \label{fig:mesh}
\end{figure}

A reference solution is computed as the first eigencouple $(\lambda^\eps,u^\eps)$ of~\eqref{pb reaction diffusion stationnaire} on a fine mesh $\mathcal{T}_h$ of the whole domain $\Omega$ (with $h \ll \eps$), using a $\PP_1$-finite element method. We then compute the MsFEM first eigencouple $(\lambda^\eps_H,u^\eps_H)$, solution to~\eqref{pb reaction diffusion stationnaire MsFEM 1 groupe ueps}. Throughout these numerical results, the oversampling ratio defined above is set at $\rho = 2$, and the filter function $\tau_0$ is chosen as $\tau_0(x) = C \, x^2 \, (1-x)^2$ for some constant $C$ such that $\| \tau_0 \|_{L^1(0,1)} = 1$ (this corresponds to a filter of order $k=2$ in the sense of~\eqref{hypotheses filtre varphi0}).

For the scalar-valued variant of the problem considered here, we restrict our tests to the case when $\sigma^\eps = 1$ (more general cases of $\sigma^\eps$ are considered in the vector-valued variant of the problem that we consider in Section~\ref{sec:MsFEM multigroup}). For the sake of comparison, we also compute the first eigencouple obtained by the classical $\PP_1$-method on the coarse mesh $\mathcal{T}_H$. In the periodic test cases, we also compute the first eigencouple obtained by the preliminary method described in Section~\ref{sec:MsFEM preliminaire} (which represents an ideal result we are actually aiming at with our actual MsFEM method). We denote generically by $(\lambda^\eps_H, u^\eps_H)$ the first eigencouple obtained by the MsFEM-method, the preliminary method, or the $\PP_1$-method on the coarse mesh $\mathcal{T}_H$.

The relative error on the eigenvalue is defined by
\begin{equation} \label{erreur valeur propre relative}
  \frac{| \lambda^\eps - \lambda^\eps_H |}{| \lambda^\eps |}. 
\end{equation}
Since the approximation space~\eqref{MsFEM espace 1 groupe} is non-conforming, the error on the eigenvector is defined using the $H^1$ broken norm as
\begin{equation} \label{erreur vecteur propre relative}
  \frac{\sqrt{\sum_{K \in \mathcal{T}_H} \| u^\eps - u^\eps_H \|^2_{H^1(K)}} }{\| u^\eps \|_{H^1(\Omega)}}.
\end{equation}

\subsubsection{Periodic case} \label{sec:num_periodic}

We consider the diffusion and reaction coefficients defined in the periodic cell $Y$ as follows: for any $(y_1,y_2) \in Y$,
\begin{equation} \label{coefficients periodic 1 groupe}
  \begin{aligned}
    A(y_1,y_2) &= \Big[ 6 + 5 \cos\big(2\pi(y_1 + 2y_2)\big) \, \sin\big(2\pi(y_1 - y_2)\big) \Big] \, \text{Id}_2,
    \\[3pt]
    \Sigma(y_1,y_2) &= 20 \Big[ 2 + \cos\big(2\pi(y_1 - 2y_2)\big) \, \sin\big(2\pi(y_1 - y_2)\big) \Big],
  \end{aligned}
\end{equation}
where $\text{Id}_2$ denotes the $2 \times 2$ identity matrix. The corresponding oscillatory coefficients are defined on $\Omega$ by~\eqref{eq:period}.

\begin{rmq} \label{rk:contrast}
  The choice of the periodic coefficients~\eqref{coefficients periodic 1 groupe}, and of the quasi-periodic coefficients~\eqref{coefficients quasi periodic 1 groupe_A}--\eqref{coefficients quasi periodic 1 groupe_S} below, has been made to ensure that the relative $H^1$ error~\eqref{erreur vecteur propre relative} produced by the $\PP_1$ method, in the case when $\eps = H = 1/8$, is of the order of 50\%. This allows for a meaningful comparison of results obtained with different coefficient structures, all exhibiting a similar level of ``difficulty''.
\end{rmq}

The results shown on Figures~\ref{fig:ErreurH1_Penche_avec_filtre_filtre_puissance_1} and~\ref{fig:ErreurVP_Penche_avec_filtre_filtre_puissance_1} are obtained by fixing $H$ (at its value $H=1/8$) and varying $\eps$. The errors are plotted as a function of the ratio $2H/\eps$, which measures how large the oversampling domain $S_K$ is (it is a square with edges of length $2H$) compared to the size of the periodic cell (when $2H/\eps = 1$, the oversampling domain $S_K$ contains exactly one periodic cell). Our aim, by monitoring the results as a function of $\eps$, is to investigate the robustness of the approaches with respect to the presence of small scales in the problem.

As explained in Remark~\ref{rk:contrast}, the classical $\PP_1$ method yields a relative $H^1$ error of 50\% for $\eps = H$, and, as expected, larger errors for smaller values of $\eps$ (until reaching an error of 100\% for $\eps \leq H/2$). This is a pragmatic indication that the problem is indeed a difficult, multiscale problem. The preliminary MsFEM approach yields very accurate results: as soon as $\eps \leq H/2$ (i.e. as soon as $S_K$ indeed contains a few periodic cells), the error~\eqref{erreur vecteur propre relative} is smaller than 20\%, which may be considered, for difficult multiscale problems, as an acceptable level of accuracy (recall also that, for our problem, $u^\eps$ is bounded in the $L^2$ norm but not in the $H^1$ norm, which is a sign that this type of problems is more difficult than the more classical purely diffusive problems). The actual MsFEM approach provides results, the accuracy of which is very similar to those of the preliminary method: this is an indirect indication that our proxy is indeed a reliable approximation of $\psi(\cdot/\eps)$, a fact that we have also checked directly (we refer to Section~\ref{sec:num_filtre} for some results in that vein).

The conclusions for the approximation of the eigenvalue (see Figure~\ref{fig:ErreurVP_Penche_avec_filtre_filtre_puissance_1}) are similar, our actual MsFEM approach providing results very close to those of the preliminary method, and with a relative error of the order (or smaller) than $10^{-3}$ (i.e. 0.1\%).

These results demonstrate the robustness of our MsFEM approach with respect to the value of the characteristic size $\eps$ of the small scales.

\begin{figure}[htpb]
  \centering
  \includegraphics[width=0.9\textwidth]{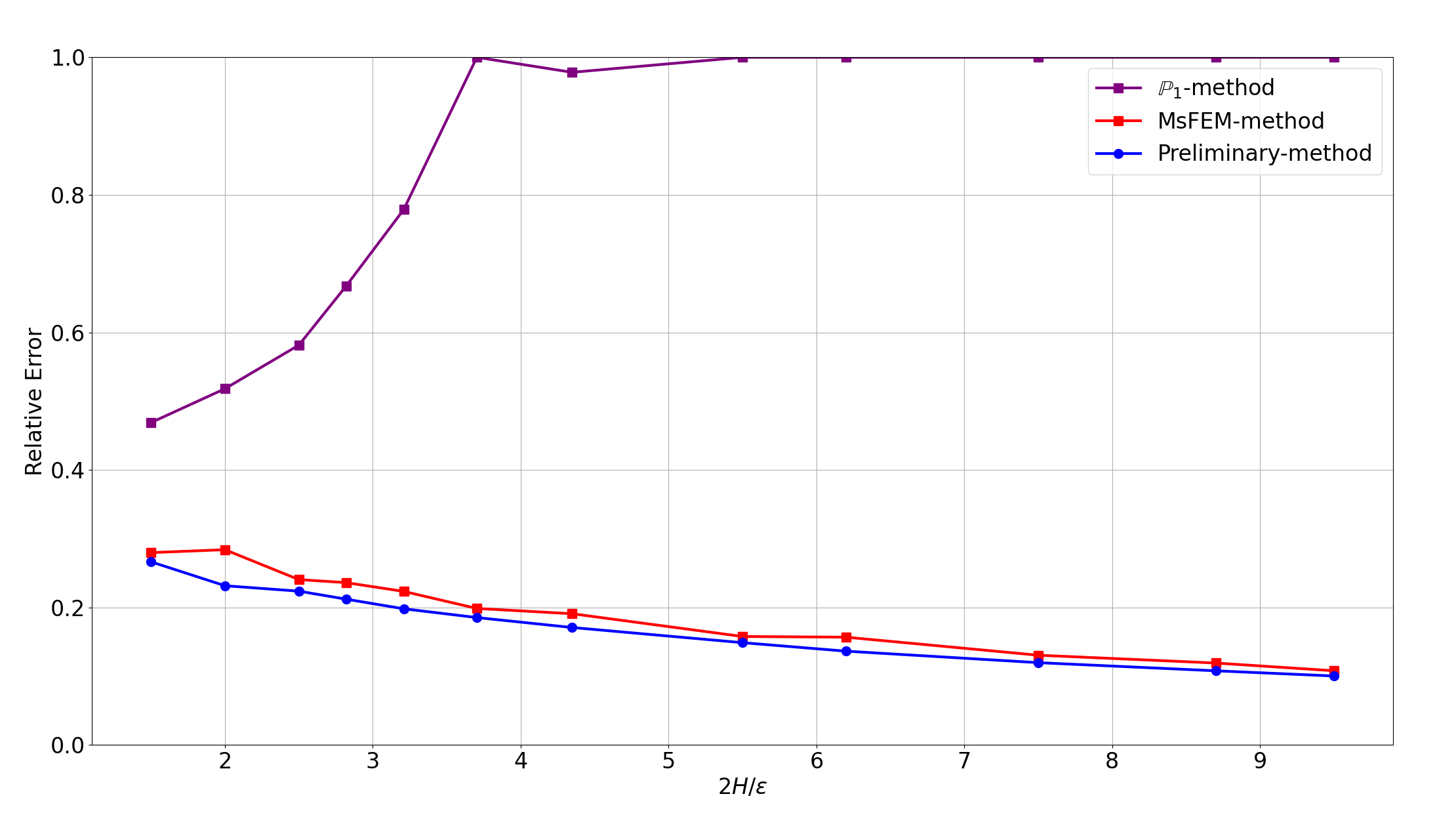}
  \caption{Periodic case~\eqref{coefficients periodic 1 groupe}: relative error~\eqref{erreur vecteur propre relative} on the eigenvector for the MsFEM method, the preliminary method and the $\PP_1$-method, as a function of $\eps$ ($H = 1/8$ fixed; note that the abscissa, here and in many figures below, is $2H/\eps$).}
  \label{fig:ErreurH1_Penche_avec_filtre_filtre_puissance_1}
\end{figure}

\begin{figure}[htbp]
  \centering
  \includegraphics[width=0.9\textwidth]{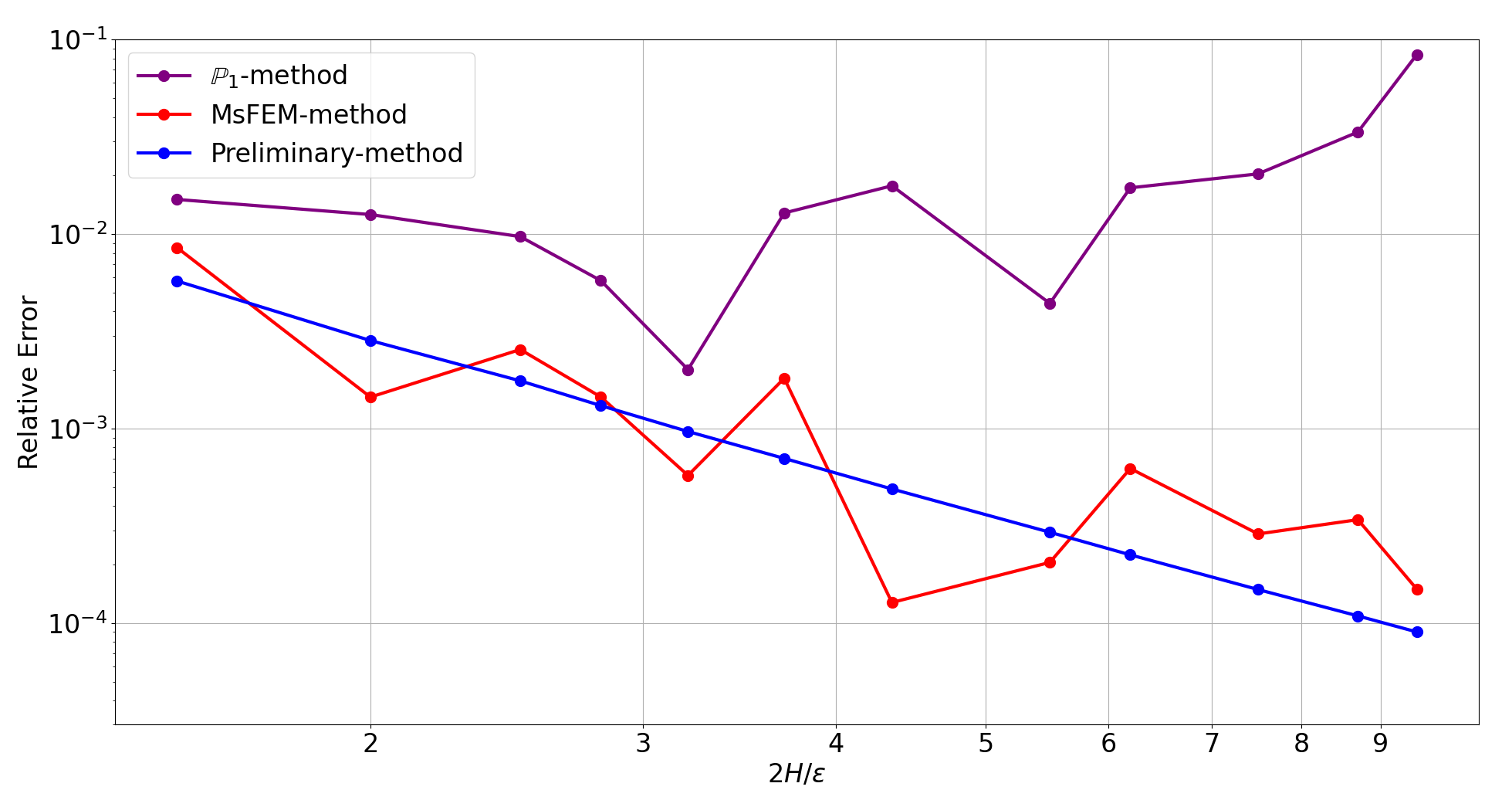}
  \caption{Periodic case~\eqref{coefficients periodic 1 groupe}: relative error~\eqref{erreur valeur propre relative} on the eigenvalue for the MsFEM method, the preliminary method and the $\PP_1$-method, as a function of $\eps$ ($H = 1/8$ fixed, $\operatorname{log}-\operatorname{log}$ scale).}
  \label{fig:ErreurVP_Penche_avec_filtre_filtre_puissance_1}
\end{figure}

\subsubsection{Quasi-periodic case}

We now consider the oscillatory diffusion and reaction coefficients defined on $\Omega$ as follows: for any $(x_1,x_2) \in \Omega$,
\begin{multline} \label{coefficients quasi periodic 1 groupe_A}
  A^\eps(x_1,x_2) = \left( 5 + 1.25 \left[ \cos\left(\frac{2\pi x_1}{\eps}\right) + \cos\left(\frac{2\sqrt{2}\pi x_1}{\eps}\right) \right] \right. \\ \left. \times \left[ \sin\left(\frac{2\pi x_2}{\eps}\right) + \sin\left(\frac{2\sqrt{2}\pi x_2}{\eps}\right) \right] \right) \text{Id}_2
\end{multline}
and
\begin{multline} \label{coefficients quasi periodic 1 groupe_S}
  \Sigma^\eps(x_1,x_2) = 40 \left( 2 + 0.25 \left[ \cos\left(\frac{2\pi x_1}{\eps}\right) + \cos\left(\frac{2\sqrt{2}\pi x_1}{\eps}\right) \right] \right. \\ \left. \times \left[ \sin\left(\frac{2\pi x_2}{\eps}\right) + \sin\left(\frac{2\sqrt{2}\pi x_2}{\eps}\right) \right] \right).
\end{multline}
These coefficients are quasi-periodic, since they are expressed as a sum of periodic functions with periods $\eps$ and $\eps/\sqrt{2}$.

\medskip

Fixing $H=1/8$ and varying $\eps$, we obtain the results shown on Figures~\ref{fig:ErreurH1_QP} and~\ref{fig:ErreurVP_QP}. We again obtain accurate results (with an error of the order of 20\% for the eigenvector, and smaller than 0.1\% for the eigenvalue). We also notice that the errors are of the same order as for the periodic case considered in Section~\ref{sec:num_periodic}: despite the fact that the microstructure of the problem is more complex, our method yields comparable errors. Again, this demonstrates its robustness.

\begin{figure}[htbp]
  \centering
  \includegraphics[width=0.9\textwidth]{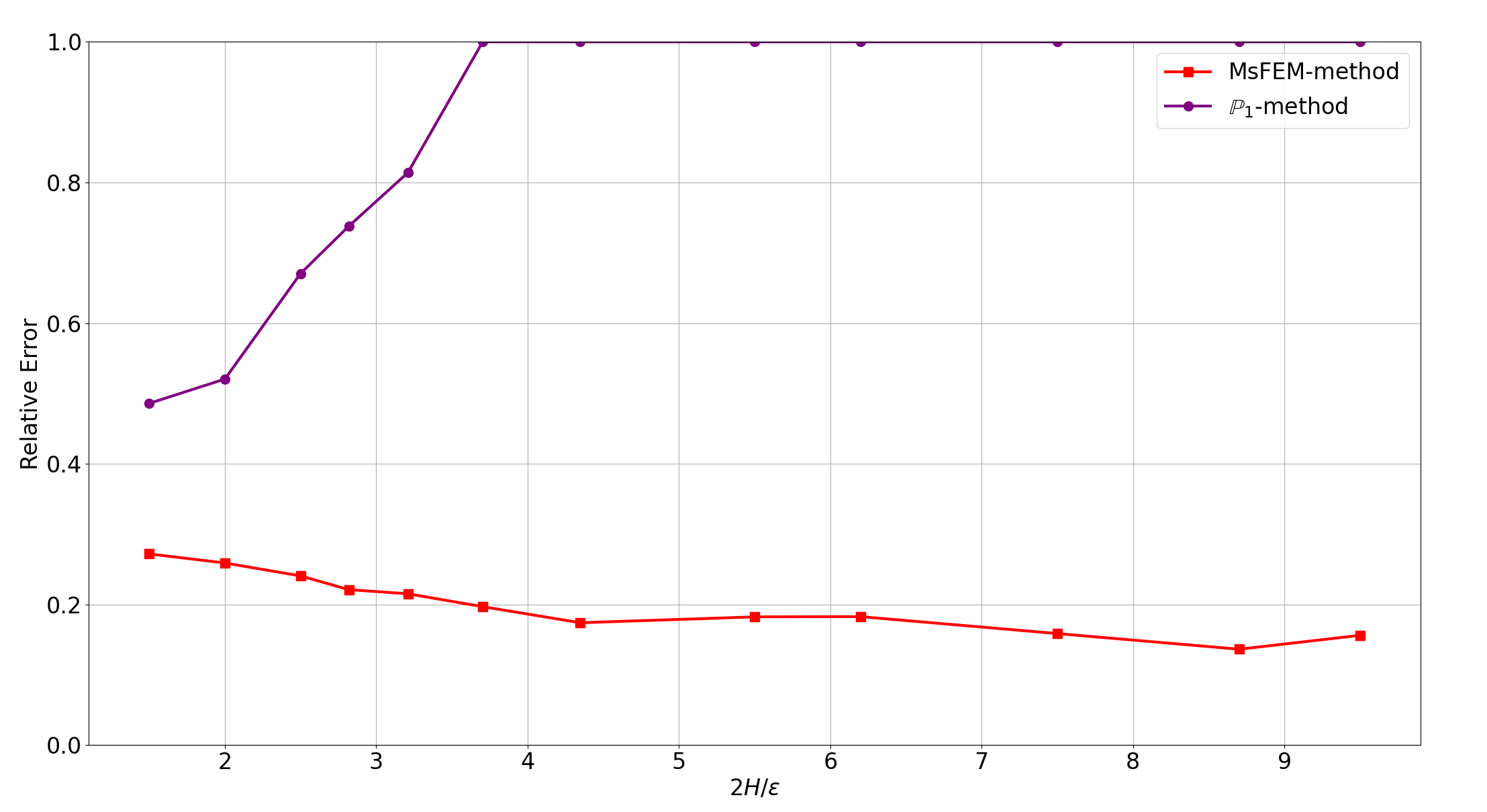}
  \caption{Quasi-periodic case~\eqref{coefficients quasi periodic 1 groupe_A}--\eqref{coefficients quasi periodic 1 groupe_S}: relative error~\eqref{erreur vecteur propre relative} on the eigenvector for the MsFEM method and the $\PP_1$-method, as a function of $\eps$ ($H = 1/8$ fixed).}
  \label{fig:ErreurH1_QP}
\end{figure}

\begin{figure}[htbp]
  \centering
  \includegraphics[width=0.9\textwidth]{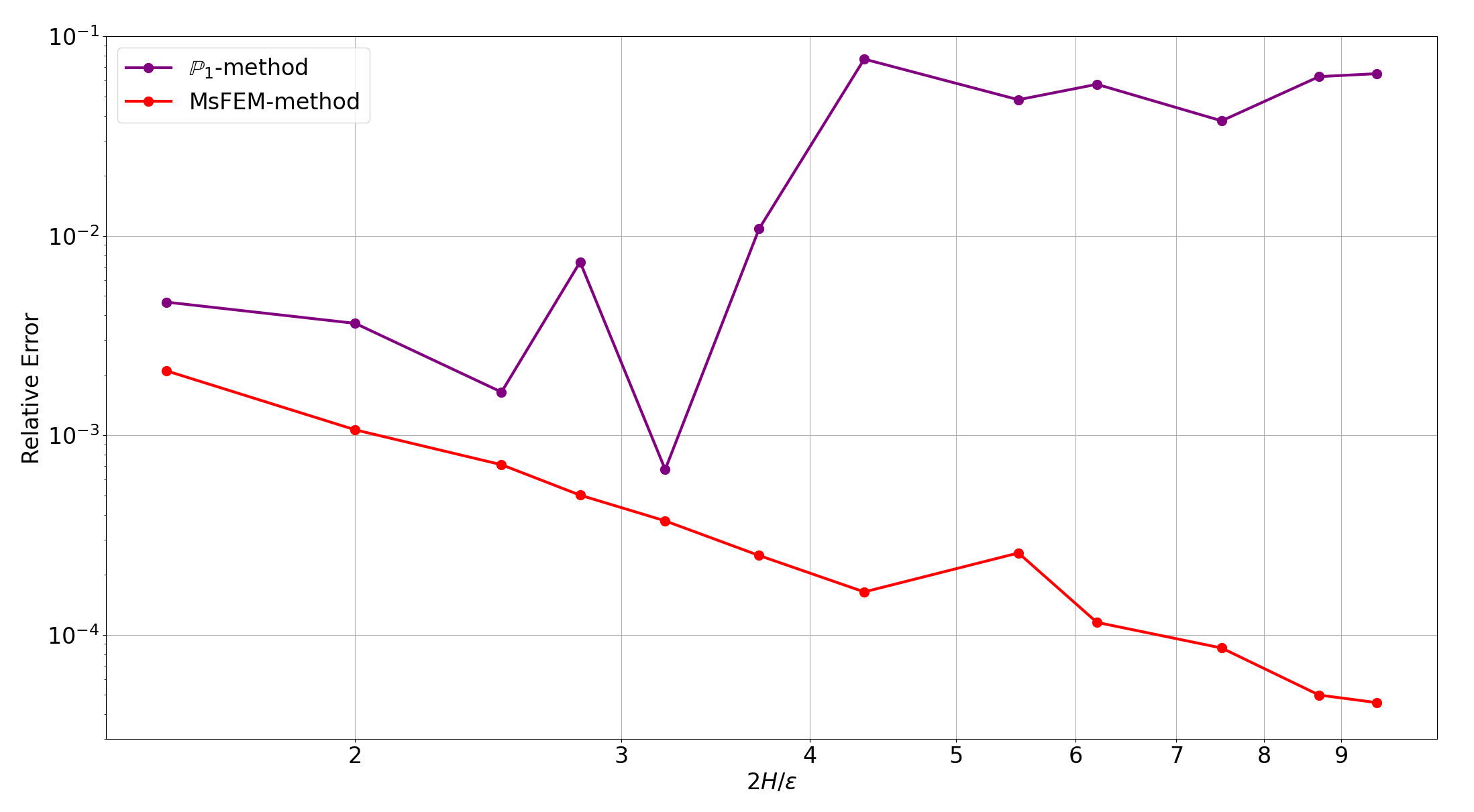}
  \caption{Quasi-periodic case~\eqref{coefficients quasi periodic 1 groupe_A}--\eqref{coefficients quasi periodic 1 groupe_S}: relative error~\eqref{erreur valeur propre relative} on the eigenvalue, for the MsFEM method and the $\PP_1$-method, as a function of $\eps$ ($H = 1/8$ fixed, $\operatorname{log}-\operatorname{log}$ scale).}
  \label{fig:ErreurVP_QP}
\end{figure}

\medskip

In practice, $\eps$ is typically small but fixed, while only the size of the coarse mesh $H$ can be adjusted. By decreasing $H$, one can reduce the error provided by the MsFEM method, both in the approximation of the first eigenvector and in the approximation of the corresponding eigenvalue. To illustrate this, we fix $\eps = 1/30$ and vary $H$, which yields the results in Table~\ref{tab:ErreurH1_VP_QP_H}.

For both the eigenvector and the eigenvalue approximations, we observe that the relative error decreases as the coarse mesh size $H$ decreases (while remaining smaller than 20\% for the eigenvector and smaller than 0.02\% for the eigenvalue). However, the accuracy deteriorates when the coarse mesh size $H$ becomes comparable to $\eps$. This loss of accuracy can be attributed to two main reasons:
\begin{itemize}
  \item A sufficient number of oscillations is required within the oversampling patch $S_K$ in order for the proxy $\widetilde{\psi}^\eps_K$, computed on $S_K$ using the filtering technique, to be reliable. The estimate~\eqref{estimee convergence vecteur propre}, established in the periodic case in a perturbative regime, indeed shows that $\eps$ must be small compared to the size of the computational domain for $\widetilde{\psi}^\eps_K$ to accurately approximate the function $\psi(\cdot/\eps)$. More details can be found in Section~\ref{sec:demo small perturbations}.
  \item The presence of the term $\eps \, \sqrt{\eps/H}$ in the estimate~\eqref{th taux de convergence methode triche} of the accuracy of the preliminary method. A loss of accuracy is numerically observed for that preliminary method when $H$ becomes of the same order as $\eps$ (results not shown), and it is thus not unexpected that the actual method shows the same behaviour when $H$ becomes comparable to $\eps$.
\end{itemize}

\begin{table}[htbp]
  \begin{center}
    \begin{tabular}{|c||c|c|c||c|}
      \hline
      Value of $H$ & 1/4 & 1/8 & 1/16 & 1/32 \\
      \hline
      Relative error~\eqref{erreur vecteur propre relative} & 0.220 & 0.126 & 0.112 & 0.233 \\
      Relative error~\eqref{erreur valeur propre relative} & $2.13 \times 10^{-4}$ & $5.85 \times 10^{-5}$ & $1.72 \times 10^{-6}$ & $8.70 \times 10^{-5}$ \\
      \hline
    \end{tabular}
    \caption{Quasi-periodic case~\eqref{coefficients quasi periodic 1 groupe_A}--\eqref{coefficients quasi periodic 1 groupe_S}: relative errors~\eqref{erreur vecteur propre relative} on the eigenvector and~\eqref{erreur valeur propre relative} on the eigenvalue for the MsFEM method, as a function of $H$ ($\eps = 1/30$ fixed). All values of $H$ are larger than $\eps$, except the smallest one, which is just below $\eps$ (thus the double vertical bar before the last column). When using the $\PP_1$-method, the relative error~\eqref{erreur vecteur propre relative} is equal to 1 (thus an error of 100\%) for the four values of $H$, while the relative error~\eqref{erreur valeur propre relative} is of the order of $4 \times 10^{-2}$ for the largest three values of $H$, and of the order of $10^{-2}$ for the smallest value of $H$: for all values of $H$, the $\PP_1$ method provides unreliable results.}
    \label{tab:ErreurH1_VP_QP_H}
  \end{center}
\end{table}

\section{The vector-valued case} \label{sec:MsFEM multigroup}

We now turn to the case where the problem of interest is vector-valued, and adopt the notations of~\cite{allaireHomogenizationSpectralProblem2000}. As pointed out in the introduction, we recall that this case is very relevant from the application viewpoint. From a mathematical viewpoint, the problem is in general non self-adjoint, and thus more challenging.

We thus consider~\eqref{pb reaction diffusion stationnaire} where $u^\eps$ is a vector of size $\kappa$, $\Sigma^\eps$ and $\sigma^\eps$ are $\kappa \times \kappa$ matrices, and $A^\eps$ is a fourth-order tensor. We assume that $A^\eps$ is block diagonal, so that
\begin{equation} \label{eq: A diagonal bloc}
A^\eps \nabla u^\eps = \big( A_1^\eps \nabla u^\eps_1, \hdots, A_\kappa^\eps \nabla u^\eps_\kappa \big)^T \in \RR^{\kappa \times d}, 
\end{equation}
where, for any $1 \leq k \leq \kappa$, $A_k^\eps$ is a symmetric $d \times d$ matrix, and $u^\eps_k$ is the $k$-th component of the vector $u^\eps$. Problem~\eqref{pb reaction diffusion stationnaire} hence reads: for any $1 \leq k \leq \kappa$,
$$
\sum_{\ell=1}^\kappa \Sigma^\eps_{k,\ell} \, u^\eps_\ell - \eps^2 \operatorname{div} \left(A^\eps_k \nabla u^\eps_k \right) = \lambda^\eps \sum_{\ell=1}^\kappa \sigma^\eps_{k,\ell} \, u^\eps_\ell \ \ \text{in $\Omega$}, \qquad u^\eps_k = 0 \ \ \text{on $\partial \Omega$}.
$$
The coupling between the components of $u^\eps$ comes from the fact that $\Sigma^\eps$ and/or $\sigma^\eps$ are non-diagonal matrices. Note also that $\Sigma^\eps$ and $\sigma^\eps$ are not supposed to be symmetric.

We assume $A^\eps$ to be bounded and coercive, in the sense that $A^\eps$ belongs to $L^\infty(\Omega)$ and satisfies the following bounds: there exists $\beta \geq \alpha > 0$ such that, for any $\eps$, any $1 \leq k \leq \kappa$ and any $\xi,\eta \in \RR^d$, we have
\begin{equation} \label{hypothese coercivite Aeps_vecteur}
  \alpha \, |\xi|^2 \leq \xi^T A^\eps_k(x) \, \xi \qquad \text{and} \qquad | \eta^T A^\eps_k(x) \, \xi | \leq \beta \, |\eta| \, |\xi| \qquad \text{a.e. in $\Omega$}. 
\end{equation}
We also assume $\Sigma^\eps$ and $\sigma^\eps$ to be bounded on $\Omega$.

\medskip

We first recall in Section~\ref{sec:homogenization results multigroup} below the periodic homogenization results established in~\cite{allaireHomogenizationSpectralProblem2000} on the vector-valued variant of~\eqref{pb reaction diffusion stationnaire}. Relaxing the periodicity assumption, we next present a MsFEM method in Section~\ref{sec:MsFEM multigroup description}, and collect the obtained numerical results in Section~\ref{sec:num_vector-valued}.

\subsection{Periodic homogenization results} \label{sec:homogenization results multigroup}

We assume here the coefficients to be periodic and hence to satisfy Assumption~\eqref{eq:period} for some $Y$-periodic coefficients $A$, $\Sigma$ and $\sigma$, which are furthermore assumed to belong to $L^\infty(Y)$. We assume $A$ to be block-diagonal in the sense of~\eqref{eq: A diagonal bloc}, and coercive in the sense of~\eqref{hypothese coercivite Aeps_vecteur}. Moreover, we suppose that
\begin{equation} \label{hypotheses sigma}
  \begin{aligned}
    &\forall 1 \leq k,\ell \leq \kappa, \ k \neq \ell, \qquad \Sigma_{k,k} > 0, \qquad \Sigma_{k,k-1} < 0, \qquad \Sigma_{k,\ell} \leq 0,
    \\[3pt]
    &\forall 1 \leq k,\ell \leq \kappa, \qquad \sigma_{1,\kappa} > 0, \qquad \sigma_{k,\ell} \geq 0,
    \\
    &\forall 1 \leq k \leq \kappa, \qquad \sum_{\ell=1}^\kappa \Sigma_{k,\ell} > 0.
  \end{aligned}
\end{equation}
The generalization to the vector-valued case of Theorems~\ref{thm:spectral} and~\ref{thm:spectral psi} can respectively be found in~\cite[Theorem~2.3 and Corollary~2.5]{allaireHomogenizationSpectralProblem2000}. We recall them in Theorems~\ref{thm:spectral multigroup} and~\ref{thm:spectral psi multigroup} below.

\begin{theorem}[Theorem~2.3 of~\cite{allaireHomogenizationSpectralProblem2000}] \label{thm:spectral multigroup}
  Under Assumptions~\eqref{eq:period}, \eqref{eq: A diagonal bloc}, \eqref{hypothese coercivite Aeps_vecteur} and~\eqref{hypotheses sigma}, Problem~\eqref{pb reaction diffusion stationnaire} admits a countable number of eigenvalues (possibly complex), with associated eigenvectors in $[H^1_0(\Omega)]^\kappa$. In addition, the first eigenvalue of~\eqref{pb reaction diffusion stationnaire} (i.e. the smallest in modulus) is real and simple, and the corresponding eigenfunction has real-valued components and can be chosen such that all its components are positive on $\Omega$.
\end{theorem}

We next define two cell eigenvectors $\psi$ and $\psi^\star$ (each of which is a vector of size $\kappa$): $\psi$ is the first eigenvector of the spectral problem
\begin{equation} \label{pb spectral 2 groupes}
  \Sigma \, \psi - \operatorname{div} \left(A \nabla \psi \right) = \lambda^\infty \, \sigma \, \psi \ \ \text{in $\RR^d$}, \qquad \text{$\psi$ is $Y$-periodic},
\end{equation}
while $\psi^\star$ is the first eigenvector of the adjoint spectral problem
\begin{equation} \label{pb spectral adjoint 2 groupes}
  \Sigma^\star \, \psi^\star - \operatorname{div} \left(A \nabla \psi^\star \right) = \lambda^\infty \, \sigma^\star \, \psi^\star \ \ \text{in $\RR^d$}, \qquad \text{$\psi^\star$ is $Y$-periodic},
\end{equation}
where $\Sigma^\star$ (resp. $\sigma^\star$) is the adjoint (or transpose) matrix of $\Sigma$ (resp. $\sigma$). Recall that $A$ is block-diagonal and that each $A_k$ is symmetric, so this part of the equation is self-adjoint.

\begin{theorem}[Corollary~2.5 of~\cite{allaireHomogenizationSpectralProblem2000}] \label{thm:spectral psi multigroup}
  Under Assumptions~\eqref{eq: A diagonal bloc}, \eqref{hypothese coercivite Aeps_vecteur} and~\eqref{hypotheses sigma}, Problems~\eqref{pb spectral 2 groupes} and~\eqref{pb spectral adjoint 2 groupes} admit a common first (i.e. smallest in modulus) eigenvalue $\lambda^\infty$. This first eigenvalue is real and simple, and the corresponding first eigenfunctions $\psi$ and $\psi^\star$ have all their components in $H^1_{\rm per}(Y)$ and real-valued. In addition, $\psi$ and $\psi^\star$ can be chosen such that all their components are positive on $Y$.
\end{theorem}

Denoting by $\psi_k$ (resp. $\psi_k^\star$) the $k$-th component of the vector $\psi$ (resp. $\psi^\star$), and introducing the vector-valued function $v^\eps$ defined by
$$
v^\eps = \left( v^\eps_k \right)_{1 \leq k \leq \kappa} \quad \text{with} \quad v^\eps_k(x) = \frac{u^\eps_k(x)}{\psi_k(x/\eps)},
$$
it is shown in~\cite{allaireHomogenizationSpectralProblem2000} that
\begin{equation} \label{pb reaction diffusion stationnaire v}
  \left\{
  \begin{aligned}
    - \operatorname{div}\big( D^\eps(\psi_\eps,\psi_\eps^\star) \, \nabla v^\eps \big) + \frac{1}{\eps^2} \, Q^\eps(\lambda^\infty,\psi_\eps,\psi^\star_\eps)(v^\eps) &= \nu^\eps \, B^\eps(\psi_\eps,\psi^\star_\eps) \, v^\eps \ \ \text{in $\Omega$},
    \\
    v^\eps &= 0 \ \ \text{on $\partial \Omega$},
  \end{aligned}
  \right.
\end{equation}
where we have used the notation $\psi_\eps := \psi(\cdot/\eps)$ and $\psi_\eps^\star := \psi^\star(\cdot/\eps)$. The quantities $D^\eps$, $Q^\eps$ and $B^\eps$ are defined as follows: for any generic functions $u$, $v$ and $w$ valued in $\RR^\kappa$ (we denote e.g. by $u_k$ the $k$-th component of $u$), $B^\eps(u,w)$ is a $\kappa \times \kappa$ matrix defined as
$$
\big( B^\eps(u,w) \big)_{k,\ell} = \sigma^\eps_{k,\ell} \, u_\ell \, w_k.
$$
The diffusion coefficient $D^\eps(u,w)$ is a block diagonal fourth-order tensor, so that $D^\eps(u,w) \nabla v$ is a matrix of size $\kappa \times d$ defined by
\begin{equation*} 
  D^\eps(u,w) \nabla v = \big( D^\eps_1(u_1,w_1) \nabla v_1, \hdots, D^\eps_\kappa(u_\kappa,w_\kappa) \nabla v_\kappa \big)^T \in \RR^{\kappa \times d},
\end{equation*}
where $D^\eps_k(u_k,w_k)$ is the $d \times d$ matrix defined by $D^\eps_k(u_k,w_k) = u_k \, w_k \, A_k^\eps$, with $A_k^\eps = A_k(\cdot/\eps)$ the $d \times d$ matrix appearing in~\eqref{eq: A diagonal bloc}. As in~\eqref{pb reaction diffusion stationnaire}, the divergence of $D^\eps(u,w) \nabla v$ is taken line by line, which means that $\operatorname{div}\big( D^\eps(u,w) \, \nabla v \big)$ is a vector, the $k$-th component of which is $\operatorname{div}\big( D^\eps_k(u_k,w_k) \, \nabla v_k \big)$.

For any $\lambda \in \RR$, the vector $Q^\eps(\lambda,u,w)(v) \in \RR^\kappa$ is defined by
$$
Q^\eps(\lambda,u,w)(v) = \eps^2 \, J^\eps(u,w) \cdot \nabla v + \widetilde{Q}^\eps(\lambda,u,w) \, v,
$$
where $\widetilde{Q}^\eps(\lambda,u,w)$ is a $\kappa \times \kappa$ matrix defined by
$$
\left( \widetilde{Q}^\eps(\lambda,u,w) \right)_{k,\ell} = \begin{cases}
  \left( \Sigma^\eps_{k,\ell} - \lambda \, \sigma^\eps_{k,\ell} \right) u_\ell \, w_k & \text{if $k \neq \ell$},
  \\[3pt]
  \dis - \sum^\kappa_{\substack{\ell' = 1 \\ \ell' \neq k}} \left( \widetilde{Q}^\eps(\lambda,u,w) \right)_{k,\ell'} & \text{if $k = \ell$},
\end{cases}
$$
and where $J^\eps(u,w)$ is a $\kappa \times d$ matrix, with rows $\big(J^\eps(u,w)\big)_k$ defined by
$$
\big(J^\eps(u,w)\big)_k = A^\eps_k \left( u_k \nabla w_k - w_k \nabla u_k \right) \in \RR^d.
$$
The product $J^\eps(u,w) \cdot \nabla v$ is defined by
$$
J^\eps(u,w) \cdot \nabla v = \begin{pmatrix} \big(J^\eps(u,w)\big)_1 \cdot \nabla v_1 \\ \vdots \\ \big((J^\eps(u,w)\big)_\kappa \cdot \nabla v_\kappa \end{pmatrix}.
$$
The homogenization limit of~\eqref{pb reaction diffusion stationnaire v} is established in~\cite{allaireHomogenizationSpectralProblem2000}. If the symmetry condition
\begin{equation} \label{symmetry condition}
  \sum_{k=1}^\kappa \int_Y A_k \left( \psi_k \nabla \psi_k^\star - \psi_k^\star \nabla \psi_k \right) = 0
\end{equation}
is satisfied, then each component $v^\eps_k$ weakly converges in $H^1_0(\Omega)$ to $v^\star$, the first eigenvector of the homogenized problem
\begin{equation*} 
  - \operatorname{div} \left( D^\star \nabla v^\star \right) = \nu^\star \, \sigma^\star \, v^\star \ \ \text{in $\Omega$}, \qquad v^\star = 0 \ \ \text{on $\partial \Omega$},
\end{equation*}
where the constant real number $\sigma^\star$ and the constant matrix $D^\star \in \RR^{d \times d}$ are defined in~\cite[Theorem~3.2]{allaireHomogenizationSpectralProblem2000}. The eigenvector $v^\star$ is a scalar-valued function, and all components of $v^\eps$ thus converge to the same limit.

\begin{rmq}
  It is mentioned in~\cite{allaireHomogenizationSpectralProblem2000} that the symmetry condition~\eqref{symmetry condition} is satisfied if $A$, $\Sigma$ and $\sigma$ (which are defined on the periodic cell $Y$) all have a cubic symmetry (in the case when $Y = (-1/2,1/2)^d$, this means that $A$, $\Sigma$ and $\sigma$ are all even functions with respect to $y_i$, for any $1 \leq i \leq d$).
\end{rmq}

\subsection{Description of the actual method} \label{sec:MsFEM multigroup description}

We now describe the MsFEM method we propose for the vector variant of the problem. This method is based on the preliminary method described in Section~\ref{sec:MsFEM preliminaire}, on the filtering ideas exposed in Section~\ref{sec:approx_psi}, and it is also inspired by the homogenization results recalled in Section~\ref{sec:homogenization results multigroup}. For each element $K$ of the coarse mesh $\mathcal{T}_H$, we construct a square-shaped oversampling patch $S_K$ around the element $K$, as shown in Figure~\ref{fig:patch_oversampling_comparison}.

\medskip

We proceed as follows:
\begin{description}
\item[$\bullet$ Step~1 (offline):] Computation of a proxy.
\end{description}
For each element $K$ of the coarse mesh $\mathcal{T}_H$, we consider the filter $\tau_K$ defined on $S_K$ by~\eqref{hypotheses filtre varphi}. We then compute, on $S_K$, the eigenvector $\widetilde{\psi}^\eps \in [H^1(\Omega)]^\kappa$ associated with the smallest eigenvalue $\widetilde{\lambda}^\eps \in \RR$ and with the Lagrange multipliers $\widetilde{\mu}^\eps_k \in \RR^d$ (for any $1 \leq k \leq \kappa$) such that, for any $v \in [H^1(\Omega)]^\kappa$ and any $\mu_k \in \RR^d$ (with $1 \leq k \leq \kappa$),
\begin{multline} \label{pb patch oversampling filtre 2 groupes_1}
  \eps^2 \sum_{1 \leq k \leq \kappa} \int_{S_K} \tau_K \, (\nabla v_k)^T A^\eps_k \nabla \widetilde{\psi}_k^\eps + \int_{S_K} \tau_K \, v^T \Sigma^\eps \, \widetilde{\psi}^\eps
  \\
  = \widetilde{\lambda}^\eps \int_{S_K} \tau_K \, v^T \sigma^\eps \, \widetilde{\psi}^\eps + \sum_{1 \leq k \leq \kappa} \widetilde{\mu}^\eps_k \cdot \int_{S_K} \tau_K \nabla v_k,
\end{multline}
and
\begin{equation} \label{pb patch oversampling filtre 2 groupes_2}
  \sum_{1 \leq k \leq \kappa} \mu_k \cdot \int_{S_K} \tau_K \nabla \widetilde{\psi}_k^\eps = 0.
\end{equation}
On each element $K \in \mathcal{T}_H$, we then define $\widetilde{\psi}^\eps_K := \widetilde{\psi}^\eps|_K$ as a proxy of the function $\psi(\cdot/\eps)$.

\medskip

We proceed similarly (replacing $\Sigma$ and $\sigma$ by their transpose) to define a proxy $\widetilde{\psi}^{\eps,\star}_K$ of the function $\psi^\star(\cdot/\eps)$ on $K$. We numerically observe that that problem and~\eqref{pb patch oversampling filtre 2 groupes_1}--\eqref{pb patch oversampling filtre 2 groupes_2} share the same first eigenvalue $\widetilde{\lambda}^\eps$, which is used as a proxy for $\lambda^\infty$ on $K$, and that we denote $\widetilde{\lambda}^\eps_K$ hereafter.

\begin{description}
\item[$\bullet$ Step~2 (offline):] Computation of the MsFEM-lin basis functions.
\end{description}

We follow the same ideas as for the scalar variant of the problem, and use the functions $\widetilde{\psi}_K^\eps$ and $\widetilde{\psi}_K^{\eps,\star}$ and the scalar $\widetilde{\lambda}^\eps_K$ as proxies for $\psi$, $\psi^\star$ and $\lambda^\infty$ on $K$ in~\eqref{pb reaction diffusion stationnaire v}. Restricting ourselves momentarily to the case $\kappa=2$ for the sake of simplicity, we define, for any $1 \leq i \leq N$ (we recall that $N$ is the number of internal vertices of the coarse mesh $\mathcal{T}_H$), the function $\chi_i^\eps$ as the solution in $[ H^1_0(\Omega) ]^2$ to the following problem: for any $K \in \mathcal{T}_H$,
\begin{equation} \label{chi phi_vector}
  \left\{
  \begin{aligned}
    - \operatorname{div} \left[ D^\eps\left(\widetilde{\psi}_K^\eps,\widetilde{\psi}_K^{\eps,\star}\right) \nabla \chi_i^\eps \right] + \frac{1}{\eps^2} \, Q^\eps\left(\widetilde{\lambda}^\eps_K,\widetilde{\psi}_K^\eps,\widetilde{\psi}_K^{\eps,\star}\right)(\chi_i^\eps) &= 0 && \text{in $K$},
    \\
    \chi_i^\eps &= \begin{pmatrix} \chi_i^{\PP_1} \\ \chi_i^{\PP_1} \end{pmatrix} && \text{on $\partial K$}.
  \end{aligned}
  \right.
\end{equation}
We next define the MsFEM basis functions $\left\{ \phi_{i,1}^\eps \right\}_{1 \leq i \leq N}$ and $\left\{ \phi_{i,2}^\eps \right\}_{1 \leq i \leq N}$ on $\Omega$ as follows: for any $K \in \mathcal{T}_H$,
\begin{equation} \label{eq:def_phi_actual_vector}
  \phi_{i,1}^\eps = \begin{pmatrix} (\chi_i^\eps)_1 \, (\widetilde{\psi}_K^\eps)_1 \\ 0 \end{pmatrix} \quad \text{and} \quad \phi_{i,2}^\eps = \begin{pmatrix} 0 \\ (\chi_i^\eps)_2 \, (\widetilde{\psi}_K^\eps)_2 \end{pmatrix} \ \ \text{in $K$},
\end{equation}
where $(\widetilde{\psi}_K^\eps)_k$ is the $k$-th component of the vector $\widetilde{\psi}_K^\eps \in \RR^2$ (and likewise for $(\chi_i^\eps)_k$). We finally consider the approximation space
\begin{equation} \label{MsFEM espace 2 groupe}
  V_{\eps,H} = \spn \left\{ \phi_{i,1}^\eps, \ \ \phi_{i,2}^\eps, \quad 1 \leq i \leq N \right\}.
\end{equation}

\begin{rmq}
  Alternatively, one can define the MsFEM-lin functions $\chi_i^{\eps,j}$, for $j=1,2$, as solutions in $[ H^1_0(\Omega) ]^2$ to the same PDE: for any $j=1,2$, for any $K \in \mathcal{T}_H$,
  \begin{equation*}
    - \operatorname{div} \left[ D^\eps\left(\widetilde{\psi}_K^\eps,\widetilde{\psi}_K^{\eps,\star}\right) \nabla \chi_i^{\eps,j} \right] + \frac{1}{\eps^2} \, Q^\eps\left(\widetilde{\lambda}^\eps_K, \widetilde{\psi}_K^\eps,\widetilde{\psi}_K^{\eps,\star}\right)(\chi_i^{\eps,j}) = 0 \quad \text{in $K$},
  \end{equation*}
  with the boundary conditions $\dis \chi_i^{\eps,1} = \begin{pmatrix} \chi_i^{\PP_1} \\ 0 \end{pmatrix}$ and $\dis \chi_i^{\eps,2} = \begin{pmatrix} 0 \\ \chi_i^{\PP_1} \end{pmatrix}$ on $\partial K$. We next define the basis functions $\left\{ \phi_{i,j}^\eps \right\}_{1 \leq i \leq N}$ (for any $j=1,2$) on $\Omega$ as
  \begin{equation*}
    \forall K \in \mathcal{T}_H, \quad \phi_{i,j}^\eps = \begin{pmatrix} (\chi_i^{\eps,j})_1 \, (\widetilde{\psi}_K^\eps)_1 \\ (\chi_i^{\eps,j})_2 \, (\widetilde{\psi}_K^\eps)_2 \end{pmatrix} \ \ \text{in $K$},
  \end{equation*}
  and then introduce the approximation space defined by~\eqref{MsFEM espace 2 groupe}. In practice, we have observed that these two methods yield very similar numerical results. We hence decided to retain the method described above, which is less computationally expensive than the method described in this remark.
\end{rmq}

The extension to the cases $\kappa \geq 2$ is straightforward.

\begin{description}
\item[$\bullet$ Step~3 (online):] Solution to the global problem.
\end{description}
We perform a Galerkin approximation of~\eqref{pb reaction diffusion stationnaire} on the space~\eqref{MsFEM espace 2 groupe} and look for the first eigencouple $(\lambda_H^\eps,u^\eps_H) \in \RR \times V_{\eps,H}$ such that, for any $w \in V_{\eps,H}$,
\begin{equation} \label{pb reaction diffusion stationnaire MsFEM ueps 2 groupes}
  \eps^2 \sum_{1 \leq k \leq \kappa} \sum_{K \in \mathcal{T}_H} \int_K (\nabla w_k)^T A^\eps_k \nabla [u^\eps_H]_k + \int_\Omega w^T \Sigma^\eps \, u^\eps_H = \lambda^\eps_H \int_\Omega w^T \sigma^\eps \, u^\eps_H.
\end{equation}
The first integral is written as a broken sum since the approach is non-conforming.

\subsection{Numerical results} \label{sec:num_vector-valued}

We now present several numerical experiments in the case $\kappa = 2$ which confirm the efficiency of the MsFEM approach we propose, and demonstrate that our method is not restricted to scalar-valued problems. As in Section~\ref{sec:num_one_group}, all the computations have been performed with FreeFEM~\cite{hechtNewDevelopmentFreeFem2012} and the associated scripts are available at~\cite{code_alberic}.

We consider here a test case (frequently used in practice for neutronics applications) which corresponds to the following problem:
\begin{equation} \label{pb reaction diffusion stationnaire model 2 groupes}
  \left\{
  \begin{aligned}
    -\eps^2 \text{div}\left(A_1^\eps \nabla u_1^\eps \right) + \Sigma_{11}^\eps \, u_1^\eps &= \lambda^\eps \left( \sigma_{11}^\eps \, u_1^\eps + \sigma_{12}^\eps \, u_2^\eps \right) && \text{in $\Omega$},
    \\
    -\eps^2 \text{div}\left(A_2^\eps \nabla u_2^\eps \right) + \Sigma_{22}^\eps \, u_2^\eps &= - \Sigma_{21}^\eps \, u_1^\eps && \text{in $\Omega$},
  \end{aligned}
  \right.
\end{equation}
with the boundary conditions $u^\eps_1 = u^\eps_2 = 0$ on $\partial\Omega$. This corresponds to setting $\Sigma_{12}^\eps = 0$ and $\sigma_{21}^\eps = \sigma_{22}^\eps = 0$ in~\eqref{pb reaction diffusion stationnaire}. In the periodic case, the values of the coefficients $\Sigma^\eps$ and $\sigma^\eps$ are chosen such that Assumption~\eqref{hypotheses sigma} is satisfied.

\begin{rmq}
In~\cite[Chapter~3, Section~3.3.3]{these_lefort}, a case with $\sigma^\eps = \text{Id}_2$ is also considered. The hypothesis~\eqref{hypotheses sigma} is not satisfied (since $\sigma = \text{Id}_2$ does not satisfy the condition $\sigma_{1,\kappa} > 0$), and hence the homogenization results of Section~\ref{sec:homogenization results multigroup} do not apply. However, we numerically observe that the assertion of Theorem~\ref{thm:spectral multigroup} concerning the first eigenvalue of~\eqref{pb reaction diffusion stationnaire} is still satisfied. We can still apply the MsFEM approach described in Section~\ref{sec:MsFEM multigroup description}. The conclusions we draw from numerical experiments in that case $\sigma^\eps = \text{Id}_2$ are qualitatively similar to those obtained in the case discussed here.
\end{rmq}

In the setting of~\eqref{pb reaction diffusion stationnaire model 2 groupes}, we have considered three types of coefficients: periodic and symmetric coefficients (for which~\eqref{symmetry condition} is satisfied), periodic and non-symmetric coefficients (for which we have observed~\eqref{symmetry condition} not to be satisfied) and quasi-periodic coefficients. To show the robustness of our approach, we focus here on the latter two types and refer to~\cite[Chapter~3, Section~3.3.3]{these_lefort} for a discussion of the results obtained with the first type of coefficients.

We proceed as in the scalar version of the problem. We thus again consider a two-dimensional setting, take $\Omega = (0,1)^2$, and consider a coarse triangular mesh $\mathcal{T}_H$ with $H = 1/8$. Each element of this coarse mesh is itself meshed with a fine triangular mesh $\mathcal{T}_h$ with $h \ll \eps $, as shown in Figure~\ref{fig:mesh}. A reference solution is computed as the first eigencouple $(\lambda^\eps,u^\eps)$ of~\eqref{pb reaction diffusion stationnaire} on the fine mesh $\mathcal{T}_h$ using a $\PP_1$-finite element method. We then compute the MsFEM first eigencouple $(\lambda^\eps_H,u^\eps_H)$, solution to~\eqref{pb reaction diffusion stationnaire MsFEM ueps 2 groupes}. Throughout these numerical tests, and as in the scalar-valued variant, we fix the oversampling ratio at $\rho = 2$ and we consider the filter function $\tau_0(x) = C \, x^2 \, (1-x)^2$, for some constant $C$ such that $\| \tau_0 \|_{L^1(0,1)} = 1$ (this corresponds to a filter of order $k=2$ in the sense of~\eqref{hypotheses filtre varphi0}).

For the sake of comparison, we also compute the first eigencouple obtained by the classical $\PP_1$-method on the coarse mesh $\mathcal{T}_H$. We denote generically by $(\lambda^\eps_H, u^\eps_H)$ the first eigencouple obtained by the MsFEM-method or the $\PP_1$-method on the coarse mesh $\mathcal{T}_H$. Since the two types of coefficients considered here are not covered by homogenization theory, we do not consider any preliminary type method.

The relative error on the eigenvalue is defined as in the scalar-valued case by~\eqref{erreur valeur propre relative}. Denoting $u^\eps = \begin{pmatrix} u^\eps_1 \\ u^\eps_2 \end{pmatrix}$ and $u^\eps_H = \begin{pmatrix} u^\eps_{H,1} \\ u^\eps_{H,2} \end{pmatrix}$ the eigenvectors, and recalling that the approximation space~\eqref{MsFEM espace 2 groupe} is non-conforming, the error on the eigenvector is defined as
\begin{equation} \label{erreur vecteur propre relative multigroup}
  \frac{1}{\sqrt{2}} \sqrt{ \frac{\sum_{K \in \mathcal{T}_H} \| u^\eps_1 - u^\eps_{H,1} \|^2_{H^1(K)}}{\| u^\eps_1 \|^2_{H^1(\Omega)}} + \frac{\sum_{K \in \mathcal{T}_H} \| u^\eps_2 - u^\eps_{H,2} \|^2_{H^1(K)}}{\| u^\eps_2\|^2_{H^1(\Omega)}}},
\end{equation}
so that each component contributes equally to the error estimate. The prefactor $1/\sqrt{2}$ ensures a fair comparison of the relative errors with the scalar-valued case.

\medskip

In the definitions of the coefficients that are given below, we recall that $\text{Id}_2$ is the identity matrix in dimension 2 and that
$$
\Sigma^\eps = \begin{pmatrix} \Sigma^\eps_{11} & \Sigma^\eps_{12} \\ \Sigma^\eps_{21} & \Sigma^\eps_{22} \end{pmatrix} = \begin{pmatrix} \Sigma^\eps_{11} & 0 \\ \Sigma^\eps_{21} & \Sigma^\eps_{22} \end{pmatrix},
\qquad
\sigma^\eps = \begin{pmatrix} \sigma^\eps_{11} & \sigma^\eps_{12} \\ \sigma^\eps_{21} & \sigma^\eps_{22} \end{pmatrix} = \begin{pmatrix} \sigma^\eps_{11} & \sigma^\eps_{12} \\ 0 & 0 \end{pmatrix}.
$$

\subsubsection{Periodic and non-symmetric coefficients}

We define the diffusion and reaction coefficients in the cell $Y$ as follows: for any $(y_1,y_2) \in Y$,
\begin{equation} \label{coefficients periodic nonsymmetric 2 groupe second}
  \begin{aligned}
    A_1(y_1,y_2) &= \big(10 + \cos(2\pi y_1) \sin(2\pi y_2) \big) \, \text{Id}_2,
    \\[3pt]
    A_2(y_1,y_2) &= \big(6 + 2 \sin(2\pi(y_1 + 2y_2)) \cos(2\pi(y_1 - y_2)) \big) \, \text{Id}_2,
    \\[3pt]
    \Sigma_{11}(y_1,y_2) &= 11 + 10 \cos(2\pi y_1) \sin(2\pi y_2),
    \\[3pt]
    \Sigma_{21}(y_1,y_2) &= -5 \cos^2\left(\pi y_1 + \frac{\pi}{4}\right) - 1,
    \\[3pt]
    \Sigma_{22}(y_1,y_2) &= 8 \left( 5 + \cos^2(2\pi(y_1 - y_2)) \right),
    \\[3pt]
    \sigma_{11}(y_1,y_2) &= 1 + \frac{1}{2} \cos(2\pi(y_1 + y_2)) \sin(2 \pi y_2),
    \\[3pt]
    \sigma_{12}(y_1,y_2) &= \frac{1}{10} \big(4 + 3 \sin(2 \pi y_1) \cos(2 \pi y_2) \big).
  \end{aligned}
\end{equation}
We observe numerically that Assumption~\eqref{symmetry condition} is indeed not satisfied. The corresponding oscillatory coefficients are defined on $\Omega$ by~\eqref{eq:period}.

\medskip

The results shown on Figures~\ref{fig:compare_triche_MsFEM_coeffs_nonsymetriques_avec_coeff_absorbtion_erreur_H1_a} and~\ref{fig:compare_triche_MsFEM_coeffs_nonsymetriques_avec_coeff_absorbtion_erreur_H1_b} are obtained by fixing $H$ (at its value $H=1/8$) and varying $\eps$. These results confirm those of the scalar-valued variant: our MsFEM approach is robust with respect to the value of the characteristic size $\eps$ of the small scales, and yields results with an accuracy of the order of 20\% for the eigenvector, and between 1\% and 0.1\% for the eigenvalue. The accuracy here is thus comparable to that obtained for scalar-valued periodic cases (see Figures~\ref{fig:ErreurH1_Penche_avec_filtre_filtre_puissance_1} and~\ref{fig:ErreurVP_Penche_avec_filtre_filtre_puissance_1} of Section~\ref{sec:num_periodic}), even though the problem is more complicated (it is vector-valued and not self-adjoint) and the microstructure is also more complex (since it does not satisfy the assumptions under which homogenization results are established).

\begin{figure}[htbp]
  \centering
  \includegraphics[width=0.9\textwidth]{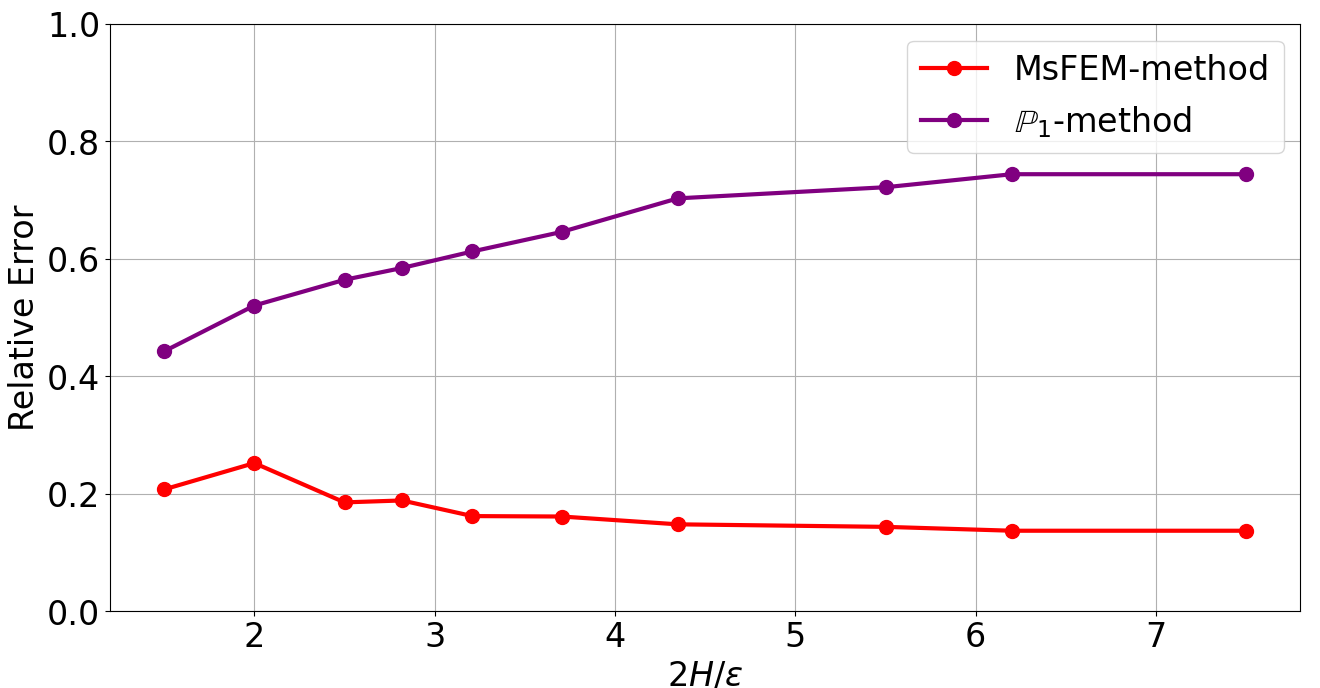}
  \caption{Periodic and non-symmetric case~\eqref{coefficients periodic nonsymmetric 2 groupe second}: relative error~\eqref{erreur vecteur propre relative multigroup} on the eigenvector for the MsFEM method and the $\PP_1$-method, as a function of $\eps$ ($H = 1/8$ fixed; note again that the abscissa, here and in many figures below, is $2H/\eps$).}
  \label{fig:compare_triche_MsFEM_coeffs_nonsymetriques_avec_coeff_absorbtion_erreur_H1_a}
\end{figure}

\begin{figure}[htbp]
  \centering
  \includegraphics[width=0.9\textwidth]{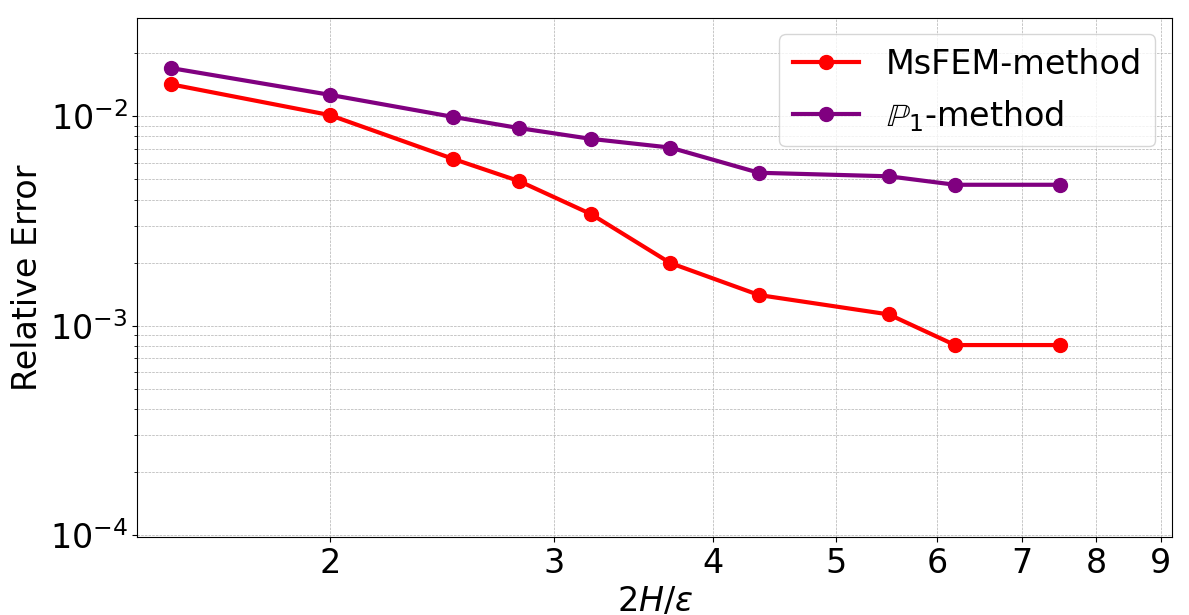}
  \caption{Periodic and non-symmetric case~\eqref{coefficients periodic nonsymmetric 2 groupe second}: relative error~\eqref{erreur valeur propre relative} on the eigenvalue for the MsFEM method and the $\PP_1$-method, as a function of $\eps$ ($H = 1/8$ fixed, $\operatorname{log}-\operatorname{log}$ scale).}
  \label{fig:compare_triche_MsFEM_coeffs_nonsymetriques_avec_coeff_absorbtion_erreur_H1_b}
\end{figure}

\subsubsection{Quasi-periodic coefficients}

We now consider the diffusion and reaction coefficients defined on $\Omega$ as follows: for any $(x_1,x_2) \in \Omega$,
{\footnotesize
\begin{equation} \label{coefficients quasi periodic 2 groupe second_a}
  \begin{aligned}
    A_1^\eps(x_1,x_2) &= \left[ 10 + \frac{1}{4} \left( \cos\left(\frac{2\pi}{\eps}x_1\right) + \cos\left(\frac{2\sqrt{2}\pi}{\eps}x_1\right) \right) 
      \left( \sin\left(\frac{2\pi}{\eps}x_2\right) + \sin\left(\frac{2\sqrt{2}\pi}{\eps}x_2\right) \right) \right] \text{Id}_2,
    \\[3pt]
    A_2^\eps(x_1,x_2) &= \left[ 1 + \frac{1}{10} \left( \cos\left(\frac{2\pi}{\eps}x_1\right) + \cos\left(\frac{\sqrt{2}\pi}{\eps}x_1\right) \right) 
      \left( \sin\left(\frac{2\pi}{\eps}x_2\right) + \sin\left(\frac{\sqrt{2}\pi}{\eps}x_2\right) \right) \right] \text{Id}_2,
  \end{aligned} 
\end{equation}
}
and
{\footnotesize
\begin{equation} \label{coefficients quasi periodic 2 groupe second_b}
  \begin{aligned}
    \Sigma^\eps_{11}(x_1,x_2) &= 9 + 2 \left[ \cos\left(\frac{2\pi}{\eps}x_1\right) + \cos\left(\frac{2\sqrt{2}\pi}{\eps}x_1\right) \right] 
    \left[ \sin\left(\frac{2\pi}{\eps}x_2\right) + \sin\left(\frac{2\sqrt{2}\pi}{\eps}x_2\right) \right],
    \\[3pt]
    \Sigma^\eps_{21}(x_1,x_2) &= - 9 -2 \left[ \cos\left(\frac{2\pi}{\eps}x_1\right) + \cos\left(\frac{2\sqrt{3}\pi}{\eps}x_1\right) \right]
    \left[ \sin\left(\frac{2\pi}{\eps}x_2\right) + \sin\left(\frac{2\sqrt{3}\pi}{\eps}x_2\right) \right],
    \\[3pt]
    \Sigma^\eps_{22}(x_1,x_2) &= 36 + 8 \left[ \cos\left(\frac{2\pi}{\eps}x_1\right) + \cos\left(\frac{1.5\sqrt{5}\pi}{\eps}x_1\right) \right]
    \left[ \sin\left(\frac{2\pi}{\eps}x_2\right) + \sin\left(\frac{1.5\sqrt{5}\pi}{\eps}x_2\right) \right],
    \\[3pt]
    \sigma^\eps_{11}(x_1,x_2) &= \frac{5}{2} + \frac{1}{8} \left[ \cos\left(\frac{2\pi}{\eps}x_1\right) + \cos\left(\frac{\sqrt{3}\pi}{\eps}x_1\right) \right] 
    \left[ \sin\left(\frac{2\pi}{\eps}x_2\right) + \sin\left(\frac{\sqrt{3}\pi}{\eps}x_2\right) \right],
    \\[3pt]
    \sigma^\eps_{12}(x_1,x_2) &= \frac{3}{2} + \frac{1}{4} \left[ \cos\left(\frac{2\pi}{\eps}x_1\right) + \cos\left(\frac{2\sqrt{2}\pi}{\eps}x_1\right) \right]
    \left[ \sin\left(\frac{2\pi}{\eps}x_2\right) + \sin\left(\frac{2\sqrt{2}\pi}{\eps}x_2\right) \right].
  \end{aligned} 
\end{equation}
}
These coefficients are obviously quasi-periodic. Fixing $H=1/8$ and varying $\eps$, we obtain the excellent results shown on Figures~\ref{fig:ErreurH1_QP_vector} and~\ref{fig:ErreurVP_QP_vector}: the error on the eigenvector remains smaller than 20\%, and that on the eigenvalue remains smaller than 0.1\% (and decreases down to 0.01\% for the largest values of the ratio $2H/\eps$). The comparison with the results of Figures~\ref{fig:ErreurH1_QP} and~\ref{fig:ErreurVP_QP} confirms the robustness of our approach.

\begin{figure}[htbp]
  \centering
  \includegraphics[width=0.9\textwidth]{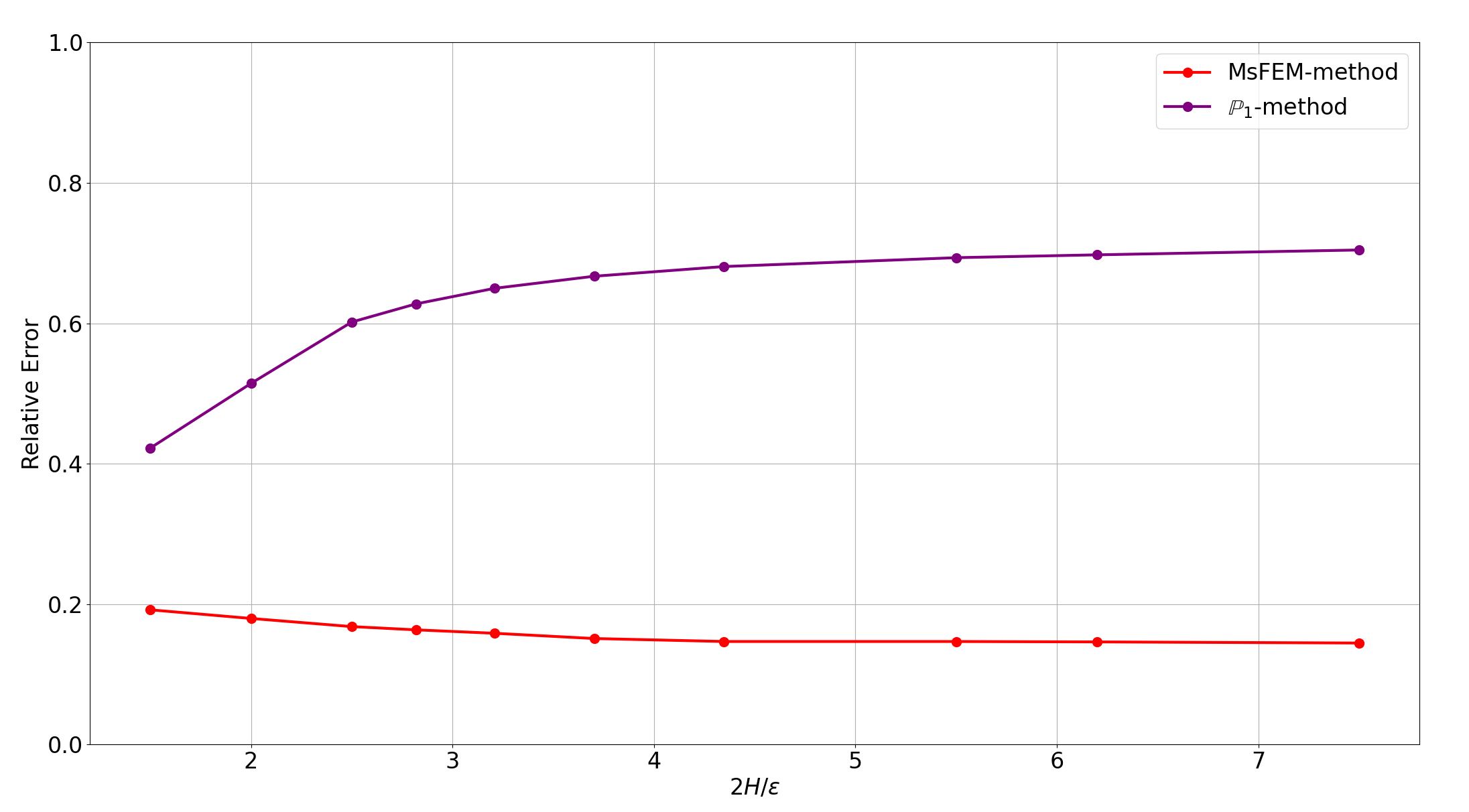}
  \caption{Quasi-periodic case~\eqref{coefficients quasi periodic 2 groupe second_a}--\eqref{coefficients quasi periodic 2 groupe second_b}: relative error~\eqref{erreur vecteur propre relative multigroup} on the eigenvector for the MsFEM method and the $\PP_1$-method, as a function of $\eps$ ($H = 1/8$ fixed).}
  \label{fig:ErreurH1_QP_vector}
\end{figure}

\begin{figure}[htbp]
  \centering
  \includegraphics[width=0.9\textwidth]{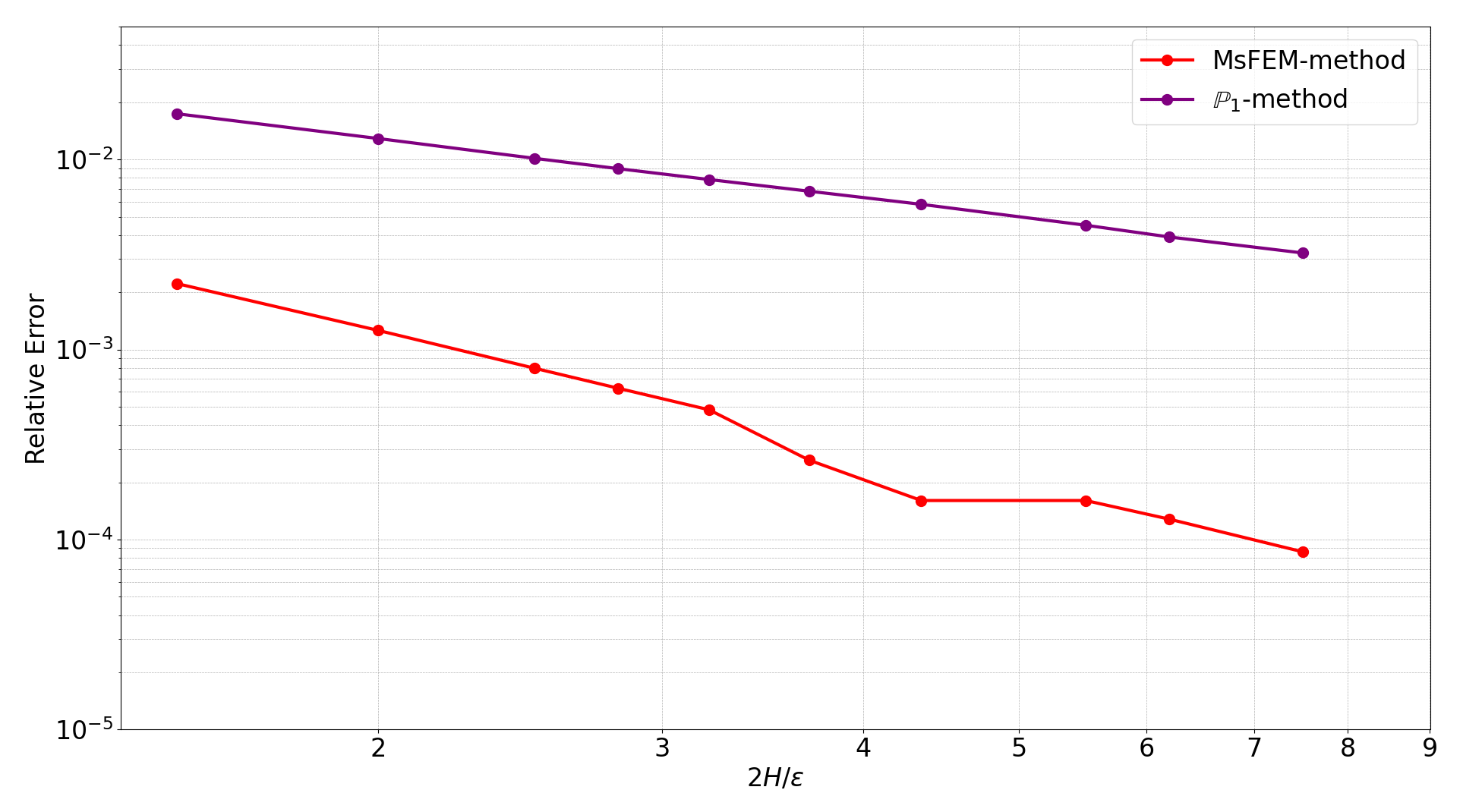}
  \caption{Quasi-periodic case~\eqref{coefficients quasi periodic 2 groupe second_a}--\eqref{coefficients quasi periodic 2 groupe second_b}: relative error~\eqref{erreur valeur propre relative} on the eigenvalue for the MsFEM method and the $\PP_1$-method, as a function of $\eps$ ($H = 1/8$ fixed, $\operatorname{log}-\operatorname{log}$ scale).}
  \label{fig:ErreurVP_QP_vector}
\end{figure}

\section{Proof of convergence of the preliminary method} \label{sec:demo MsFEM preliminaire}

In this section, focusing on the scalar-valued case, we establish an error bound on 
\begin{equation} \label{taux de convergence}
  \frac{\|u^\eps - u^{\eps,\psi}_H \|_{H^1(\Omega)}}{\|u^\eps\|_{H^1(\Omega)}},
\end{equation}
where $u^\eps$ is the first eigenvector of~\eqref{pb reaction diffusion stationnaire} and $u^{\eps,\psi}_H$ is its approximation using the preliminary method, namely the first eigenvector of~\eqref{pb reaction diffusion stationnaire MsFEM ueps triche}. As explained above, this preliminary method is not the one we are practically interested in, but Theorem~\ref{thm taux de convergence methode triche} combined with the results of Section~\ref{sec:demo small perturbations} already provide a good indication of the efficiency of the actual MsFEM method.

To establish a bound on~\eqref{taux de convergence}, we will partly draw inspiration from the proof of convergence of the MsFEM-lin method, detailed in~\cite[Chapter~5]{blancHomogeneisationMilieuPeriodique2022b}. Throughout this proof, we denote by $C$ any constant that does not depend on $\eps$ or $H$ (and that may vary from one line to the next). We assume that the coefficients $A$, $\sigma$ and $\Sigma$ are $Y$-periodic, satisfy~\eqref{hypothese coercivite D} and are regular enough so that
\begin{equation} \label{eq:regul_w_psi}
  \widetilde{w}_j \in W^{1,\infty}(Y) \cap C^0(\overline{Y}) \quad \text{and} \quad \psi \in W^{1,\infty}(Y),
\end{equation}
where we recall that $\widetilde{w}_j$ and $\psi$ are defined by~\eqref{eq correcteur} and~\eqref{pb spectral}, respectively. Furthermore, we assume that $v^\star$, the first eigenvector of the homogenized problem~\eqref{pb reaction diffusion homogeneise}, belongs to $H^1_0(\Omega) \cap H^2(\Omega)$. Finally, we assume that the ratio $\Sigma/\sigma$ is not constant, which implies that $\|\nabla \psi\|_{L^2(Y)} > 0$. This assumption is not very restrictive, as if $\Sigma/\sigma$ is constant, then $\psi = 1$, and Problem~\eqref{pb reaction diffusion stationnaire} becomes a purely diffusive eigenvalue problem, much easier to address.

\begin{theorem} \label{thm taux de convergence methode triche}
  Let $\mathcal{T}_H$ be a regular mesh of $\Omega$ consisting of triangles $K$ with characteristic size $H$. Let $u^\eps$ be the first eigenvector of~\eqref{pb reaction diffusion stationnaire} and $u^{\eps,\psi}_H$ be the first eigenvector of~\eqref{pb reaction diffusion stationnaire MsFEM ueps triche}. We assume that $H \geq \eps$. Under the assumptions mentioned above, we have the following estimate: there exists $\eps_0$ (independent of $H$) such that, for any $\eps \leq \eps_0$,
  \begin{equation} \label{th taux de convergence methode triche}
    \frac{\|u^\eps - u^{\eps,\psi}_H \|_{H^1(\Omega)}}{\|u^\eps\|_{H^1(\Omega)}} \leq C \left( H^2 + \eps + \eps \, \sqrt{\frac{\eps}{H}} + R(\eps) \right),
  \end{equation}
  where $R(\eps)$, defined by~\eqref{thm:strong convergence veps L2}, converges to 0 when $\eps \to 0$. 
\end{theorem}

\begin{proof}
Since the operator associated with the variational formulation of~\eqref{pb reaction diffusion stationnaire} is linear, coercive and continuous, we can use an equivalent of Céa's lemma for elliptic eigenvalue problems to estimate the numerator of~\eqref{taux de convergence}. This result is notably stated in~\cite[Theorem~3.1]{BabuskaEstimatesErrorsEigenvalue1987a}. Since $u^\eps$ is simple (recall Theorem~\ref{thm:spectral}), we can write
\begin{equation} \label{eq:cea}
\|u^\eps - u^{\eps,\psi}_H \|_{H^1(\Omega)} \leq C \inf_{w_H \in V_{\eps,\psi,H}} \| u^\eps - w_H \|_{H^1(\Omega)},
\end{equation}
where we recall that $V_{\eps,\psi,H}$ is defined by~\eqref{eq:def_V_eps_psi_H}. We next recall the following classical finite element result: since $v^\star$ (defined by~\eqref{pb reaction diffusion homogeneise}) belongs to $H^1_0(\Omega) \cap H^2(\Omega)$, there exist some $\{ v_i^\star \}_{1 \leq i \leq N}$ such that
\begin{equation} \label{eq:approx_P1}
  \left\| v^\star - \sum_{i=1}^N v_i^\star \, \chi_i^{\PP_1} \right\|_{H^1(\Omega)} \leq C \, H \ \ \text{and} \ \ \left\| v^\star - \sum_{i=1}^N v_i^\star \, \chi_i^{\PP_1} \right\|_{L^2(\Omega)} \leq C \, H^2.
\end{equation}
Choosing $\dis w_H = \sum_{i=1}^N v_i^\star \, \phi_i^{\eps,\psi}$ in~\eqref{eq:cea}, we have
\begin{align*}
\| u^\eps - u^{\eps,\psi}_H \|_{H^1(\Omega)} 
& \leq C \left\| u^\eps - \sum_{i=1}^N v_i^\star \, \phi_i^{\eps,\psi} \right\|_{H^1(\Omega)}
\\
&= C \left\| \psi\left(\frac{\cdot}{\eps}\right) v^\eps - \sum_{i=1}^N v_i^\star \, \chi_i^{\eps,\psi} \, \psi\left(\frac{\cdot}{\eps}\right) \right\|_{H^1(\Omega)}
\\
&= C \left\| \psi\left(\frac{\cdot}{\eps}\right) \left( v^\eps - \sum_{i=1}^N v_i^\star \, \chi_i^{\eps,\psi} \right) \right\|_{H^1(\Omega)}
\end{align*}
and thus
\begin{multline} \label{eq:normandie10}
\| u^\eps - u^{\eps,\psi}_H \|_{H^1(\Omega)} \leq C \left( \left\| \frac{1}{\eps} \, \nabla \psi\left(\frac{\cdot}{\eps}\right) \left( v^\eps - \sum_{i=1}^N v_i^\star \, \chi_i^{\eps,\psi} \right) \right\|_{L^2(\Omega)} \right. \\ \left. + \left\| \psi\left(\frac{\cdot}{\eps}\right) \left( \nabla v^\eps - \sum_{i=1}^N v_i^\star \, \nabla \chi_i^{\eps,\psi} \right) \right\|_{L^2(\Omega)} \right),
\end{multline}
where we have used the Poincaré inequality at the last line, since the functions $v^\eps$ and $\dis \sum_{i=1}^N v_i^\star \, \chi_i^{\eps,\psi}$ both vanish on the boundary of $\Omega$. Using~\eqref{eq:regul_w_psi}, the above first term satisfies
\begin{equation} \label{demo methode triche terme 1}
\left\| \frac{1}{\eps} \, \nabla \psi\left(\frac{\cdot}{\eps}\right) \left( v^\eps - \sum_{i=1}^N v_i^\star \, \chi_i^{\eps,\psi} \right) \right\|_{L^2(\Omega)} \leq \frac{C}{\eps} \left\| v^\eps - \sum_{i=1}^N v_i^\star \, \chi_i^{\eps,\psi} \right\|_{L^2(\Omega)},
\end{equation}
while the second term satisfies
\begin{equation} \label{demo methode triche terme 2}
\left\| \psi\left(\frac{\cdot}{\eps}\right) \! \left( \nabla v^\eps - \sum_{i=1}^N v_i^\star \nabla \chi_i^{\eps,\psi} \right) \right\|_{L^2(\Omega)} \leq C \left\| \nabla v^\eps - \sum_{i=1}^N v_i^\star \nabla \chi_i^{\eps,\psi} \right\|_{L^2(\Omega)}.
\end{equation}
We successively estimate the two above terms and next conclude.

\medskip

\noindent
{\bf Step~1: estimation of~\eqref{demo methode triche terme 1}.} Introducing the two-scale expansion
\begin{equation} \label{eq:def_v_eps1}
  v^{\eps,1} := v^\star + \eps \sum_{j=1}^d \widetilde{w}_j(\cdot/\eps) \, \partial_j v^\star,
\end{equation}
we write
\begin{align}
  \left\| v^\eps - \sum_{i=1}^N v_i^\star \, \chi_i^{\eps,\psi} \right\|_{L^2(\Omega)}
  &\leq
  \| v^\eps - v^{\eps,1} \|_{L^2(\Omega)} + \left\| v^{\eps,1} - \sum_{i=1}^N v_i^\star \, \chi_i^{\eps,\psi} \right\|_{L^2(\Omega)}
  \nonumber
  \\
  &=
  R(\eps) + \left\| v^{\eps,1} - \sum_{i=1}^N v_i^\star \, \chi_i^{\eps,\psi} \right\|_{L^2(\Omega)},
  \label{eq:normandie}
\end{align}
where we have used the quantity $R(\eps)$ defined by~\eqref{thm:strong convergence veps L2}, that we know to converge to 0 when $\eps \to 0$. To estimate the second term of~\eqref{eq:normandie}, we introduce the two-scale expansion of the MsFEM-lin basis functions $\chi_i^{\eps,\psi}$ defined by~\eqref{chi psi}. Using the fact that the homogenized limit of $\chi_i^{\eps,\psi}$ is $\chi_i^{\PP_1}$, and that the corresponding corrector function is $\widetilde{w}_j$, we write
\begin{equation} \label{eq:def_Theta}
\chi_i^{\eps,\psi} = \chi_i^{\PP_1} + \eps \sum_{j=1}^d \widetilde{w}_j\left(\frac{\cdot}{\eps}\right) \partial_j \chi_i^{\PP_1} + \Theta_i^\eps,
\end{equation}
with $\dis \lim_{\eps \to 0} \| \Theta_i^\eps \|_{H^1(K)} = 0$ on each element $K$. A more precise estimate will be considered below. We can thus write
\begin{multline} \label{eq:normandie2}
\left\| v^{\eps,1} - \sum_{i=1}^N v_i^\star \, \chi_i^{\eps,\psi} \right\|_{L^2(\Omega)} \leq \left\| v^\star - \sum_{i=1}^N v_i^\star \, \chi_i^{\PP_1} \right\|_{L^2(\Omega)} \\ + \left\| \sum_{j=1}^d \eps \, \widetilde{w}_j\left(\frac{\cdot}{\eps}\right) \left( \partial_j v^\star - \sum_{i=1}^N v_i^\star \, \partial_j \chi_i^{\PP_1} \right) \right\|_{L^2(\Omega)} + \left\| \sum_{i=1}^N v_i^\star \, \Theta_i^\eps \right\|_{L^2(\Omega)}.
\end{multline}
The first term of~\eqref{eq:normandie2} is bounded by $C \, H^2$ using~\eqref{eq:approx_P1}. Using~\eqref{eq:regul_w_psi} and again~\eqref{eq:approx_P1}, we bound the second term of~\eqref{eq:normandie2}:
\begin{multline} \label{eq:estimation v* chi P1 norme H1}
\left\| \sum_{j=1}^d \eps \, \widetilde{w}_j\left(\frac{\cdot}{\eps}\right) \left( \partial_j v^\star - \sum_{i=1}^N v_i^\star \, \partial_j \chi_i \right) \right\|_{L^2(\Omega)} \\ \leq C \, \eps \left\| \sum_{j=1}^d \left( \partial_j v^\star - \sum_{i=1}^N v_i^\star \, \partial_j \chi_i \right) \right\|_{L^2(\Omega)} \leq C \, \eps \, H.
\end{multline}
Let us examine the equation satisfied by $\dis \sum_{i=1}^N v_i^\star \, \Theta_i^\eps$ to estimate its $L^2(\Omega)$ norm (and thus the third term of~\eqref{eq:normandie2}). Recall that we denote $\widetilde{A}(y) = \psi(y)^2 \, A(y)$ and $\widetilde{A}_\eps = \widetilde{A}(\cdot/\eps)$. On any element $K$ of the coarse mesh, we have
$$
- \operatorname{div} \left( \widetilde{A}_\eps \nabla \sum_{i=1}^N v_i^\star \, \Theta_i^\eps \right) = - \sum_{i=1}^N v_i^\star \, \operatorname{div} \left( \widetilde{A}_\eps \, \nabla \Theta_i^\eps \right),
$$
and 
\begin{align*}
  - \operatorname{div} \left( \widetilde{A}_\eps \, \nabla \Theta_i^\eps \right)
  &=
  - \operatorname{div} \left( \widetilde{A}_\eps \, \nabla \left( \chi_i^{\eps,\psi} - \chi_i^{\PP_1} - \eps \sum_{j=1}^d \widetilde{w}_j\left(\frac{\cdot}{\eps}\right) \partial_j \chi_i^{\PP_1} \right) \right)
  \\
  &=
  \operatorname{div} \left( \widetilde{A}_\eps \, \nabla \chi_i^{\PP_1} \right) + \sum_{j=1}^d \operatorname{div} \left( \widetilde{A}_\eps \nabla \widetilde{w}_j\left(\frac{\cdot}{\eps}\right) \partial_j \chi_i^{\PP_1} \right),
\end{align*} 
using the equation satisfied by $\chi_i^{\eps,\psi}$, and the fact that $\partial_j \chi_i^{\PP_1}$ is a constant. Using again that argument for the first term above, we obtain
$$
- \operatorname{div} \left( \widetilde{A}_\eps \, \nabla \Theta_i^\eps \right) = \sum_{j=1}^d \partial_j \chi_i^{\PP_1} \operatorname{div} \left( \widetilde{A}_\eps \, e_j \right) + \sum_{j=1}^d \partial_j \chi_i^{\PP_1} \, \operatorname{div} \left( \widetilde{A}_\eps \nabla \widetilde{w}_j\left(\frac{\cdot}{\eps}\right) \right) = 0,
$$
where we have eventually used the corrector equation~\eqref{eq correcteur}. We thus deduce that, on each element $K$, the function $\dis \theta^\eps = \sum_{i=1}^N v_i^\star \, \Theta_i^\eps$ satisfies
\begin{equation} \label{eq:theta eps}
  \left\{
  \begin{aligned}
    -\operatorname{div} \left( \widetilde{A}_\eps \nabla \theta^\eps \right) &= 0 && \text{in $K$},
    \\
    \theta^\eps &= - \eps \sum_{i=1}^N v_i^\star \, \sum_{j=1}^d \widetilde{w}_j\left(\frac{\cdot}{\eps}\right) \partial_j \chi_i^{\PP_1} && \text{on $\partial K$},
  \end{aligned}
  \right.
\end{equation}
where we have used that $\chi_i^{\eps,\psi} = \chi_i^{\PP_1}$ on $\partial K$. Using the maximum principle for~\eqref{eq:theta eps} and the assumption~\eqref{eq:regul_w_psi}, we deduce that, for any $x \in K$,
\begin{multline*}
  \left| \sum_{i=1}^N v_i^\star \, \Theta_i^\eps (x) \right| \leq \eps \left\| \sum_{j=1}^d \widetilde{w}_j\left(\frac{x}{\eps}\right) \sum_{i=1}^N v_i^\star \, \partial_j \chi_i^{\PP_1} \right\|_{L^\infty(\partial K)} \\ \leq C \, \eps \, \sup_j \left\| \sum_{i=1}^N v_i^\star \, \partial_j \chi_i^{\PP_1} \right\|_{L^\infty(\partial K)} = C \, \eps \left\| \sum_{i=1}^N v_i^\star \, \nabla \chi_i^{\PP_1} \right\|_{L^\infty(\partial K)}.
\end{multline*}
Using that $\nabla \chi_i^{\PP_1}$ is constant on $K$, we infer
\begin{align*}
\left\| \sum_{i=1}^N v_i^\star \, \Theta_i^\eps \right\|_{L^2(K)}^2 
&\leq
C \, \eps^2 \, | K | \left\| \sum_{i=1}^N v_i^\star \, \nabla \chi_i^{\PP_1} \right\|_{L^\infty(\partial K)}^2
\\
&=
C \, \eps^2 \left\| \sum_{i=1}^N v_i^\star \, \nabla \chi_i^{\PP_1} \right\|_{L^2(K)}^2.
\end{align*} 
Summing over the elements $K$, we get
\begin{align}
\left\| \sum_{i=1}^N v_i^\star \, \Theta_i^\eps \right\|_{L^2(\Omega)}^2
&\leq
C \, \eps^2 \left\| \sum_{i=1}^N v_i^\star \, \nabla \chi_i^{\PP_1} \right\|_{L^2(\Omega)}^2
\nonumber
\\
&\leq
C \, \eps^2 \left[ \left\| \nabla v^\star - \sum_{i=1}^N v_i^\star \, \nabla \chi_i^{\PP_1} \right\|_{L^2(\Omega)}^2 + \| \nabla v^\star \|_{L^2(\Omega)}^2 \right]
\nonumber
\\
&\leq C \, \eps^2 \, (H^2 +1),
\label{eq:estimation somme v* theta eps}
\end{align}
and therefore (since $H \leq 1$)
\begin{equation} \label{eq:normandie3}
\left\| \sum_{i=1}^N v_i^\star \, \Theta_i^\eps \right\|_{L^2(\Omega)} \leq C \, \eps. 
\end{equation}
We have completed the estimation of the right-hand side of~\eqref{demo methode triche terme 1}: collecting~\eqref{eq:normandie}, \eqref{eq:normandie2}, \eqref{eq:estimation v* chi P1 norme H1} and~\eqref{eq:normandie3}, we deduce
\begin{align}
\frac{1}{\eps} \left\| v^\eps - \sum_{i=1}^N v_i^\star \, \chi_i^{\eps,\psi} \right\|_{L^2(\Omega)}
&\leq
\frac{C}{\eps} \left( R(\eps) + H^2 + \eps \, H + \eps \right)
\nonumber
\\
&\leq
C \left( \frac{H^2}{\eps} + 1 + \frac{R(\eps)}{\eps} \right),
\label{eq:normandie8}
\end{align}
where the simplification in the last line stems from the fact that $H \leq 1$.

\medskip

\noindent
{\bf Step~2: estimation of~\eqref{demo methode triche terme 2}.} Again using the two-scale expansion $v^{\eps,1}$ defined by~\eqref{eq:def_v_eps1}, we write
\begin{align}
  \left\| \nabla v^\eps - \sum_{i=1}^N v_i^\star \, \nabla \chi_i^{\eps,\psi} \right\|_{L^2(\Omega)}
  &\leq
  \left\| \nabla v^\eps - \nabla v^{\eps,1} \right\|_{L^2(\Omega)} + \left\| \nabla v^{\eps,1} - \sum_{i=1}^N v_i^\star \, \nabla \chi_i^{\eps,\psi} \right\|_{L^2(\Omega)}
  \nonumber
  \\
  &=
  \widetilde{R}(\eps) + \left\| \nabla v^{\eps,1} - \sum_{i=1}^N v_i^\star \, \nabla \chi_i^{\eps,\psi} \right\|_{L^2(\Omega)},
  \label{eq:normandie4}
\end{align}
where we have used the quantity $\widetilde{R}(\eps)$ defined by~\eqref{thm:strong convergence veps H1}, that we know to converge to 0 when $\eps \to 0$. In view of~\eqref{eq:def_Theta}, we recall that
\begin{multline*}
  v^{\eps,1} - \sum_{i=1}^N v_i^\star \, \chi_i^{\eps,\psi} = \left( v^\star - \sum_{i=1}^N v_i^\star \, \chi_i^{\PP_1} \right) \\ + \sum_{j=1}^d \eps \, \widetilde{w}_j\left(\frac{\cdot}{\eps}\right) \left( \partial_j v^\star - \sum_{i=1}^N v_i^\star \, \partial_j \chi_i^{\PP_1} \right) - \sum_{i=1}^N v_i^\star \, \Theta_i^\eps.
\end{multline*} 
We thus have
\begin{align}
  & \nabla v^{\eps,1} - \sum_{i=1}^N v_i^\star \, \nabla \chi_i^{\eps,\psi}
  \nonumber
  \\
  &=
  \nabla v^\star - \sum_{i=1}^N v_i^\star \, \nabla \chi_i^{\PP_1} + \sum_{j=1}^d \nabla \widetilde{w}_j\left(\frac{\cdot}{\eps}\right) \left( \partial_j v^\star - \sum_{i=1}^N v_i^\star \, \partial_j \chi_i^{\PP_1} \right)
  \nonumber
  \\
  &+ \sum_{j=1}^d \eps \, \widetilde{w}_j\left(\frac{\cdot}{\eps}\right) \partial_j \nabla v^\star - \sum_{i=1}^N v_i^\star \left( \nabla \Theta_i^\eps + \sum_{j=1}^d \eps \, \widetilde{w}_j\left(\frac{\cdot}{\eps}\right) \partial_j \nabla \chi_i^{\PP_1} \right).
  \label{eq:decomposition gradient veps chi eps psi}
\end{align}
Care must be taken before writing the $L^2$ norm of these terms. Indeed, the gradients $\nabla \chi_i^{\PP_1}$ are not in $H^1(\Omega)$. However, by grouping the last terms, we have
$$
\sum_{i=1}^N v_i^\star \left( \nabla \Theta_i^\eps + \sum_{j=1}^d \eps \, \widetilde{w}_j\left(\frac{\cdot}{\eps}\right) \partial_j \nabla \chi_i^{\PP_1} \right) \in L^2(\Omega),
$$
because all the other terms in~\eqref{eq:decomposition gradient veps chi eps psi} are in $L^2(\Omega)$. In this way, we can write the $L^2$ norm of this sum, which yields
\begin{align}
  & \left\| \nabla v^{\eps,1} - \sum_{i=1}^N v_i^\star \, \nabla \chi_i^{\eps,\psi} \right\|_{L^2(\Omega)}
  \nonumber
  \\
  &\leq
  \left\| \nabla v^\star - \sum_{i=1}^N v_i^\star \, \nabla \chi_i^{\PP_1} \right\|_{L^2(\Omega)}
  \nonumber
  \\
  &+
  \left\| \sum_{j=1}^d \nabla \widetilde{w}_j\left(\frac{\cdot}{\eps}\right) \left( \partial_j v^\star - \sum_{i=1}^N v_i^\star \, \partial_j \chi_i^{\PP_1} \right) \right\|_{L^2(\Omega)}
  \nonumber
  \\
  &+
  \left\| \sum_{j=1}^d \eps \, \widetilde{w}_j\left(\frac{\cdot}{\eps}\right) \partial_j \nabla v^\star \right\|_{L^2(\Omega)}
  \nonumber
  \\
  &+
  \left\| \sum_{i=1}^N v_i^\star \left( \nabla \Theta_i^\eps + \sum_{j=1}^d \eps \, \widetilde{w}_j\left(\frac{\cdot}{\eps}\right) \partial_j \nabla \chi_i^{\PP_1}\right) \right\|_{L^2(\Omega)}.
\label{eq:estimation gradient veps chi eps psi}
\end{align}
Using~\eqref{eq:approx_P1} and~\eqref{eq:regul_w_psi}, we can bound the first two terms in the above estimate by $C \, H$. For the third term, we have
\begin{equation} \label{eq:normandie5}
\left\| \sum_{j=1}^d \eps \, \widetilde{w}_j\left(\frac{\cdot}{\eps}\right) \partial_j \nabla v^\star \right\|_{L^2(\Omega)} \leq C \, \eps,
\end{equation}
since we have assumed $v^\star \in H^2(\Omega)$. It remains to estimate the last term of~\eqref{eq:estimation gradient veps chi eps psi}. We split the $L^2$ norm of this term into the $L^2$ norm over each element $K$ of the mesh:
\begin{align}
  & \left\| \sum_{i=1}^N v_i^\star \left( \nabla \Theta_i^\eps + \sum_{j=1}^d \eps \, \widetilde{w}_j\left(\frac{\cdot}{\eps}\right) \partial_j \nabla \chi_i^{\PP_1} \right) \right\|^2_{L^2(\Omega)} 
  \nonumber
  \\
  &=
  \sum_{K \in \mathcal{T}_H} \left\| \sum_{i=1}^N v_i^\star \left( \nabla \Theta_i^\eps + \sum_{j=1}^d \eps \, \widetilde{w}_j\left(\frac{\cdot}{\eps}\right) \partial_j \nabla \chi_i^{\PP_1} \right) \right\|^2_{L^2(K)}
  \nonumber
  \\
  &= \sum_{K \in \mathcal{T}_H} \left\| \sum_{i=1}^N v_i^\star \, \nabla \Theta_i^\eps \right\|^2_{L^2(K)},
  \label{eq:normandie6}
\end{align}
since $\nabla \chi_i^{\PP_1}$ is constant on each element $K$. It thus remains to estimate $\dis \left\| \sum_{i=1}^N v_i^\star \, \nabla \Theta_i^\eps \right\|_{L^2(K)}$ for each element $K$ to conclude.

Consider some $\delta < H$ to be fixed later on, and let $K_\delta$ be the subset of the element $K$ defined as
$$
K_\delta = \{ x \in K, \quad \operatorname{dist}(x, \partial K) \geq \delta \}.
$$
Let $\xi \in C^\infty(K)$ such that $0 \leq \xi \leq 1$ on $K$, with $\xi = 0$ on $K_\delta$, $\xi = 1$ on $K \setminus K_{\delta/2}$, and such that $\| \nabla \xi \|_{L^\infty(K)} \leq C \, \delta^{-1}$. The function $\xi$ is equal to 1 in a neighboorhood of $\partial K$, and vanishes in the bulk of $K$. Since the mesh is regular, we can choose $C$ independent of the mesh element $K$ and of $H$. Using~\eqref{eq:theta eps}, we then observe that
\begin{equation*}
  \sum_{i=1}^N v_i^\star \, \Theta_i^\eps + \xi \, \eps \, \sum_{i=1}^N v_i^\star \, \sum_{j=1}^d \widetilde{w}_j\left(\frac{\cdot}{\eps}\right) \partial_j \chi_i^{\PP_1} \in H^1_0(K).
\end{equation*}
We can thus use the above function in the variational formulation of~\eqref{eq:theta eps} and obtain
\begin{equation} \label{eq:produit scalaire theta eps}
  \int_K \widetilde{A}_\eps \nabla \left( \sum_{i=1}^N v_i^\star \, \Theta_i^\eps \right) \cdot \nabla \left( \sum_{i=1}^N v_i^\star \, \Theta_i^\eps + \xi \, \eps \, \sum_{i=1}^N v_i^\star \, \sum_{j=1}^d \widetilde{w}_j\left(\frac{\cdot}{\eps}\right) \partial_j \chi_i^{\PP_1} \right) = 0.
\end{equation}
Using the coercivity and boundedness of $\widetilde{A}_\eps$ (recall~\eqref{hypothese coercivite D} and~\eqref{eq:borne_psi}), as well as~\eqref{eq:produit scalaire theta eps}, we deduce
\begin{align*}
  \left\| \sum_{i=1}^N v_i^\star \, \nabla \Theta_i^\eps \right\|_{L^2(K)}^2
  &
  \leq \frac{1}{\widetilde{\alpha}} \int_K \widetilde{A}_\eps \nabla \left( \sum_{i=1}^N v_i^\star \, \Theta_i^\eps \right) \cdot \nabla \left( \sum_{i=1}^N v_i^\star \, \Theta_i^\eps \right)
  \\
  &
  = \frac{-1}{\widetilde{\alpha}} \int_K \widetilde{A}_\eps \nabla \left( \sum_{i=1}^N v_i^\star \, \Theta_i^\eps \right) \cdot \nabla \left( \xi \, \eps \, \sum_{i=1}^N v_i^\star \, \sum_{j=1}^d \widetilde{w}_j\left(\frac{\cdot}{\eps}\right) \partial_j \chi_i^{\PP_1} \right)
  \\
  &
  \leq C \left\| \sum_{i=1}^N v_i^\star \, \nabla \Theta_i^\eps \right\|_{L^2(K)} \ \left\| \nabla \left( \xi \, \eps \, \sum_{i=1}^N v_i^\star \, \sum_{j=1}^d \widetilde{w}_j\left(\frac{\cdot}{\eps}\right) \partial_j \chi_i^{\PP_1} \right) \right\|_{L^2(K)},
\end{align*}
by the Cauchy-Schwarz inequality. Using~\eqref{eq:regul_w_psi} and the fact that $\xi$ vanishes in $K_\delta$, we infer
\begin{align*}
  & \left\| \sum_{i=1}^N v_i^\star \, \nabla \Theta_i^\eps \right\|_{L^2(K)}
  \\
  & \leq C \left\| \nabla \left( \xi \, \eps \sum_{i=1}^N v_i^\star \, \sum_{j=1}^d \widetilde{w}_j\left(\frac{\cdot}{\eps}\right) \partial_j \chi_i^{\PP_1} \right) \right\|_{L^2(K)}
  \\
  & \leq C \left( \left\| \nabla \xi \, \eps \sum_{i=1}^N v_i^\star \, \sum_{j=1}^d \widetilde{w}_j\left(\frac{\cdot}{\eps}\right) \partial_j \chi_i^{\PP_1} \right\|_{L^2(K)} + \left\| \xi \, \sum_{i=1}^N v_i^\star \, \sum_{j=1}^d \nabla \widetilde{w}_j\left(\frac{\cdot}{\eps}\right) \partial_j \chi_i^{\PP_1} \right\|_{L^2(K)} \right)
  \\
  & \leq C \left( \eps \, \delta^{-1} \, \sqrt{| K \setminus K_{\delta}|} \, \left\| \sum_{i=1}^N v_i^\star \, \nabla \chi_i^{\PP_1} \right\|_{L^\infty(K)} + \sqrt{| K \setminus K_{\delta}|} \, \left\| \sum_{i=1}^N v_i^\star \, \nabla \chi_i^{\PP_1} \right\|_{L^\infty(K)} \right)
  \\
  &= C \, \sqrt{| K \setminus K_{\delta}|} \, \left\| \sum_{i=1}^N v_i^\star \, \nabla \chi_i^{\PP_1} \right\|_{L^\infty(K)} \, (1 + \eps \, \delta^{-1}).
\end{align*}
Moreover, we have $| \partial K | \leq C \, H^{d-1}$ and $| K \setminus K_{\delta}| \leq \delta \, | \partial K |$. Thus, by taking $\delta = \eps$ (which is a possible choice since we consider the regime $\eps \leq H$), we obtain the estimate
\begin{align} 
  \left\| \sum_{i=1}^N v_i^\star \, \nabla \Theta_i^\eps \right\|_{L^2(K)}
  &\leq C \, \sqrt{\eps \, H^{d-1}} \, \left\| \sum_{i=1}^N v_i^\star \, \nabla \chi_i^{\PP_1} \right\|_{L^\infty(K)}
  \nonumber
  \\
  &= C \, \sqrt{\frac{\eps \, H^{d-1}}{|K|}} \, \left\| \sum_{i=1}^N v_i^\star \, \nabla \chi_i^{\PP_1} \right\|_{L^2(K)}
  \nonumber
  \\
  &= C \, \sqrt{\frac{\eps}{H}} \, \left\| \sum_{i=1}^N v_i^\star \, \nabla \chi_i^{\PP_1} \right\|_{L^2(K)},
\label{eq:estimation gradient theta eps}
\end{align}
where $C$ is independent of $H$, $\eps$ and the mesh element $K$ (we have used that $\nabla \chi_i^{\PP_1}$ is constant on $K$ and the regularity of the mesh). We now sum~\eqref{eq:estimation gradient theta eps} over $K$:
\begin{align}
\sum_{K \in \mathcal{T}_H} \left\| \sum_{i=1}^N v_i^\star \, \nabla \Theta_i^\eps \right\|_{L^2(K)}^2
&
\leq C \, \frac{\eps}{H} \, \sum_{K \in \mathcal{T}_H} \left\| \sum_{i=1}^N v_i^\star \, \nabla \chi_i^{\PP_1} \right\|_{L^2(K)}^2
\nonumber
\\
&
= C \, \frac{\eps}{H} \, \left\| \sum_{i=1}^N v_i^\star \, \nabla \chi_i^{\PP_1} \right\|_{L^2(\Omega)}^2
\nonumber
\\
&
\leq C \, \frac{\eps}{H} \, (H^2 + 1),
\label{eq:normandie7}
\end{align}
by the same argument as for the estimate~\eqref{eq:estimation somme v* theta eps}. Collecting~\eqref{eq:normandie4}, \eqref{eq:estimation gradient veps chi eps psi}, \eqref{eq:normandie5}, \eqref{eq:normandie6} and~\eqref{eq:normandie7}, we obtain
\begin{equation} \label{eq:normandie9}
  \left\| \nabla v^\eps - \sum_{i=1}^N v_i^\star \, \nabla \chi_i^{\eps,\psi} \right\|_{L^2(\Omega)} \leq C \left( H + \eps + \sqrt{\frac{\eps}{H}} + \widetilde{R}(\eps) \right). 
\end{equation}

\medskip

\noindent
{\bf Step~3: conclusion.} In view of~\eqref{eq:normandie10}, \eqref{demo methode triche terme 1}, \eqref{demo methode triche terme 2}, \eqref{eq:normandie8} and~\eqref{eq:normandie9}, we deduce
$$
\|u^\eps - u^{\eps,\psi}_H \|_{H^1(\Omega)} \leq C \left( H + \frac{H^2}{\eps} + 1 + \eps + \sqrt{\frac{\eps}{H}} + \frac{R(\eps)}{\eps} + \widetilde{R}(\eps) \right).
$$
Since $\eps \leq H \leq 1$, we can recast the above bound as
\begin{equation} \label{eq:normandie11}
\|u^\eps - u^{\eps,\psi}_H \|_{H^1(\Omega)} \leq C \left( \frac{H^2}{\eps} + 1 + \sqrt{\frac{\eps}{H}} + \frac{R(\eps)}{\eps} + \widetilde{R}(\eps) \right).
\end{equation}
We now establish a lower bound on the denominator of~\eqref{taux de convergence}, namely $\|u^\eps\|_{H^1(\Omega)}$. We write
\begin{align*}
  \| u^\eps \|_{H^1(\Omega)}
  &
  \geq \| \nabla u^\eps \|_{L^2(\Omega)}
  \\
  &
  = \left\| \frac{1}{\eps} \left(\nabla \psi \right)\left(\frac{\cdot}{\eps}\right) v^\eps + \psi\left(\frac{\cdot}{\eps}\right) \nabla v^\eps \right\|_{L^2(\Omega)}
  \\
  &
  \geq \left| \left\| \frac{1}{\eps} \left(\nabla \psi \right)\left(\frac{\cdot}{\eps}\right) v^\eps \right\|_{L^2(\Omega)} - \left\| \psi\left(\frac{\cdot}{\eps}\right) \nabla v^\eps \right\|_{L^2(\Omega)} \right|
  \\
  &
  \geq \left| \frac{1}{\eps} \left\| \left(\nabla \psi \right)\left(\frac{\cdot}{\eps}\right) v^\eps \right\|_{L^2(\Omega)} - C \right|,
\end{align*}
where the last line stems from the fact that $\psi$ is bounded (see~\eqref{eq:regul_w_psi}) and $\nabla v^\eps$ is bounded in $L^2(\Omega)$ (see the claim below Theorem~\ref{changement variable veps}). When $\eps$ is sufficiently small, the quantity $\left\| \left(\nabla \psi \right)\left(\cdot/\eps\right) v^\eps \right\|_{L^2(\Omega)}$is bounded away from 0. Indeed, using that $v^\eps$ converges strongly to $v^\star$ in $L^2(\Omega)$ and that $|\nabla \psi|^2$ is a periodic function (of average $\langle |\nabla \psi|^2 \rangle$), we have
$$
\int_\Omega \left(\nabla \psi \right)^2\left(\frac{\cdot}{\eps}\right) (v^\eps)^2
\underset{\eps \rightarrow 0}{\longrightarrow}
\int_\Omega \langle |\nabla \psi|^2 \rangle \, (v^\star)^2
=
\| \nabla \psi \|_{L^2(Y)}^2 \, \| v^\star \|_{L^2(\Omega)}^2
> 0,
$$
where the last inequality stems from our assumption that $\|\nabla \psi\|_{L^2(Y)} > 0$. Thus, for $\eps$ small enough, we have
$$
\| u^\eps \|_{H^1(\Omega)} \geq \frac{1}{2\eps} \, \| \nabla \psi \|_{L^2(Y)} \, \| v^\star \|_{L^2(\Omega)}.
$$
Using the fact that, for $\eps$ small enough, $\widetilde{R}(\eps) \leq 1$, we are in position to deduce from~\eqref{eq:normandie11} that
$$
\frac{\|u^\eps - u^{\eps,\psi}_H \|_{H^1(\Omega)}}{\|u^\eps\|_{H^1(\Omega)}} \leq C \left( H^2 + \eps + \eps \, \sqrt{\frac{\eps}{H}} + R(\eps) \right),
$$
which concludes the proof of~\eqref{th taux de convergence methode triche}.
\end{proof}

\section{Partial analysis of the filtering method} \label{sec:demo small perturbations}

In this section, we present some calculations to explain how the filtering method ensures, in the periodic setting, the convergence of the first eigencouple of~\eqref{pb patch oversampling filtre tmp} to the first eigencouple of~\eqref{pb spectral}. We thus focus on the scalar-valued setting (the vector-valued case is briefly commented upon in Remark~\ref{rem:vector_valued}). Our main results are the identity~\eqref{eq:resu_ordre_0}, the estimate~\eqref{estimee convergence valeur propre} on the eigenvalue and the estimate~\eqref{estimee convergence vecteur propre} on the eigenvector (we also refer to the estimate~\eqref{mu eps 1} on the Lagrange multiplier, which helps understand the consistency of~\eqref{pb patch oversampling filtre tmp} with~\eqref{pb spectral} in the limit $\eps \to 0$).

To simplify the notation, we are going to assume that~\eqref{pb patch oversampling filtre tmp} is posed on $\omega=(0,1)^d$ instead of $S_K$, and denote by $(\lambda^\eps, \psi^\eps, \mu^\eps)$ its solution (instead of $(\widetilde{\lambda}^\eps, \widetilde{\psi}^\eps, \widetilde{\mu}^\eps)$). Throughout this section, $C$ represents a positive constant independent of the small scale $\eps$ and of the perturbative parameter $\delta$ introduced in~\eqref{petites perturbations:pb spectral sigma} below. Let $\tau_0$ be a function defined on $(0,1)$ and satisfying~\eqref{hypotheses filtre varphi0}, and let $\tau$ be the filter function defined on $\omega$ by
\begin{equation} \label{definition filtre d dimension}
  \forall x=(x_1,\dots,x_d) \in \omega, \quad \tau(x) = \prod_{i=1}^d \tau_0(x_i).
\end{equation}
Let $(\lambda^\infty,\psi)$ be the first eigencouple of the cell problem~\eqref{pb spectral}, where we fix the normalisation of $\psi$ by assuming $\| \psi \|_{L^2(Y)} = 1$. The approximation method consists in looking for a Lagrange multiplier $\mu^\eps \in \RR^d$ and the eigenvector $\psi^\eps \in H^1(\omega)$ associated to the smallest eigenvalue $\lambda^\eps \in \RR$ such that, for any $v \in H^1(\omega)$ and any $\mu \in \RR^d$,
\begin{equation} \label{demo filtre : pb patch oversampling filtre}
  \left\{
  \begin{aligned}
    \eps^2 \int_\omega \tau \, A^\eps \nabla \psi^\eps \cdot \nabla v + \int_\omega \tau \, \Sigma^\eps \, \psi^\eps \, v
    &=
    \lambda^\eps \int_\omega \tau \, \sigma^\eps \, \psi^\eps \, v + \mu^\eps \cdot \int_\omega \tau \, \nabla v,
    \\
    \mu \cdot \int_\omega \tau \, \nabla \psi^\eps
    &=
    0,
  \end{aligned}
  \right.
\end{equation}
where we recall that $A^\eps = A(\cdot/\eps)$ (and likewise for $\Sigma^\eps$ and $\sigma^\eps$). We fix the normalisation of $\psi^\eps$ by assuming $\| \sqrt{\tau} \, \psi^\eps \|_{L^2(\omega)} = |\omega|^{1/2} = 1$.

\begin{proposition}
  Let $\tau$ be given by~\eqref{definition filtre d dimension} with $\tau_0$ satisfying~\eqref{hypotheses filtre varphi0}. Then $(\lambda^\eps, \psi^\eps, \mu^\eps) \in \RR \times H^1(\omega) \times \RR^d$ is a solution to~\eqref{demo filtre : pb patch oversampling filtre} if and only if $(\lambda^\eps, \psi^\eps)$ is the first eigencouple to the problem
  \begin{equation} \label{demo pb reaction diffusion stationnaire filtre}
    \left\{
    \begin{aligned}
      \tau \, \Sigma^\eps \, \psi^\eps - \operatorname{div} \Big( \tau \left[ \eps^2 \, A^\eps \nabla \psi^\eps - \mu^\eps \right] \Big)
      &=
      \lambda^\eps \, \tau \, \sigma^\eps \, \psi^\eps \quad \text{in $\omega$},
      \\
      \int_\omega \tau \, \nabla \psi^\eps
      &=
      0.
    \end{aligned}
    \right.
  \end{equation}
\end{proposition}

\begin{proof}
  The result is straightforward by integration by parts, using the fact that $\tau$ vanishes at the boundary of $\omega$.
\end{proof}

In this section, we consider the framework of small perturbations for the functions $\Sigma$ and $\sigma$. We thus assume that $\Sigma$ and $\sigma$ satisfy the following expansions:
\begin{equation} \label{petites perturbations:pb spectral sigma}
  \Sigma(y) = \Sigma_0 + \delta \, \Sigma_1(y) + \text{h.o.t.},
  \qquad
  \sigma(y) = \sigma_0 + \delta \, \sigma_1(y) + \text{h.o.t.},
\end{equation}
where $\Sigma_0$ and $\sigma_0$ are constant, $\delta \ll 1$ is a small perturbation parameter, and the higher-order terms ($\text{h.o.t.}$) are of magnitude bounded by $C \, \delta^2$. Similarly, we assume that the solution $(\lambda^\infty,\psi)$ to~\eqref{pb spectral} satisfies the following expansion:
\begin{equation} \label{petites perturbations:pb spectral psi}
  \psi = \psi_0 + \delta \, \psi_1 + \dots,
  \qquad
  \lambda^\infty = \lambda_0^\infty + \delta \, \lambda_1^\infty + \dots
\end{equation}
Inserting~\eqref{petites perturbations:pb spectral sigma} and~\eqref{petites perturbations:pb spectral psi} into~\eqref{pb spectral}, and identifying in powers of $\delta$, we obtain the following results: at the order $\delta^0$,
\begin{equation} \label{pb psi ordre 0}
  \Sigma_0 \, \psi_0 - \operatorname{div} \left( A \nabla \psi_0 \right) = \lambda_0^\infty \, \sigma_0 \, \psi_0,
\end{equation}
and, at the order $\delta^1$,
\begin{equation} \label{pb psi ordre 1}
  \Sigma_0 \, \psi_1 + \Sigma_1 \, \psi_0 - \operatorname{div} \left( A \nabla \psi_1 \right) = \lambda_0^\infty \, \sigma_0 \, \psi_1 + \lambda_0^\infty \, \sigma_1 \, \psi_0 + \lambda_1^\infty \, \sigma_0 \, \psi_0,
\end{equation}
with $Y$-periodic boundary conditions on $\psi_0$ and $\psi_1$.

Inserting~\eqref{petites perturbations:pb spectral psi} into the normalisation condition $\| \psi \|_{L^2(Y)} = 1$ and expanding in $\delta$, we have
\begin{equation} \label{normalisation psi}
  \int_Y \psi_0^2 = 1, \qquad \int_Y \psi_0 \, \psi_1 = 0.
\end{equation}
We multiply~\eqref{pb psi ordre 0} by $\psi_0$, integrate over $Y$ and use~\eqref{normalisation psi} to deduce $\dis \Sigma_0 + \int_Y \nabla \psi_0 \cdot A \nabla \psi_0 = \lambda_0^\infty \, \sigma_0$. Since $A$ is coercive, this yields $\lambda_0^\infty \, \sigma_0 \geq \Sigma_0$. Since $(\lambda_0^\infty,\psi_0)$ is the first eigencouple of~\eqref{pb psi ordre 0}, we conclude that
\begin{equation} \label{psi ordre 0}
  \lambda_0^\infty \, \sigma_0 = \Sigma_0, \quad \psi_0 = 1.
\end{equation}
In view of~\eqref{psi ordre 0}, we recast~\eqref{pb psi ordre 1} as
\begin{equation} \label{psi ordre 1}
  \Sigma_1 - \operatorname{div} \left( A \nabla \psi_1 \right) = \lambda_0^\infty \, \sigma_1 + \lambda_1^\infty \, \sigma_0.
\end{equation}
Integrating this equation over $Y$, using the periodic boundary conditions and denoting by $\langle \cdot \rangle$ the mean of a function over $Y$, we obtain
\begin{equation} \label{lambda 1 etoile}
  \lambda_1^\infty \, \sigma_0 = \langle \Sigma_1 \rangle - \lambda_0^\infty \langle \sigma_1 \rangle.
\end{equation}
We now proceed similarly for~\eqref{demo pb reaction diffusion stationnaire filtre}, assuming
\begin{equation*} 
  \psi^\eps = \psi_0^\eps + \delta \, \psi_1^\eps + \dots,
  \quad
  \lambda^\eps = \lambda_0^\eps + \delta \, \lambda_1^\eps + \dots,
  \quad
  \mu^\eps = \mu^\eps_0 + \delta \, \mu^\eps_1 + \dots,
\end{equation*}
and denoting $\Sigma_i^\eps = \Sigma_i(\cdot/\eps)$ for $i=0,1$ (and likewise for $\sigma_i$). The equation at order $\delta^0$ reads
\begin{equation} \label{phi filtre ordre 0}
  \left\{
  \begin{aligned}
    \tau \, \Sigma_0 \, \psi_0^\eps - \operatorname{div} \Big( \tau \left[ \eps^2 \, A^\eps \nabla \psi_0^\eps - \mu^\eps_0 \right] \Big) &= \lambda_0^\eps \, \tau \, \sigma_0 \, \psi_0^\eps \quad \text{in $\omega$},
    \\
    \int_\omega \tau \, \nabla \psi_0^\eps &= 0,
  \end{aligned}
  \right.
\end{equation}
and that at order $\delta^1$ reads
\begin{multline} \label{pb phi filtre ordre 1}
  \tau \left( \Sigma_0 \, \psi_1^\eps + \Sigma_1^\eps \, \psi_0 \right) - \operatorname{div} \Big( \tau \left[ \eps^2 \, A^\eps \nabla \psi_1^\eps - \mu^\eps_1 \right] \Big) \\ = \tau \left( \lambda_0^\eps \, \sigma_0 \, \psi_1^\eps + \lambda_0^\eps \, \sigma_1^\eps \, \psi_0^\eps + \lambda_1^\eps \, \sigma_0 \, \psi_0^\eps \right) \quad \text{in $\omega$},
\end{multline}
with the constraint $\dis \int_\omega \tau \, \nabla \psi_1^\eps = 0$. Expanding with respect to $\delta$ the normalisation condition $\| \sqrt{\tau} \, \psi^\eps \|_{L^2(\omega)} = 1$, we get
\begin{equation} \label{normalisation phi}
  \int_\omega \tau \, (\psi_0^\eps)^2 = 1, \qquad \int_\omega \tau \, \psi_0^\eps \, \psi_1^\eps = 0.
\end{equation}
Multiplying~\eqref{phi filtre ordre 0} by $\psi_0^\eps$, integrating on $\omega$ and using~\eqref{normalisation phi}, we infer
$$
\Sigma_0 + \eps^2 \int_\omega \tau \, \nabla \psi_0^\eps \cdot A^\eps \nabla \psi_0^\eps = \lambda_0^\eps \, \sigma_0.
$$
Proceeding as above, we see that the smallest possible value of $\lambda_0^\eps$ is reached for the choice
\begin{equation} \label{phi ordre 0}
  \lambda_0^\eps \, \sigma_0 = \Sigma_0, \quad \psi_0^\eps = 1, \quad \mu^\eps_0 = 0.
\end{equation}
Using~\eqref{phi ordre 0}, we deduce from~\eqref{pb phi filtre ordre 1} that
\begin{equation} \label{pb phi filtre ordre 1 bis}
  \left\{
  \begin{aligned}
    \tau \, \Sigma_1^\eps - \operatorname{div} \Big( \tau \left[ \eps^2 \, A^\eps \nabla \psi_1^\eps - \mu^\eps_1 \right] \Big) &= \tau \left( \lambda_0^\eps \, \sigma_1^\eps + \lambda_1^\eps \, \sigma_0 \right) \quad \text{in $\omega$},
    \\
    \int_\omega \tau \, \nabla \psi_1^\eps &= 0.
  \end{aligned}
  \right.
\end{equation}
Integrating this equation over $\omega$ and using the fact that $\tau$ vanishes on $\partial \omega$, we obtain
\begin{equation} \label{lambda 1 eps}
  \lambda_1^\eps \, \sigma_0 + \lambda_0^\eps \int_\omega \tau \, \sigma_1^\eps = \int_\omega \tau \, \Sigma_1^\eps.
\end{equation}
Comparing~\eqref{psi ordre 0} and~\eqref{phi ordre 0}, we see that, at the leading order in $\delta$, we have an equality between the solution $(\lambda^\infty, \psi)$ to~\eqref{pb spectral} and the solution $(\lambda^\eps, \psi^\eps)$ to~\eqref{demo pb reaction diffusion stationnaire filtre}:
\begin{equation} \label{eq:resu_ordre_0}
  \lambda_0^\infty = \lambda_0^\eps = \frac{\Sigma_0}{\sigma_0}, \qquad \psi_0 = \psi_0^\eps = 1.
\end{equation}
In the following sections, we compare the next terms in the expansion in $\delta$. 

\subsection{Convergence of the eigenvalue}

We show here that $\lambda_1^\eps$ converges to $\lambda_1^\infty$ when $\eps \to 0$. We start by recalling the following result, which is a reformulation of~\cite[Theo.~1]{cances2005long} or~\cite[Prop.~3]{blancImprovingComputationHomogenized2010}:

\begin{theorem} \label{theoreme convergence valeur propre}
  Let $\omega = (0,1)^d$. Let $\tau$ be given by~\eqref{definition filtre d dimension} with $\tau_0$ satisfying~\eqref{hypotheses filtre varphi0}. Let $g$ be a $Y$-periodic function in $L^2(Y)$, with mean $\langle g \rangle$, and let $g^\eps$ be defined by $g^\eps(x) = g(x/\eps)$. We then have
  \begin{equation*} 
    \left| \int_\omega \tau \, g^\eps - \langle g \rangle \right| \leq C \, \eps^{k+1},
  \end{equation*}
  where $C$ is independent of $\eps$ but depends on $d$, $g$ and $\tau_0$ (and thus its order $k$ defined in~\eqref{hypotheses filtre varphi0}).
\end{theorem}

In view of~\eqref{lambda 1 etoile}, \eqref{lambda 1 eps} and~\eqref{eq:resu_ordre_0}, we deduce from Theorem~\ref{theoreme convergence valeur propre} that
\begin{equation} \label{estimee convergence valeur propre}
  \left| \lambda_1^\eps - \lambda_1^\infty \right| \leq C \, \eps^{k+1}.
\end{equation}

\subsection{Convergence of the Lagrange multiplier and the eigenvector} \label{sec:cv_eigenvector}

Restricting ourselves to the one-dimensional case and assuming that $A=1$, we now show that $\psi_1^\eps$ converges to $\psi_1(\cdot/\eps)$, in an interior domain of $\omega= (0,1)$. To do this, we first estimate the Lagrange multiplier $\mu^\eps_1$.

\medskip

\noindent{\bf Step~1: Convergence of the Lagrange multiplier.} Under these assumptions, Equation~\eqref{pb phi filtre ordre 1 bis} reads
\begin{equation} \label{pb phi filtre ordre 1 1D}
  \left\{
  \begin{aligned}
    - \left( \tau \left[ \eps^2 \, (\psi_1^\eps)' - \mu^\eps_1 \right] \right)' &= \tau \left( \lambda_0^\eps \, \sigma_1^\eps + \lambda_1^\eps \, \sigma_0 - \Sigma_1^\eps \right) \quad \text{in $\omega$},
    \\
    \int_\omega \tau \, (\psi_1^\eps)' &= 0.
  \end{aligned}
  \right.
\end{equation}
Integrating the first line of~\eqref{pb phi filtre ordre 1 1D} between $0$ and $x$ and using that $\tau(0) = 0$, we obtain
\begin{equation} \label{equation phi 1}
  - \tau(x) \left[ \eps^2 \, (\psi_1^\eps)'(x) - \mu^\eps_1 \right] = \int_0^x \tau(t) \left( \lambda_0^\eps \, \sigma_1^\eps(t) + \lambda_1^\eps \, \sigma_0 - \Sigma_1^\eps(t) \right) dt.
\end{equation}
We then integrate over $\omega$ and use the second line of~\eqref{pb phi filtre ordre 1 1D} to obtain
\begin{align}
  \mu^\eps_1
  &= \int_\omega \int_0^x \tau(t) \left( \lambda_0^\eps \, \sigma_1^\eps(t) + \lambda_1^\eps \, \sigma_0 - \Sigma_1^\eps(t) \right) dt \, dx
  \nonumber
  \\
  &= \int_\omega (1-t) \, \tau(t) \left( \lambda_0^\eps \, \sigma_1^\eps(t) + \lambda_1^\eps \, \sigma_0 - \Sigma_1^\eps(t) \right) dt
  \nonumber
  \\
  &= \int_\omega (1-t) \, \tau(t) \, \sigma_0 \left( \lambda_1^\eps - \lambda_1^\infty \right) dt + \int_\omega (1-t) \, \tau(t) \, Q^\eps(t) \, dt,
  \label{eq:normandie12}
\end{align}
where $Q^\eps(t) = Q(t/\eps)$ for the $Y$-periodic function
\begin{equation} \label{eq:def_Q}
  Q(y) = \lambda_0^\infty \, \sigma_1(y) + \lambda_1^\infty \, \sigma_0 - \Sigma_1(y).
\end{equation}
Note that we have used the fact that $\lambda_0^\eps = \lambda_0^\infty$ (recall~\eqref{eq:resu_ordre_0}) to identify $Q$. In view of~\eqref{lambda 1 etoile}, we observe that $\langle Q \rangle = 0$. 

The first term in the right-hand side of~\eqref{eq:normandie12} is easy to estimate in view of~\eqref{estimee convergence valeur propre}:
\begin{equation} \label{eq:normandie13}
\left| \int_\omega (1-t) \, \tau(t) \, \sigma_0 \left( \lambda_1^\eps - \lambda_1^\infty \right) dt \right| = C \left| \lambda_1^\eps - \lambda_1^\infty \right| \leq C \, \eps^{k+1}.
\end{equation}
We now introduce $\widehat{\tau}(t) = (1-t) \, \tau(t)$. It is easy to verify that $\widehat{\tau} \in C^{k+1}(0,1)$ and that, for any $0 \leq i \leq k-1$, we have $\widehat{\tau}^{(i)}(0) = \widehat{\tau}^{(i)}(1) = 0$. We can thus use Theorem~\ref{theoreme convergence valeur propre} with the filter $\widehat{\tau}$ and the $Y$-periodic function $Q$, and we obtain, for the second term in the right-hand side of~\eqref{eq:normandie12}, that
$$
\left| \int_\omega (1-t) \, \tau(t) \, Q^\eps(t) - \langle Q \rangle \, \widehat{Z} \right| \leq C \, \eps^{k+1},
$$
with $\dis \widehat{Z} = \int_0^1 (1-t) \, \tau(t) \, dt$. Since the mean of $Q$ vanishes, we deduce from~\eqref{eq:normandie12}, \eqref{eq:normandie13} and the above result that
\begin{equation} \label{mu eps 1}
  |\mu^\eps_1| \leq C \, \eps^{k+1}.
\end{equation}

\medskip

\noindent{\bf Step~2: Convergence of the eigenvector.} We can recast~\eqref{equation phi 1} as
\begin{align}
  & (\psi_1^\eps)'(x)
  \nonumber
  \\
  &= \frac{\mu_1^\eps}{\eps^2} + \frac{1}{\eps^2 \, \tau(x)} \int_0^x \tau(t) \left( \Sigma_1^\eps(t) - \lambda_0^\eps \, \sigma_1^\eps(t) - \lambda_1^\eps \, \sigma_0 \right) dt
  \nonumber
  \\
  &= \frac{\mu_1^\eps}{\eps^2} + \frac{1}{\eps^2 \, \tau(x)} \int_0^x \tau(t) \, \sigma_0 \left( \lambda_1^\infty - \lambda_1^\eps \right) dt - \frac{1}{\eps^2 \, \tau(x)} \int_0^x \tau(t) \, Q^\eps(t) \, dt,
  \label{eq:normandie14}
\end{align}
where we use again the function $Q$ introduced in~\eqref{eq:def_Q}.

Under the assumptions made at the beginning of Section~\ref{sec:cv_eigenvector}, Equation~\eqref{psi ordre 1} reads $-\psi_1'' = Q$ in $Y$, with periodic boundary conditions. We thus deduce that
\begin{equation} \label{psi 1 prime}
  \forall y \in Y, \quad \psi_1'(y) = \int_{s=0}^1 \int_{t=0}^s Q(t) \, dt \, ds - \int_{t=0}^y Q(t) \, dt.
\end{equation}
We expand $Q$ as a Fourier series, recalling that $\langle Q \rangle = 0$. We thus have
\begin{equation} \label{Fourier decomposition sigma 1}
  \forall t \in Y, \quad Q(t) = \sum_{j \in \mathbb{Z}, j \neq 0} \widehat{Q}_j \, e^{2i \pi j t}.
\end{equation}
We then insert~\eqref{Fourier decomposition sigma 1} in~\eqref{psi 1 prime} and obtain
\begin{align}
  \psi_1'(y) &= \int_{s=0}^1 \int_{t=0}^s \sum_{j \in \mathbb{Z}, j \neq 0} \widehat{Q}_j \, e^{2i \pi j t} \, dt \, ds - \int_0^y \sum_{j \in \mathbb{Z}, j \neq 0} \widehat{Q}_j \, e^{2i \pi j t} \, dt
  \nonumber
  \\
  &= - \sum_{j \in \mathbb{Z}, j \neq 0} \frac{\widehat{Q}_j}{2i \pi j} - \sum_{j \in \mathbb{Z}, j \neq 0} \widehat{Q}_j \, \frac{e^{2i \pi j y} - 1}{2i \pi j}
  \nonumber
  \\
  &= - \sum_{j \in \mathbb{Z}, j \neq 0} \widehat{Q}_j \, \frac{e^{2i \pi j y}}{2i \pi j}.
  \label{eq:normandie15}
\end{align}
Turning to the third term of~\eqref{eq:normandie14}, we write~\eqref{Fourier decomposition sigma 1} in the form
$$
\forall t \in \omega, \quad Q^\eps(t) = \sum_{j \in \mathbb{Z}, j \neq 0} \widehat{Q}_j \, e^{\frac{2i \pi j t}{\eps}}
$$
and obtain
\begin{align*}
  & \frac{1}{\eps^2 \, \tau(x)} \int_0^x \tau(t) \, Q^\eps(t) \, dt
  \\
  &= \sum_{j \in \mathbb{Z}, j \neq 0} \frac{\widehat{Q}_j}{\eps^2 \, \tau(x)} \int_0^x \tau(t) \, e^{\frac{2i \pi j t}{\eps}} \, dt
  \\
  &= \sum_{j \in \mathbb{Z}, j \neq 0} \frac{\widehat{Q}_j}{\eps \, \tau(x)} \, \frac{1}{2i \pi j} \left(\tau(x) \, e^{\frac{2i \pi j x}{\eps}} - \int_0^x \tau'(t) \, e^{\frac{2i \pi j t}{\eps}} \, dt \right)
  \\
  &= \sum_{j \in \mathbb{Z}, j \neq 0} \frac{\widehat{Q}_j}{2i \pi j \eps} \, e^{\frac{2i \pi j x}{\eps}} - \sum_{j \in \mathbb{Z}, j \neq 0} \frac{\widehat{Q}_j}{\eps \, \tau(x)} \, \frac{1}{2i \pi j} \int_0^x \tau'(t) \, e^{\frac{2i \pi j t}{\eps}} \, dt
  \\
  &= - \frac{1}{\eps} \, \psi_1'\left(\frac{x}{\eps}\right) - \sum_{j \in \mathbb{Z}, j \neq 0} \frac{\widehat{Q}_j}{\eps \, \tau(x)} \, \frac{1}{2i \pi j} \int_0^x \tau'(t) \, e^{\frac{2i \pi j t}{\eps}} \, dt,
\end{align*}
where we have used~\eqref{eq:normandie15} at the last line. We thus deduce from~\eqref{eq:normandie14} that
\begin{multline*}
  (\psi_1^\eps)'(x) - \frac{1}{\eps} \, \psi_1'\left(\frac{x}{\eps}\right) = \frac{\mu_1^\eps}{\eps^2} + \frac{1}{\eps^2 \, \tau(x)} \int_0^x \tau(t) \, \sigma_0 \left( \lambda_1^\infty - \lambda_1^\eps \right) dt \\ + \sum_{j \in \mathbb{Z}, j \neq 0} \frac{\widehat{Q}_j}{\eps \, \tau(x)} \, \frac{1}{2i \pi j} \int_0^x \tau'(t) \, e^{\frac{2i \pi j t}{\eps}} \, dt,
\end{multline*}
and thus, using~\eqref{estimee convergence valeur propre}, \eqref{mu eps 1} and an integration by part for the last term,
\begin{multline*}
  \left| (\psi_1^\eps)'(x) - \frac{1}{\eps} \, \psi_1'\left(\frac{x}{\eps}\right) \right| \leq C \, \eps^{k-1} \left( 1 + \frac{1}{\tau(x)} \right) \\ + \left| \sum_{j \in \mathbb{Z}, j \neq 0} \frac{\widehat{Q}_j}{\tau(x)} \, \frac{1}{(2i \pi j)^2} \left( \tau'(x) \, e^{\frac{2i \pi j x}{\eps}} - \tau'(0) - \int_0^x \tau''(t) \, e^{\frac{2i \pi j t}{\eps}} \, dt \right) \right|,
\end{multline*}
and therefore
\begin{equation} \label{eq:normandie16}
\left| (\psi_1^\eps)'(x) - \frac{1}{\eps} \, \psi_1'\left(\frac{x}{\eps}\right) \right| \leq C \, \eps^{k-1} \left( 1 + \frac{1}{\tau(x)} \right) + \frac{C}{\tau(x)} \sum_{j \in \mathbb{Z}, j \neq 0} \frac{\left|\widehat{Q}_j\right|}{j^2}.
\end{equation}
The right-most series converges, since $\sigma$ and $\Sigma$ both belong to $L^\infty(Y)$, and thus $Q \in L^2(Y)$, which implies that $\sum_{j \in \mathbb{Z}} \left| \widehat{Q}_j \right|^2$ converges.

We next estimate the $L^2$ norm of the above difference in an interior domain of $\omega$. This is indeed a relevant norm for us: recall indeed that, in practice, we compute our proxy by solving Problem~\eqref{pb patch oversampling filtre tmp} on the oversampled element $S_K$, and next only use the restriction of the proxy on the element $K$ (see e.g.~\eqref{chi phi}). We thus introduce some $0 < \alpha < 1/2$. Using that $\| 1/\tau \|_{L^2(\alpha,(1-\alpha))}$ is finite, we have
$$
\left\| (\psi_1^\eps)' - [\psi_1(\cdot/\eps)]' \right\|_{L^2\left(\alpha,(1-\alpha)\right)} \leq C \left( \eps^{k-1} + 1 \right).
$$
In order to estimate a relative difference, we now argue that
$$
\left\| [\psi_1(\cdot/\eps)]' \right\|^2_{L^2\left(\alpha,(1-\alpha)\right)} \geq \frac{1}{\eps} \left( \frac{(1-2\alpha)}{\eps} -2 \right) \int_0^1 | \psi_1'(y) |^2 \, dy \geq \frac{C}{\eps^2},
$$
where the right-most estimate stems from the assumption that we have assumed $\psi$ (and therefore $\psi_1$) to not be constant. We thus eventually obtain
\begin{equation} \label{estimee convergence vecteur propre}
\frac{\left\| (\psi_1^\eps)' - [\psi_1(\cdot/\eps)]' \right\|_{L^2(\alpha,(1-\alpha))}}{\left\| [\psi_1(\cdot/\eps)]' \right\|_{L^2(\alpha,(1-\alpha))}} \leq C \left( \eps^k + \eps \right) \leq C \, \eps,
\end{equation}
since the order of the filter satisfies $k \geq 1$.

It is interesting to note that the above analysis points to the fact that, in terms of rate of convergence for the eigenvector (which is the relevant quantity in our approach), it is not useful to increase the order of the filter: the convergence (in the $H^1$ semi-norm) always holds at the rate $O(\eps)$, whatever $k \geq 1$.

\begin{rmq}
  Looking at~\eqref{eq:normandie16}, we see terms of the form $1/\tau$, where $\tau$ is a function that vanishes at the boundaries of $\omega$. It is thus delicate to estimate the relative error $\dis \frac{\left\| (\psi_1^\eps)' - [\psi_1(\cdot/\eps)]' \right\|_{L^2(\omega)}}{\left\| [\psi_1(\cdot/\eps)]' \right\|_{L^2(\omega)}}$ over the entire domain $\omega$. This is consistent with the numerical observation that $(\psi_1^\eps)' - [\psi_1(\cdot/\eps)]'$ may be large at the boundary of $\omega$, thus the need of an oversampling strategy.
\end{rmq}

\begin{rmq}
  The case when we do not use any filtering function formally corresponds to the choice $k=0$. In that case, the estimate~\eqref{estimee convergence vecteur propre} does not allow to conclude to the convergence of $(\psi_1^\eps)'$ to $[\psi_1(\cdot/\eps)]'$ (in the sense of a converging relative error), an observation which is consistent with the numerical results discussed in Section~\ref{sec:num_filtre} below.
\end{rmq}

\subsection{Numerical comparisons} \label{sec:num_filtre}

We now present some numerical results in dimension 1, on the domain $\omega = (0,1)$. We compare the rates of convergence predicted by the estimates~\eqref{estimee convergence valeur propre} (on the eigenvalue) and~\eqref{estimee convergence vecteur propre} (on the eigenvectors), obtained through arguments in the framework of small perturbations for $\Sigma$ and $\sigma$, with actual rates of convergence observed on numerical results. The comparison is performed in a periodic but non perturbative regime: the functions $\Sigma$ and $\sigma$ are periodic but their oscillations are not necessarily small. 

We are going to compare the first eigencouple $(\lambda^\infty, \psi)$ of the cell problem~\eqref{pb spectral} (which is our reference quantity) with the first eigencouple $(\lambda^\eps, \psi^\eps)$ of the problem~\eqref{demo filtre : pb patch oversampling filtre} (problem posed on an oversampling domain and using a filtering function). We also consider, in our test, the first eigencouple $(\lambda^{\eps,\#}, \psi^{\eps,\#})$ of the following problem: find the eigenvector $\psi^{\eps,\#} \in H^1_{\rm per}(\omega)$ (i.e. satisfying periodic boundary conditions) associated to the smallest eigenvalue $\lambda^{\eps,\#}$ of the eigenvalue problem
\begin{equation*} 
  \forall v \in H^1_{\rm per}(\omega), \quad \int_\omega \Sigma^\eps \, \psi^{\eps,\#} \, v + \eps^2 \int_\omega (\nabla v)^T A^\eps \nabla \psi^{\eps,\#} = \lambda^{\eps,\#} \int_\omega \sigma^\eps \, \psi^{\eps,\#} \, v.
\end{equation*}
This problem is a formulation on $\omega$ of Problem~\eqref{pb patch oversampling}, posed on an oversampling domain and without any filtering function.

For the eigenvalue, we consider the relative errors
\begin{equation} \label{eq_ober:relative_error_lambda_eps filtre}
  \frac{| \lambda^\eps - \lambda^\infty |}{| \lambda^\infty |}
\end{equation}
and
\begin{equation} \label{eq_ober:relative_error_lambda_eps periodique}
  \frac{| \lambda^{\eps,\#} - \lambda^\infty |}{| \lambda^\infty |}.
\end{equation}
For the eigenvector, we consider the relative error on a domain $\omega_{\rm int}$ interior to $\omega$. Two norms are relevant: the $L^\infty$ norm, for which we use the relative errors
\begin{equation} \label{eq_ober:relative_error_phi_eps filtre_Linfty}
  \frac{\left\| \psi^\eps - \psi(\cdot/\eps) \right\|_{L^\infty(\omega_{\rm int})}}{\left\| \psi(\cdot/\eps) \right\|_{L^\infty(\omega_{\rm int})}}
\end{equation}
and
\begin{equation} \label{eq_ober:relative_error_phi_eps periodique_Linfty}
  \frac{\left\| \psi^{\eps,\#} - \psi(\cdot/\eps) \right\|_{L^\infty(\omega_{\rm int})}}{\left\| \psi(\cdot/\eps) \right\|_{L^\infty(\omega_{\rm int})}},
\end{equation}
and the $H^1$ semi-norm, for which we use the relative errors
\begin{equation} \label{eq_ober:relative_error_phi_eps filtre}
  \frac{\left\| (\psi^\eps)' - [\psi(\cdot/\eps)]' \right\|_{L^2(\omega_{\rm int})}}{\left\| [\psi(\cdot/\eps)]' \right\|_{L^2(\omega_{\rm int})}}
\end{equation}
and
\begin{equation} \label{eq_ober:relative_error_phi_eps periodique}
  \frac{\left\| (\psi^{\eps,\#})' - [\psi(\cdot/\eps)]' \right\|_{L^2(\omega_{\rm int})}}{\left\| [\psi(\cdot/\eps)]' \right\|_{L^2(\omega_{\rm int})}}.
\end{equation}
In practice, we choose $\omega_{\rm int} = (1/3,2/3)$. In the numerical tests, we refer to the errors~\eqref{eq_ober:relative_error_lambda_eps filtre}, \eqref{eq_ober:relative_error_phi_eps filtre_Linfty} and~\eqref{eq_ober:relative_error_phi_eps filtre} as \textit{filtered} and the errors~\eqref{eq_ober:relative_error_lambda_eps periodique}, \eqref{eq_ober:relative_error_phi_eps periodique_Linfty} and~\eqref{eq_ober:relative_error_phi_eps periodique} as \textit{periodic}.

We choose the diffusion and reaction coefficients as follows:
$$
\forall y \in Y, \quad A(y) = 5 \, \sin^2(\pi y) + 1, \quad \Sigma(y) = 10 \, \cos^2(\pi y) + 1, \quad \sigma(y) = 1.
$$

\subsubsection{Second order filter}

We set
\begin{equation} \label{ober_definition filtre ordre 2 numerique}
  \tau_0(x) = 30 \, x^2 \, (1-x)^2.
\end{equation}
According to~\eqref{hypotheses filtre varphi0}, this corresponds to a filter of order $k=2$. The results we get are shown on Figures~\ref{fig_ober:filter2_VP} and~\ref{fig_ober:filter2_H1}. In view of~\eqref{estimee convergence valeur propre} and~\eqref{estimee convergence vecteur propre}, we expect a convergence of the eigenvalue at the rate $\eps^3$ and a relative error (in the $H^1$ semi-norm) on the eigenvector of the order of $\eps$. This is indeed what we numerically observe on the red curves. When not using a filter, we expect (from~\eqref{estimee convergence valeur propre} and~\eqref{estimee convergence vecteur propre} with $k=0$) a convergence of the eigenvalue at the rate $\eps$ and a relative error on the eigenvector which does not converge to 0 when $\eps \to 0$. This is indeed what we numerically observe on the blue curves (with actually the relative error~\eqref{eq_ober:relative_error_phi_eps periodique} increasing as $\eps$ goes to 0). 

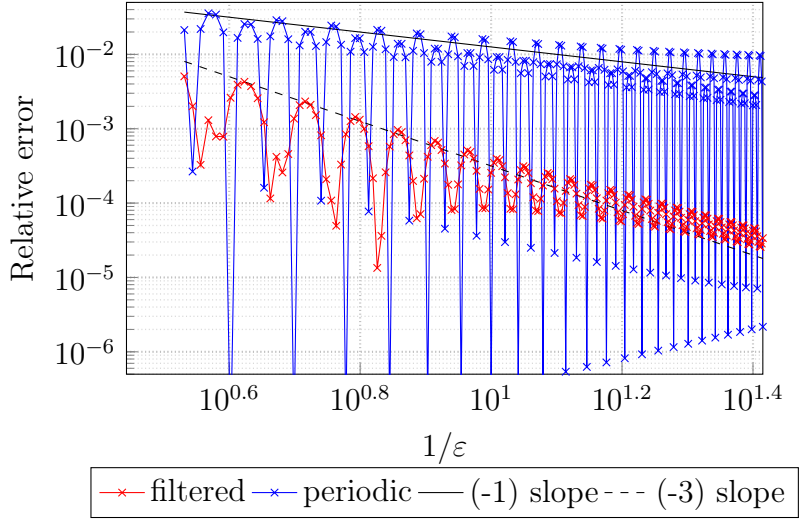
\begin{figure}[H]
  \centering
  \begin{tikzpicture}
    \begin{loglogaxis}[
        xlabel={$1/\eps$},
        ylabel={Relative error},
        grid=both, 
        major grid style={line width=0.8pt, draw=gray!50},
        minor grid style={line width=0.5pt, draw=gray!30},
        legend style={at={(0.5,-0.25)}, anchor=north, legend columns=-1},
        width=10cm, height=6.5cm,
        ymin=0.0000005, ymax=0.05,
        xmax=26,
        cycle list name=color list
      ]
      \addplot[mark=x, color=red] table[x index=0, y index=1, col sep=comma] {Fichiers_Erreurs_15juin26/filtreXcarre/Valeur_propre/Liste_Erreur_VP_filtre.txt};
      \addlegendentry{filtered}
      
      \addplot[mark=x, color=blue] table[x index=0, y index=1, col sep=comma] {Fichiers_Erreurs_15juin26/filtreXcarre/Valeur_propre/Liste_Erreur_VP_periodique.txt};
      \addlegendentry{periodic}
      
      \addplot[solid] table[x index=0, y index=1, col sep=comma] {Fichiers_Erreurs_15juin26/filtreXcarre/Valeur_propre/pente_-1.txt};
      \addlegendentry{(-1) slope}
      
      \addplot[dashed] table[x index=0, y index=1, col sep=comma] {Fichiers_Erreurs_15juin26/filtreXcarre/Valeur_propre/pente_-3.txt};
      \addlegendentry{(-3) slope}
    \end{loglogaxis}
  \end{tikzpicture}
  \centering
  \caption{Second order filter~\eqref{ober_definition filtre ordre 2 numerique}: comparison of the relative errors~\eqref{eq_ober:relative_error_lambda_eps filtre} (in red) and~\eqref{eq_ober:relative_error_lambda_eps periodique} (in blue) on the eigenvalue as a function of $1/\eps$.}
  \label{fig_ober:filter2_VP}
\end{figure}

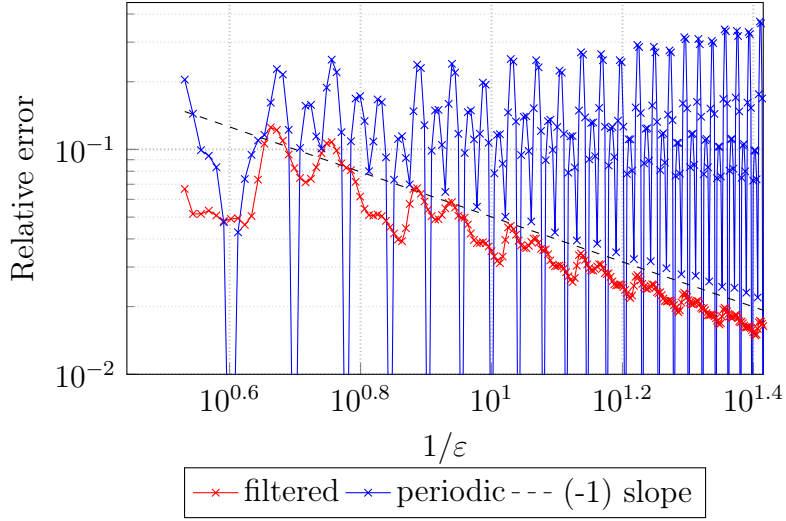
\begin{figure}[H]
  \centering
  \begin{tikzpicture}
    \begin{loglogaxis}[
        xlabel={$1/\eps$},
        ylabel={Relative error},
        grid=both, 
        major grid style={line width=0.8pt, draw=gray!50},
        minor grid style={line width=0.5pt, draw=gray!30},
        legend style={at={(0.5,-0.25)}, anchor=north, legend columns=-1},
        width=10cm, height=6.5cm,
        ymin=1e-2, ymax=0.45,
        xmax=26,
        cycle list name=color list
      ]
      \addplot[mark=x, color=red] table[x index=0, y index=1, col sep=comma] {Fichiers_Erreurs_15juin26/filtreXcarre/H1_interieur/Liste_Erreur_H1_filtre.txt};
      \addlegendentry{filtered}
      
      \addplot[mark=x, color=blue] table[x index=0, y index=1, col sep=comma] {Fichiers_Erreurs_15juin26/filtreXcarre/H1_interieur/Liste_Erreur_H1_periodique.txt};
      \addlegendentry{periodic}
      
      \addplot[dashed] table[x index=0, y index=1, col sep=comma] {Fichiers_Erreurs_15juin26/filtreXcarre/H1_interieur/pente_-1.txt};
      \addlegendentry{(-1) slope}
    \end{loglogaxis}
  \end{tikzpicture}
  \centering
  \caption{Second order filter~\eqref{ober_definition filtre ordre 2 numerique}: comparison of the relative errors~\eqref{eq_ober:relative_error_phi_eps filtre} (in red) and~\eqref{eq_ober:relative_error_phi_eps periodique} (in blue) on the eigenvector as a function of $1/\eps$.}
  \label{fig_ober:filter2_H1}
\end{figure}

\medskip

We have also monitored the error in $L^\infty$ norm and obtained the same type of results as for the $H^1$ semi-norm (results not shown): the error~\eqref{eq_ober:relative_error_phi_eps filtre_Linfty} decreases at the rate $\eps$, while the error~\eqref{eq_ober:relative_error_phi_eps periodique_Linfty} does not converge (and rather increases at the rate $1/\eps$).

\subsubsection{First order filter}

We now set 
\begin{equation} \label{ober_definition filtre ordre 1 numerique}
  \tau_0(x) = 6 \, x \, (1-x).
\end{equation}
According to~\eqref{hypotheses filtre varphi0}, this corresponds to a filter of order $k=1$. In view of~\eqref{estimee convergence valeur propre} and~\eqref{estimee convergence vecteur propre}, we expect a convergence of the eigenvalue at the rate $\eps^2$ and a relative error (in the $H^1$ semi-norm) on the eigenvector of the order of $\eps$. This is indeed what we obtain (results not shown). The results concerning the approximation of the eigenvector in the $L^\infty$ norm are shown on Figure~\ref{fig_ober:filter1_Linfty}. The error~\eqref{eq_ober:relative_error_phi_eps filtre_Linfty} remains small (of the order of 1\%) for all the values of $\eps$ we considered, but it does not decrease when $\eps \to 0$. This demonstrates the usefulness of using a filter of order at least $k \geq 2$.

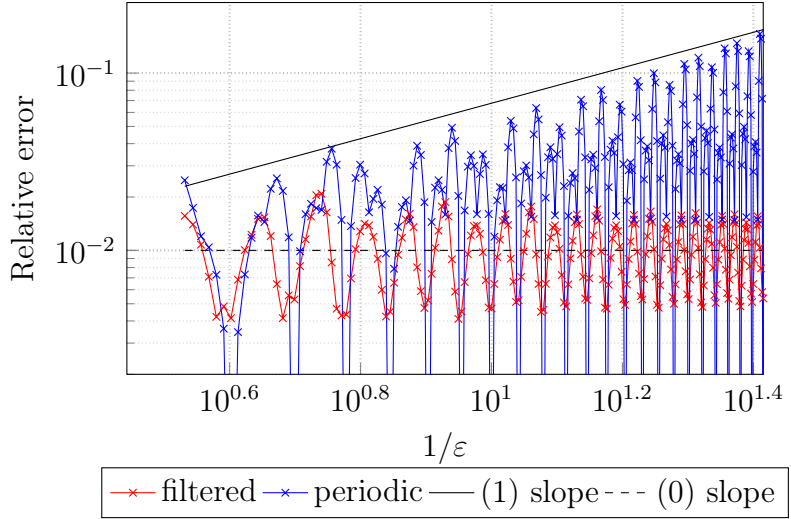
\begin{figure}
  \centering
  \begin{tikzpicture}
    \begin{loglogaxis}[
        xlabel={$1/\eps$},
        ylabel={Relative error},
        grid=both, 
        major grid style={line width=0.8pt, draw=gray!50},
        minor grid style={line width=0.5pt, draw=gray!30},
        legend style={at={(0.5,-0.25)}, anchor=north, legend columns=-1}, 
        width=10cm, height=6.5cm, 
        ymin=0.002, ymax=0.25,
        xmax=26,
        cycle list name=color list
      ]
      \addplot[mark=x, color=red] table[x index=0, y index=1, col sep=space] {Fichiers_Erreurs_15juin26/filtreX/Linfty/Liste_Erreur_Linfty_filtreX.txt};
      \addlegendentry{filtered}
      
      \addplot[mark=x, color=blue] table[x index=0, y index=1, col sep=space] {Fichiers_Erreurs_15juin26/filtreX/Linfty/Liste_Erreur_Linfty_periodique_filtreX.txt};
      \addlegendentry{periodic}
      
      \addplot[solid] table[x index=0, y index=1, col sep=space] {Fichiers_Erreurs_15juin26/filtreX/Linfty/pente_1_filtreX.txt};
      \addlegendentry{(1) slope}
      
      \addplot[dashed] table[x index=0, y index=1, col sep=space] {Fichiers_Erreurs_15juin26/filtreX/Linfty/pente_0_filtreX.txt};
      \addlegendentry{(0) slope}
    \end{loglogaxis}
  \end{tikzpicture}
  \centering
  \caption{First order filter~\eqref{ober_definition filtre ordre 1 numerique}: comparison of the relative errors~\eqref{eq_ober:relative_error_phi_eps filtre_Linfty} (in red) and~\eqref{eq_ober:relative_error_phi_eps periodique_Linfty} (in blue) on the eigenvector as a function of $1/\eps$.}
  \label{fig_ober:filter1_Linfty}
\end{figure}

\subsection{Conclusions}

In dimension 1, we thus observe numerically, for periodic non-perturbative examples, the same rate of convergence (for the error on the eigenvalue as well as for the error in $H^1$ semi-norm on the eigenvector) as those theoretically predicted for periodic, perturbative cases. Furthermore, our tests show that considering a filter of order $k=1$ may lead to inaccurate results. A better choice is to consider a filter of order $k=2$.

Our numerical tests in dimension 2, again for periodic non-perturbative cases, and again using a second-order filter, show the same rates of convergence (for the errors~\eqref{eq_ober:relative_error_lambda_eps filtre}, \eqref{eq_ober:relative_error_phi_eps filtre_Linfty} and~\eqref{eq_ober:relative_error_phi_eps filtre}) as in dimension 1 (results not shown).

\begin{rmq} \label{rem:vector_valued}
  For the vector-valued variant of the problem, we can numerically investigate the error between the first eigencouple of~\eqref{pb patch oversampling filtre 2 groupes_1}--\eqref{pb patch oversampling filtre 2 groupes_2} and the first eigencouple of~\eqref{pb spectral 2 groupes} (and likewise for the adjoint problems). We observe the same convergence rates as for the scalar-valued problem.
\end{rmq}

\section*{Acknowledgments}

We would like to thank Grégoire Allaire for stimulating and enlightening discussions about this work. The authors are grateful to ONR and EOARD for their continuous support, in particular through the current grants N00014-25-1-2299 and FA8655-24-1-7057.

\bibliographystyle{plain}


\bibliography{biblio_lll.bib}

\end{document}